\RequirePackage{fix-cm}
\documentclass{amsart} 
\usepackage{mathptmx}
\usepackage{latexsym}
\usepackage{amsmath}
\usepackage{amssymb}
\usepackage{amsfonts} 
\usepackage{eucal} 
\usepackage{amsbsy}
\usepackage{amsthm}
\usepackage{graphicx}
\usepackage[all]{xy}
\xyoption{knot}

\newtheorem{thm}{Theorem}[section]
\newtheorem*{thm*}{Theorem}
\newtheorem{lemma}[thm]{Lemma}
\newtheorem{prop}[thm]{Proposition}
\newtheorem{cor}[thm]{Corollary}
\newtheorem*{cor*}{Corollary}

\theoremstyle{definition}
\newtheorem{defn}[thm]{Definition}
\newtheorem{example}[thm]{Example}
\newtheorem{examples}[thm]{Examples}

\theoremstyle{remark}
\newtheorem{remark}[thm]{Remark}

\newcommand {\Ea}    {\ensuremath{\mbox{$\mathcal{E}$}}}
\newcommand {\Fa}    {\ensuremath{\mbox{$\mathcal{F}$}}}

\newcommand {\Na}    {\ensuremath{\mbox{$\mathcal{N}$}}}
\newcommand {\Oa}    {\ensuremath{\mbox{$\mathcal{O}$}}}

\newcommand {\real}  {\ensuremath{\mathbb{R}}}

\newcommand {\intg}  {\ensuremath{\mathbb{Z}}}

\newcommand {\cplx}  {\ensuremath{\mathbb{C}}}

\newcommand {\Hom}   {\ensuremath{\operatorname{Hom}}}

\newcommand {\smlhf} {\ensuremath{\mbox{$\frac{1}{2}$}}}

\newcommand {\Homeo} {\ensuremath{\operatorname{Homeo}}}

\newcommand {\pt}    {\ensuremath{\operatorname{pt}}}
\newcommand {\id}    {\ensuremath{\operatorname{id}}}

\newcommand {\rel}   {\ensuremath{{\operatorname{rel}}}}

\newcommand {\hotimes}   {\ensuremath{\widehat{\otimes}}}

\newcommand {\catc}   {\ensuremath{\mathbf{C}}}

\newcommand {\catf}   {\ensuremath{\mathbf{F}}}
\newcommand {\catm}   {\ensuremath{\mathbf{M}}}

\newcommand {\catz}   {\ensuremath{\mathbf{Z}}}
\newcommand {\vect}   {\ensuremath{\mathbf{Vect}}}

\newcommand {\nat}   {\ensuremath{\mathbb{N}}}
\newcommand {\bool}   {\ensuremath{\mathbb{B}}}

\newcommand {\dom} {\ensuremath{\operatorname{dom}}}
\newcommand {\cod} {\ensuremath{\operatorname{cod}}}

\newcommand {\Ob} {\ensuremath{\operatorname{Ob}}}
\newcommand {\Mor} {\ensuremath{\operatorname{Mor}}}

\newcommand {\fun} {\ensuremath{\operatorname{Fun}}}
\newcommand {\bfun} {\ensuremath{\mathbf{Fun}}}
\newcommand {\TT}   {\ensuremath{\mathbb{T}}}
\newcommand {\LL}   {\ensuremath{\mathbb{L}}}
\newcommand {\iv} {\ensuremath{\operatorname{iv}}}
\newcommand {\sign} {\ensuremath{\operatorname{sign}}}
\newcommand {\mti}   {\ensuremath{M\times [0,1]}}
\newcommand {\nti}   {\ensuremath{N\times [0,1]}}
\newcommand {\gA}   {\ensuremath{\mathfrak{A}}}
\newcommand {\Cob} {\ensuremath{\operatorname{Cob}}}
\newcommand {\cut} {{\ensuremath{\operatorname{cut}}}}
\newcommand {\ovf}  {\ensuremath{\overline{\mathcal{F}}}}

\begin{document}

\UseComputerModernTips


\title{Positive Topological Quantum Field Theories}

\author{Markus Banagl}

\address{Mathematisches Institut, Universit\"at Heidelberg,
  Im Neuenheimer Feld 288, 69120 Heidelberg, Germany}

\email{banagl@mathi.uni-heidelberg.de}

\thanks{The author was in part supported by a research grant of the
 Deutsche Forschungsgemeinschaft.}

\date{March, 2013}




\begin{abstract}
We propose a new notion of positivity for topological field theories (TFTs), based on
S. Eilenberg's concept of completeness for semirings. We show that a complete ground semiring,
a system of fields on manifolds and a system of action functionals on these fields
determine a positive TFT. The main feature of such a theory is a 
semiring-valued topologically invariant state sum that satisfies a gluing formula. 
The abstract framework has been carefully designed to cover a wide range of
phenomena. For instance, we derive P\'olya's counting theory in combinatorics from state sum
identities in a suitable positive TFT. Several other concrete examples are discussed,
among them Novikov signatures of fiber bundles over spacetimes and arithmetic
functions in number theory. In the future, we will employ the
framework presented here in constructing a new differential topological invariant that detects
exotic smooth structures on spheres. 
\end{abstract}

\maketitle


\tableofcontents


\section{Introduction}

We propose a mathematically rigorous new method for constructing topological field theories
(TFT), which allows for action functionals that take values in any monoidal category, not just 
real (or complex) values. The second novelty of our approach is that state sums
(path integrals) are based on an algebraic notion of completeness in semirings,
introduced by S. Eilenberg in his work \cite{eilenbergalm} on formal languages and automata,
based on ideas of J. H. Conway \cite{conwayregalg}.
This enables us to bypass the usual measure theoretic problems 
associated with traditional path integrals over spaces of fields.
On the other hand, the oscillatory nature of the classical path integral,
and consequently, associated asymptotic expansions via the method of
stationary phase, are not
necessarily retained in our approach. Instead, a principle of topological permanence
tends to hold: Once the action functional detects a topological feature, this feature will
remain present in the state sum invariant.
Roughly, a semiring is a ring that has only
``positive'' elements, that is, has no negatives
(and is thus sometimes also called a ``rig''). A key insight of Eilenberg was that the
absence of negatives allows for the notion of completeness, reviewed in 
Section \ref{sec.monssemirings} of the present paper.
As our state sums will
have values in semirings, we call the resulting theories \emph{positive} TFTs.
In this context one may also recall that every additively idempotent semiring
has a canonical partial order, and in this order every element is nonnegative.
In their work \cite{baezdolanhdimalg}, J. C. Baez and J. Dolan remark on page 6100:
``In Physics, linear algebra is usually done over $\real$ or $\cplx$, but for
higher-dimensional linear algebra, it is useful to start more generally with
any commutative rig'' and
``one reason we insist on such generality is to begin grappling with the remarkable fact
that many of the important vector spaces in physics are really defined over the
natural numbers, meaning that they contain a canonical lattice with a basis of
`positive' elements. Examples include the weight spaces of semisimple
Lie algebras, fusion algebras in conformal field theory, and thanks to the work
of Kashiwara and Lusztig on canonical bases, the finite-dimensional
representations of quantum groups.''
Thus this emerging theme of positivity and semirings in physics is also
reflected in the present paper. \\

A positive TFT possesses the following features, which may be taken
as axioms. These axioms are close to Atiyah's axioms \cite{atiyahtqft},
but differ in some respects, as we will discuss shortly.
As far as the algebraic environment is concerned,
instead of working with vector spaces over a field such as the real or
complex numbers, a positive TFT is defined over a pair $(Q^c, Q^m)$
of semirings, which have the same underlying additive monoid $Q$.
Furthermore, a functional tensor product $E \hotimes F$
of two function semimodules $E,F$ is used, which in the idempotent regime
is due to G. L. Litvinov, V. P. Maslov and G. B. Shpiz \cite{litmasshpiztensor}.
It comes with a canonical map
$E\otimes F \to E\hotimes F$ (where $\otimes$ is the algebraic tensor product), 
which is generally neither surjective nor even
injective (contrary to the analogous map over fields such as the complex
numbers). It can be conceptualized as a completion of the image of
the map which sends an algebraic tensor product $f\otimes g$ to the
function $(x,y)\mapsto f(x)g(y)$. This completion is necessary because
Example \ref{exple.statesumdoesnotdecomp} shows that state sums need not lie in the image of the
algebraic tensor product. The two multiplications on $Q$ induce
two generally different tensor products $f \hotimes_c g$ and $f\hotimes_m g$ of two
$Q$-valued functions $f,g$. 

An $(n+1)$-dimensional \emph{positive TFT} $Z$ assigns to every closed topological $n$-manifold
$M$ a semialgebra $Z(M)$ over both $Q^c$ and $Q^m$, called the \emph{state module}
of $M$, and to every $(n+1)$-dimensional topological bordism $W$ with boundary $\partial W$
an element
\[ Z_W \in Z(\partial W), \]
called the \emph{state sum} of $W$. We adopt the viewpoint that the 
latter is the primary invariant, while the
state module itself is of lesser importance. This assignment satisfies the following properties:

\begin{enumerate}
\item If $M$ and $N$ are closed $n$-manifolds and $M\sqcup N$ their (ordered) disjoint
union, then there is an isomorphism
\[ Z(M\sqcup N) \cong Z(M) \hotimes Z(N) \]
of $Q^c$-semialgebras and of $Q^m$-semialgebras.

\item The state module $Z(-)$ is a covariant
functor on the category of closed topological $n$-manifolds and homeomorphisms. 
In particular, the group $\operatorname{Homeo}(M)$ of self-homeomorphisms
$M\to M$ acts on $Z(M)$.

\item \emph{Isotopy Invariance}: Isotopic homeomorphisms $\phi, \psi: M\to N$ induce
 equal isomorphisms $\phi_\ast = \psi_\ast: Z(M)\to Z(N)$. In particular, the action
 of $\Homeo (M)$ on $Z(M)$ factors through the mapping class group.

\item The state sum $Z_W$ is a topological invariant: 
 If $\phi: W\to W'$ is a homeomorphism and $\phi_{\partial}$ its restriction
 to the boundary, then $\phi_{\partial \ast} (Z_W) = Z_{W'}$, where
 $\phi_{\partial \ast}: Z(\partial W)\to Z(\partial W')$ denotes the
 isomorphism induced by $\phi_\partial$.

\item The state sum $Z_{W\sqcup W'} \in Z(\partial W \sqcup \partial W')$ of a
disjoint union of bordisms $W$ and $W'$ is the tensor
product
\[ Z_{W\sqcup W'} = Z_W \hotimes_m Z_{W'} \in
  Z(\partial W)\hotimes Z(\partial W') \cong Z(\partial W \sqcup \partial W'). \]

\item \emph{Gluing}: For $n$-manifolds $M,N,P$, and
   vectors $z\in Z(M)\hotimes Z(N),$ $z' \in Z(N)\hotimes Z(P),$ 
  there is a contraction product
\[ \langle z,z' \rangle \in Z(M)\hotimes Z(P), \]
 which involves the product $\hotimes_c$.
Let $W'$ be a bordism from $M$ to $N$ and let $W''$ be a bordism from
$N$ to $P$. Let $W = W' \cup_N W''$ be the bordism from $M$ to $P$ obtained
by gluing $W'$ and $W''$ along $N$. Then the state sum of $W$ can be
calculated as the contraction product
\[ Z_W = \langle Z_{W'}, Z_{W''} \rangle \in Z(M)\hotimes Z(P)\cong Z(M\sqcup P). \]

\item The state modules $Z(M)$ have the additional structure of a noncommutative
 Frobenius semialgebra over $Q^m$ and over $Q^c$. This means that there is a
 counit functional $\epsilon_M: Z(M)\to Q$, which is a $Q$-bisemimodule
 homomorphism such that the bilinear form
\[ Z(M)\times Z(M)\to Q,~ (z,z')\mapsto \epsilon_M (zz') \]
 is nondegenerate. Maps induced on state modules by homeomorphisms respect
 the Frobenius counit.

\item For every closed $n$-manifold, there is an element $\gA (M)\in Z(M)$,
the \emph{coboundary aggregate}, which is topologically invariant:
If $\phi: M\to N$ is a homeomorphism, then $\phi_\ast \gA (M) = \gA (N)$.
\end{enumerate}

The coboundary aggregate of an $n$-manifold has no counterpart in 
classical TFTs and is a somewhat surprising new kind of topological invariant.
In a classical \emph{unitary} theory, the $Z(M)$ have natural nondegenerate 
Hermitian structures. What may substitute to some extent for this
in positive TFTs is the Frobenius counit, constructed in Section \ref{sec.frobenius}. 
The counit is available in every dimension $n$ and has nothing to do with
pairs of pants.
Atiyah's axioms imply that
state modules are always finite dimensional, but on p. 181 of
\cite{atiyahtqft}, he indicates that allowing state modules to be 
infinite dimensional may be necessary in interesting examples of TFTs.
The state modules of a positive TFT are indeed usually infinitely generated, though
finitely generated modules can result from using very small systems of fields
on manifolds. Such systems do in fact arise in practice, for example in TFTs with
a view towards finite group theory, such as the Freed-Quinn theory
\cite{quinnfreed}, 
\cite{quinnlectaxiomtqft}, \cite{freedlecnotes}.
Frobenius algebras, usually assumed to be finite dimensional over a field,
have been generalized by Jans \cite{jans} to infinite dimensions.
Atiyah's classical axioms also demand that 
the state sum $Z_{M\times I}$ of a cylinder $W=M\times I$ be the identity
when viewed as an endomorphism.
The map $Z_{M\times I}$ on $Z(M)$ should be the
``imaginary time''  evolution operator $e^{-tH}$ (where $t$ is the length
of the interval $I$), so Atiyah's axiom means that the Hamiltonian
$H=0$ and there is no dynamics along a cylinder.
We do not require this for positive TFTs and will allow interesting
propagation along the cylinder. In particular, we do not phrase positive
TFTs as monoidal functors on bordism categories, as this would imply that
the cylinder, which is an identity morphism in the bordism category, would
have to be mapped to an identity morphism. For various future applications, this is
not desirable. Since $Z_{M\times I}$ need not be the identity, it can also
generally not be deduced in a positive TFT that $Z_{M\times S^1} =
\dim Z(M)$, an identity that would not make sense in the first place,
as $Z(M)$ need not have finite dimension.
The present paper does not make substantial use of the bordism category. 
Nor do we (yet) consider $n$-categories or
manifolds with corners in this paper. Our manifolds will usually
be topological, but we will indicate the modifications necessary to
deal with smooth manifolds.\\

The main result of the present paper is that any system of fields $\Fa$ on
manifolds together with a system of action functionals $\TT$ on these fields gives
rise in a natural way to a positive topological field theory $Z$.
Our framework is not limited to particular dimensions and will produce a TFT
in any dimension.
Systems of fields are axiomatized in Definition \ref{def.fields}. 
As in \cite{kirktqft} and \cite{freedlecnotes},
they are to satisfy
the usual properties with respect to restrictions and action of homeomorphisms.
The key properties are that they must decompose on disjoint unions as
a cartesian product with factors associated to the components and it must
be possible to glue two fields along a common boundary component on which
they agree. Every field on the glued space must be of this form.
Rather than axiomatizing actions on fields (which would lead to
additive axioms), we prefer to axiomatize the exponential of an action
directly, since it is the exponential that enters into the Feynman path integral.
Also, this yields multiplicative axioms, which is closer to the nature of
TFTs, as TFTs are multiplicative, rather than additive (the former corresponding
to the quantum nature of TFTs and the latter
corresponding to homology theories).
Let $\catc$ be a (strict, small) monoidal category. We axiomatize systems $\TT$ of 
$\catc$-valued action exponentials in Definition \ref{def.actions}.
Given a system $\Fa$ of fields, $\TT$ consists of
functions $\TT_W$ that associate
to every field on an $(n+1)$-dimensional bordism $W$ a morphism in $\catc$.
For a disjoint union, one requires 
$\TT_{W\sqcup W'} (f) = \TT_W (f|_W) \otimes \TT_{W'} (f|_{W'})$
for all fields $f\in \Fa (W\sqcup W').$
If $W=W'\cup_N W''$ is obtained by gluing a bordism
$W'$ with outgoing boundary $N$ to a bordism $W''$ with incoming boundary
$N$, then
$\TT_W (f) = \TT_{W''} (f|_{W''}) \circ \TT_{W'} (f|_{W'})$
for all fields $f\in \Fa (W)$. Also, $\TT$ behaves as expected under the action
of homeomorphisms on fields.
These axioms express that the action ought to be local to a certain extent. \\

In Section \ref{sec.semiringmoncat}, motivated by the group algebra
$L^1 (G)$ in harmonic analysis and by the categorical algebra $R[\catc]$
over a ring $R$, we show that $\catc$, together with
an arbitrary choice of an Eilenberg-complete ground semiring $S$, 
determines a pair $(Q^c, Q^m)$ of complete semirings with the same underlying
additive monoid $Q = Q_S (\catc)$, using certain convolution
formulae. 
The composition law $\circ$ in $\catc$ determines the multiplication
$\cdot$ for $Q^c$, while the tensor functor $\otimes$ in $\catc$
(i.e. the monoidal structure) determines the multiplication $\times$
for $Q^m$.\\

We then construct a positive TFT $Z$ for
any given system $\Fa$ of fields and $\catc$-valued action exponentials $\TT$
in Section \ref{sec.constft}.
For a closed manifold $M$, the state vectors are $Q$-valued functions on 
the set of fields $\Fa (M)$ on $M$, solving a certain constraint equation.
Fields on closed $n$-manifolds act as boundary conditions for state sums.
The state sum of a bordism $W$ is the vector given on a boundary field
$f\in \Fa (\partial W)$ by 
\[ Z_W (f) = \sum_{F\in \Fa (W,f)} T_W (F) \in Q, \]
where we sum over all fields $F$ on $W$ which restrict to $f$ on
the boundary. The terms $T_W (F)$ correspond to the exponential
of the action evaluated on $F$ and are characteristic functions determined by $\TT_W$.
The key technical point is that this sum uses the infinite summation
law on a complete semiring and thus yields a well-defined element
of $Q$, despite the fact that the sum may well involve uncountably
many nonzero terms. It is precisely at this point, where Eilenberg's ideas
are brought to bear and where using semirings that are not rings
is crucial. The formalism developed here thus has strong ties to lattice theory
as well as to areas of logic and computer science such as automata theory
and formal languages. It may therefore be viewed as a contribution to
implementing the program envisioned by Baez and Stay in \cite{baezstayrosetta}.

Theorem \ref{thm.statetopinv} establishes the topological invariance
of the state sum. Isotopy invariance of induced maps is provided by
Theorem \ref{thm.isotopyinv}.
The state sum of a disjoint union is calculated in
Theorem \ref{thm.zdisjunion}, while the gluing formula is
the content of Theorem \ref{thm.statesumgluing}.
In Section \ref{sec.linearreps}, we show how
matrix-valued positive TFTs arise from monoidal functors into linear categories.
An important technical role is played by the Schauenburg tensor product. 
Section \ref{sec.cylidemproj} carries out a study of cylindrical state sums.
We observe that such a state sum is idempotent as an immediate consequence
of the gluing theorem. Given closed manifolds $M$ and $N$, we use the
state sum of the cylinder on $M$ (alternatively $N$) to construct a projection
operator $\pi_{M,N}: Z(M\sqcup N)\to Z(M\sqcup N)$, which acts as the
identity on all state sums of bordisms $W$ from $M$ to $N$. There is
a formal analogy to integral transforms given by an integral kernel:
The kernel corresponds to the state sum of the cylinder.
Furthermore, we analyze to what extent the projection of a tensor
product of states breaks up into a tensor product of projections of these states.
In Section \ref{sec.polya}, we derive P\'olya's theory of counting, using
positive TFT methods.
The final Section \ref{sec.examples} gives additional concrete examples and application
patterns. After considering two toy-examples, we construct two more interesting
examples based on the signature of manifolds and bundle spaces.
The final example focuses on multiplicative arithmetic functions arising in number theory.
In a forthcoming paper, we apply the framework of positive topological
field theories presented here in constructing a new high-dimensional TFT defined
on smooth manifolds, which is expected to detect exotic smooth structures on spheres.

A remark on notation: Hoping that no confusion will ensue as a consequence, we use the letter $I$ 
for unit objects of monoidal categories and sometimes also for the unit interval $[0,1]$.

\section{Monoids, Semirings, and Semimodules}
\label{sec.monssemirings}

We recall some foundational material on monoids, semirings and semimodules
over semirings. Such structures seem to have appeared first in Dedekind's study of ideals in a
commutative ring: one can add and multiply two ideals, but one cannot subtract them.
The theory has been further developed by H. S. Vandiver, S. Eilenberg, A. Salomaa,
J. H. Conway, J. S. Golan and many others.
Roughly, a semiring is a ring without general additive inverses. More precisely,
a \emph{semiring} is a set $S$ together with two binary operations $+$ and $\cdot$
and two elements $0,1\in S$ such that $(S,+,0)$ is a commutative monoid,
$(S,\cdot, 1)$ is a monoid, the multiplication $\cdot$ distributes over the addition
from either side, and $0$ is absorbing, i.e. $0\cdot s = 0 = s \cdot 0$ for every $s\in S$.
If the monoid $(S,\cdot,1)$ is commutative, the semiring $S$ is called commutative.
The addition on the \emph{Boolean monoid} $(\bool,+,0)$, $\bool =\{ 0,1 \},$ is the unique
operation such that $0$ is the neutral element and $1+1=1$. 
The Boolean monoid becomes
a commutative semiring by defining $1\cdot 1 =1$. (Actually, the multiplication on $\bool$ is
completely forced by the axioms.) A morphism of
semirings sends $0$ to $0$, $1$ to $1$ and respects addition and multiplication. \\

Let $S$ be a semiring. A \emph{left $S$-semimodule} is a commutative monoid
$(M,+,0_M)$ together with a scalar multiplication $S\times M \to M,$
$(s,m)\mapsto sm,$ such that for all $r,s\in S,$ $m,n\in M,$ we have
$(rs)m=r(sm),$ $r(m+n)=rm+rn,$ $(r+s)m=rm+sm,$ $1m=m,$ and
$r0_M = 0_M = 0m.$ Right semimodules are defined similarly using
scalar multiplications $M\times S \to M,$ $(m,s)\mapsto ms$.
Given semirings $R$ and $S$, an \emph{$R$-$S$-bisemimodule} is a commutative
monoid $(M,+,0)$, which is both a left $R$-semimodule and a right $S$-semimodule
such that $(rm)s=r(ms)$ for all $r\in R,$ $s\in S,$ $m\in M$. (Thus the notation
$rms$ is unambiguous.)
An \emph{$R$-$S$-bisemimodule homomorphism} is a homomorphism
$f:M\to N$ of the underlying additive monoids such that $f(rms) =rf(m)s$ for all $r,m,s$. 
If $R=S$, we shall also speak of an $S$-bisemimodule.
Every semimodule $M$ over a commutative semiring $S$ can and will be
assumed to be both a left and right semimodule with $sm=ms$.
In fact, $M$ is then a bisemimodule, as for all $r,s\in S,$ $m\in M,$
\[ (rm)s = s(rm) = (sr)m = (rs)m = r(sm) = r(ms). \]
In this paper, we will often refer to elements of a semimodule as ``vectors''.

Let $S$ be any semiring, not necessarily commutative.
Regarding the tensor product of a right $S$-semimodule $M$ and 
a left $S$-semimodule $N$, one
has to exercise caution because even when $S$ is commutative,
two nonisomorphic tensor products,
both called \emph{the} tensor product of $M$ and $N$ and both written
$M\otimes_S N$, exist in the literature. A map $\phi: M\times N\to A$ into
a commutative additive monoid $A$ is called \emph{middle $S$-linear}, if
it is biadditive, $\phi (ms,n)=\phi (m,sn)$ for all $m,s,n$, and $\phi (0,0)=0$.
For us, an (algebraic) \emph{tensor product of
$M$ and $N$} is a commutative monoid $M\otimes_S N$ (written additively) satisfying the
following (standard) universal property: $M \otimes_S N$ comes
equipped with a middle $S$-linear map
$M\times N \to M\otimes_S N$ such that
given any commutative monoid $A$ and 
middle $S$-linear map $\phi: M\times N \to A$, there exists a unique monoid homomorphism
$\psi: M\otimes _S N \to A$ such that
\begin{equation} \label{equ.uniproptens}
\xymatrix{
M\times N \ar[r]^{\phi} \ar[d] & A \\
M\otimes_S N \ar@{..>}[ru]_{\psi}
} \end{equation}
commutes. The existence of such a tensor product is shown for example in
\cite{katsovtensor}, \cite{katsov}. To construct it, take $M\otimes_S N$ to be the quotient 
monoid $F/\sim$, where $F$ is the free commutative monoid generated
by the set $M\times N$ and $\sim$ is the congruence relation on $F$
generated by all pairs of the form
\[ ((m+m', n), (m,n)+(m',n)),~ ((m,n+n'), (m,n)+(m,n')),~
  ((ms,n), (m,sn)), \]
$m,m'\in M,$ $n,n' \in N,$ $s\in S$. 
If $M$ is an $R$-$S$-bisemimodule and $N$ an $S$-$T$-bisemimodule, then
the monoid $M\otimes_S N$ as constructed above becomes an $R$-$T$-bisemimodule
by declaring
\[ r\cdot (m\otimes n) = (rm)\otimes n,~ (m\otimes n)\cdot t = m\otimes (nt). \]
If in diagram (\ref{equ.uniproptens}), the monoid $A$ is an $R$-$T$-semimodule
and $M\times N\to A$ satisfies
\[ \phi (rm,n) = r\phi(m,n),~ \phi (m,nt)=\phi(m,n)t \]
(in addition to being middle $S$-linear;
let us call such a map \emph{$R_S T$-linear}), then the uniquely determined monoid map 
$\psi: M\otimes_S N\to A$ is an $R$-$T$-bisemimodule homomorphism, for
\[ \psi (r(m\otimes n)) = \psi ((rm)\otimes n) = \phi (rm,n)=r\phi(m,n)=r\psi(m\otimes n) \]
and similarly for the right action of $T$.
If $R=S=T$ and $S$ is commutative, the above means that
the commutative monoid
$M\otimes_S N$ is an
$S$-semimodule with $s(m\otimes n) = (sm) \otimes n = m\otimes (sn)$ and
the diagram (\ref{equ.uniproptens}) takes place in the category of $S$-semimodules.\\

The tensor product of \cite{takahashi} and \cite{golan} --- let us here write it
as $\otimes'_S$ --- satisfies a different
universal property. A semimodule $C$
is called \emph{cancellative} if $a+c=b+c$ implies $a=b$ for all $a,b,c\in C$.
A monoid $(M,+,0)$ is \emph{idempotent} if $m+m =m$ for
all elements $m\in M$. For example, the Boolean monoid $\bool$ is idempotent.
A nontrivial idempotent semimodule is never cancellative.
Given an arbitrary right $S$-semimodule $M$ and an arbitrary left $S$-semimodule $N$, 
the product $M\otimes'_S N$ is always
cancellative. If one of the two 
semimodules, say $N$, is idempotent, then $M\otimes'_S N$ is idempotent as well,
since $m\otimes' n + m\otimes' n = m\otimes' (n+n) = m\otimes' n$.
Thus if one of $M, N$ is idempotent, then $M\otimes'_S N$ is
trivial, being both idempotent and cancellative. Since in our applications, we desire
nontrivial tensor products of idempotent semimodules, the product $\otimes'_S$
will not be used in this paper. \\

The key feature of the constituent algebraic structures of state modules that will 
allow us to form well-defined state sums
is their completeness. Thus let us recall the notion of a complete monoid, semiring, etc.
as introduced by S. Eilenberg on p. 125 of \cite{eilenbergalm}; see also
\cite{karner}, \cite{krob88}.
A \emph{complete monoid} is a commutative monoid $(M,+,0)$ together with
an assignment $\sum$, called a \emph{summation law}, which assigns to every
family $\{ m_i \}_{i\in I}$ of elements $m_i \in M$, indexed by an arbitrary set $I$, an element
$\sum_{i\in I} m_i$ of $M$ (called the \emph{sum} of the $m_i$), such that
\[ \sum_{i\in \varnothing} m_i =0,~
  \sum_{i\in \{ 1 \}} m_i = m_1,~
 \sum_{i\in \{ 1,2 \}} m_i = m_1 + m_2, \]
and for every partition $I = \bigcup_{j\in J} I_j,$
\[ \sum_{j\in J} \left( \sum_{i\in I_j} m_i \right) = \sum_{i\in I} m_i. \]
Note that these axioms imply that if $\sigma: J\to I$ is a bijection, then
\[ \sum_{i\in I} m_i = \sum_{j\in J} \sum_{i\in \{ \sigma (j) \}} m_i =
 \sum_{j\in J} m_{\sigma (j)}. \]
Also, since a cartesian product $I\times J$ comes with two canonical
partitions, namely $I\times J = \bigcup_{i\in I} \{ i \}\times J =
  \bigcup_{j\in J} I\times \{ j \},$ one has
\begin{equation} \label{equ.sumijmonoid}
\sum_{(i,j)\in I\times J} m_{ij} = \sum_{i\in I} \sum_{j\in J} m_{ij} =
 \sum_{j\in J} \sum_{i\in I} m_{ij} 
\end{equation}
for any family $\{ m_{ij} \}_{(i,j)\in I\times J}$. This is the analog of
Fubini's theorem in the theory of integration.
Given $(M,+,0)$, the summation law $\sum$, if it exists, is not in general
uniquely determined by the addition, as examples in \cite{goldstern85} show.
For a semiring $S$ to be complete one requires that $(S,+,0,\sum)$ be a
complete monoid and adds the infinite distributivity requirements
\[ \sum_{i\in I} ss_i = s \big( \sum_{i\in I} s_i \big),~
  \sum_{i\in I} s_i s = \big( \sum_{i\in I} s_i \big) s. \]
Note that in a complete semiring, the sum over any family of zeros must be
zero, as $\sum_{i\in I} 0 = \sum_i (0\cdot 0) = 0\cdot \sum_i 0 = 0.$
Complete left, right and bisemimodules are defined analogously.
If $\{ s_{ij} \}_{(i,j)\in I\times J}$ is a family in a complete semiring
of the form $s_{ij} = s_i t_j,$ then, using (\ref{equ.sumijmonoid}) together
with the infinite distributivity requirements,
\[ \sum_{(i,j)\in I\times J} s_i t_j = \sum_{i\in I} \sum_{j\in J} s_i t_j =
 \big( \sum_{i\in I} s_i \big) \big( \sum_{j\in J} t_j \big). \]
A semiring is \emph{zerosumfree}, if $s+t=0$ implies $s=t=0$ for all
$s,t$ in the semiring. Clearly, every complete semiring is zerosumfree.
In a ring we have additive inverses, which means that a nontrivial ring
is never zerosumfree and so cannot be endowed with an infinite summation
law that makes it complete. This shows that giving up additive inverses,
thereby passing to semirings that are not rings, is an essential prerequisite
for completeness and in turn essential for the construction of our
topological field theories.

\begin{examples}  \label{exple.completesemirings}
The Boolean semiring $\bool$ is complete with respect to the summation law
\[ \sum_{i\in I} b_i = \begin{cases} 0,& \text{ if } b_i =0 \text{ for all } i, \\
 1,& \text{ otherwise.} \end{cases} \]
Let $\nat^{\infty} = \nat \cup \{ \infty \},$ where $\nat = \{ 0,1,2,\ldots \}$
denotes the natural numbers. Then
$(\nat^{\infty}, +,\cdot, 0,1)$ is a semiring by extending addition and multiplication on $\nat$
via the rules
\[ n+\infty = \infty +n = \infty + \infty = \infty,~ n\in \nat, \]
\[ n\cdot \infty = \infty \cdot n = \infty \text{ for } n\in \nat,~ n>0, \]
\[ \infty \cdot \infty = \infty,~ 0\cdot \infty = \infty \cdot 0 =0. \]
It is complete with
\[ \sum_{i\in I} n_i = \sup \{ \sum_{i\in J} n_i ~|~ J\subset I,~ J \text{ finite} \}. \]
Let $\real_+$ denote the nonnegative real numbers and
$\real^{\infty}_+ = \real_+ \cup \{ \infty \}$. Extending addition and multiplication
as in the case of natural numbers, one obtains the complete
semiring $(\real^{\infty}_+, +,\cdot, 0,1)$.
The tropical semirings
$(\nat^{\infty}, \min, +, \infty, 0)$ and $(\real^{\infty}_+, \min, +, \infty, 0)$,
that is, the sum is the minimum of two numbers and the product is the ordinary addition
of numbers, are complete.
Let $\overline{\nat} = \nat \cup \{ -\infty, \infty \}$ and
$\overline{\real}_+ = \real_+ \cup \{ -\infty, \infty \}$. 
The arctic semirings
$(\overline{\nat}, \max, +, -\infty, 0)$ and $(\overline{\real}_+, \max, +, -\infty, 0)$
are complete. Let $\Sigma$ be a finite alphabet and let $\Sigma^\ast$ be the
free monoid generated by $\Sigma$. Its neutral element is the empty word
$\epsilon$. A subset of $\Sigma^\ast$ is called a \emph{formal language over}
$\Sigma$. Define the product of two formal languages $L,L' \subset \Sigma^\ast$
by
\[ L\cdot L' = \{ ww' ~|~ w\in L,~ w' \in L' \}. \]
Then $(2^{\Sigma^\ast}, \cup, \cdot, \varnothing, \{ \epsilon \})$ is a semiring,
the semiring of formal languages over $\Sigma$. It is complete.
For a set $A$, the power set $2^{A\times A}$ is the set of binary relations $R$ over $A$.
Define the product of two relations $R,R'$ over $A$ by
\[ R\cdot R' = \{ (a,a') ~|~ \exists a_0 \in A: (a,a_0)\in R,~ (a_0,a')\in R' \}. \]
Let $\Delta = \{ (a,a) ~|~ a\in A \}$ be the diagonal. Then 
$(2^{A\times A}, \cup, \cdot, \varnothing, \Delta)$ is a semiring, the semiring of
binary relations over $A$. It is complete.
Any bounded distributive lattice defines a semiring $(L,\vee,\wedge,0,1)$.
It is complete if $L$ is a join-continuous complete lattice, that is,
every subset of $L$ has a supremum in $L$ and
$a\wedge \bigvee_{i\in I} a_i = \bigvee_{i\in I} (a\wedge a_i)$ for any
subset $\{ a_i ~|~ i\in I \} \subset L$.
For example, the ideals of a ring form a complete lattice.
If $S$ is a complete semiring, then the
semiring $S[[q]]$ of formal power series over $S$ becomes a complete semiring
by transferring the summation law on $S$ pointwise to $S[[q]]$, see
\cite{karner}. Completeness of $S$ also implies the completeness of
the semirings of square matrices over $S$. 
If $\{ M_i \}_{i\in I}$ is a family of complete $S$-semimodules, then their
product $\prod_{i\in I} M_i$ is a complete $S$-semimodule.
A semiring $S$ is \emph{additively idempotent} if $s+s=s$ for all $s\in S$.
Every additively idempotent semiring can be embedded in a complete
semiring, \cite{golanwang}.
\end{examples}.

Let $R,S$ be any semirings, not necessarily commutative.
An (associative, unital) \emph{$R$-$S$-semialgebra} is a semiring $A$ which is in addition
an $R$-$S$-bisemimodule such that for all $a,b\in A,$ $r\in R,$ $s\in S$,
\[ r(ab)=(ra)b,~ (ab)s=a(bs). \]
(Note that we refrain from using the term ``$R$-$S$-bisemialgebra'' for such a
structure, since ``bialgebra'' refers to something completely different,
namely a structure with both multiplication and comultiplication.)
If $R=S$, we shall also use the term \emph{two-sided $S$-semialgebra}
for an $S$-$S$-semialgebra.
If $S$ is commutative, then a two-sided $S$-semialgebra $A$ with $sa=as$ is simply
a semialgebra over $S$ in the usual sense, as
\[ (sa)b=s(ab)=(ab)s=a(bs)=a(sb). \]
A morphism of $R$-$S$-semialgebras is a morphism of semirings which is in addition
a $R$-$S$-bisemimodule homomorphism.

\section{Function Semialgebras}
\label{sec.funsemimods}

Let $S$ be a semiring. Given a set $A$, let
\[ \fun_S (A) = \{ f:A\to S \} \]
be the set of all $S$-valued functions on $A$. If $S$ is understood,
we will also write $\fun (A)$ for $\fun_S (A)$.

\begin{prop} \label{prop.funsemimod}
Using pointwise addition and multiplication, $\fun_S (A)$ inherits the
structure of a two-sided $S$-semialgebra from the operations of $S$. 
If $S$ is complete (as a semiring), then $\fun_S (A)$ is complete as a semiring and as
an $S$-bisemimodule. If $S$ is commutative, then $\fun_S (A)$ is a commutative
$S$-semialgebra.
\end{prop}
\begin{proof}
Define $0\in \fun (A)$ to be $0(a)=0\in S$ for all $a\in A$ and define
$1\in \fun (A)$ to be $1(a)=1\in S$ for all $a\in A$. Given $f,g\in \fun (A)$,
define $f+g \in \fun (A)$ and $f\cdot g \in \fun (A)$ by
$(f+g)(a) = f(a)+g(a),$ $(f\cdot g)(a)=f(a)\cdot g(a)$
for all  $a\in A$. Then
$(\fun (A),+,0)$ is a commutative monoid and $(\fun (A),\cdot, 1)$ is a monoid, 
which is commutative if $S$ is commutative. The distributive laws hold and the $0$-function
is absorbing.
Thus $(\fun (A),+,\cdot,0,1)$ is a semiring. 
Given $s\in S,$ define $sf, fs\in \fun (A)$ by
$(sf)(a) = s\cdot (f(a)),$ $(fs)(a)=(f(a))\cdot s,$ respectively, for all $a\in A$. This makes
$\fun (A)$ into a a two-sided $S$-semialgebra. 
If $S$ is a complete semiring, then an infinite summation law in $\fun_S (A)$
can be introduced by
\[ \big( \sum_{i\in I} f_i \big) (a) = \sum_{i\in I} (f_i (a)), \]
$f_i \in \fun (A)$. With this law, $(\fun_S (A),+,0,\sum)$ is a complete
monoid, $\fun_S (A)$ is complete as a semiring and
complete as an $S$-bisemimodule.
\end{proof}

Let $B$ be another set. Then, regarding $\fun_S (A)$ and $\fun_S (B)$ as
$S$-bisemimodules, the tensor product $\fun_S (A)\otimes_S \fun_S (B)$ is
defined. It is an $S$-bisemimodule such that given any $S$-bisemimodule $M$ and
$S_S S$-linear map $\phi: \fun (A)\times \fun (B)\to M,$ there exists a unique
$S$-$S$-bisemimodule homomorphism $\psi: \fun (A)\otimes_S \fun (B)\to M$ such that
\[ \xymatrix{
\fun (A)\times \fun (B) \ar[r]^>>>>>{\phi} \ar[d] & M \\
\fun (A)\otimes_S \fun (B) \ar[ru]_{\psi}
} \]
commutes. The $S$-bisemimodule $\fun (A\times B)$ comes naturally
equipped with an $S_S S$-linear map
\[ \beta: \fun (A)\times \fun (B) \longrightarrow \fun (A\times B),~
  \beta (f,g) = ((a,b)\mapsto f(a)\cdot g(b)). \]
(If $S$ is commutative, then $\beta$ is $S$-bilinear.)
Thus, taking $M = \fun_S (A\times B)$ in the above diagram,
there exists a unique
$S$-$S$-bisemimodule homomorphism $\mu: \fun (A)\otimes_S \fun (B)\to \fun (A\times B)$ such that
\begin{equation} \label{equ.funatimesb} 
\xymatrix{
\fun (A)\times \fun (B) \ar[r]^>>>>>{\beta} \ar[d] & \fun (A\times B) \\
\fun (A)\otimes_S \fun (B) \ar[ru]_{\mu}
} \end{equation}
commutes. In the commutative setting, this homomorphism was studied in \cite{banagl-tensor}, where we
showed that it is generally neither surjective nor injective when $S$ is not
a field. If $A$ and $B$ are finite, then $\mu$ is an isomorphism.
If $A,B$ are infinite but $S$ happens to be a field, then $\mu$ is still injective,
but not generally surjective. This is
the reason why the functional analyst completes the tensor product $\otimes$ using
various topologies available, arriving at products $\hotimes$.
For example, for compact Hausdorff spaces $A$ and $B$, let $C(A),C(B)$ 
denote the Banach spaces of all complex-valued continuous functions on
$A, B,$ respectively, endowed with the supremum-norm, yielding the topology of
uniform convergence. Then the image of 
$\mu: C(A)\otimes C(B) \to C(A\times B)$, while not all of $C(A\times B),$
is however dense in $C(A\times B)$ by the Stone-Weierstra{\ss} theorem.
After completion, $\mu$ induces an isomorphism
$C(A) \hotimes_\epsilon C(B) \cong C(A\times B)$
of Banach spaces, where $\hotimes_\epsilon$ denotes the so-called
$\epsilon$-tensor product or injective tensor product of two locally convex
topological vector spaces. For $n$-dimensional Euclidean space $\real^n$,
let $L^2 (\real^n)$ denote the Hilbert space of square integrable functions
on $\real^n$. Then $\mu$ induces an isomorphism
$L^2 (\real^n) \hotimes L^2 (\real^m) \cong L^2 (\real^{n+m}) =
L^2 (\real^n \times \real^m),$ where $\hotimes$ denotes the Hilbert space
tensor product, a completion of the algebraic tensor product $\otimes$ of two Hilbert spaces.
For more information on topological tensor products see 
\cite{schatten}, \cite{grothendiecktensor}, \cite{treves}.
In \cite{banagl-tensor} we show that even over the smallest complete (in particular
zerosumfree) and additively idempotent commutative semiring, namely the Boolean semiring $\bool$,
and for the smallest infinite cardinal $\aleph_0$, modeled by a countably infinite set $A$,
the map $\mu$ is not surjective, which means that in the context of the present paper,
some form of completion must be used as well.
However, there is an even more serious complication which arises over semirings:
In marked contrast to the
situation over fields, the canonical map $\mu$ ceases to be \emph{injective} in general
when one studies functions with values in a semiring $S$.
Given two infinite sets $A$ and $B$, we construct explicitly 
in \cite{banagl-tensor}
a commutative, additively idempotent
semiring $S= S(A,B)$ such that $\mu: \fun_S (A) \otimes \fun_S (B) \to 
\fun_S (A\times B)$ is
not injective. This has the immediate consequence that in 
performing functional analysis over a semiring which is not a field, one
cannot identify the function $(a,b)\mapsto f(a)g(b)$ on $A\times B$ with $f\otimes g$ for 
$f\in \fun_S (A)$, $g\in \fun_S (B)$. Thus the algebraic tensor product is not the correct device
to formulate positive topological field theories.

In the boundedly complete idempotent setting, Litvinov, Maslov and Shpiz have
constructed in \cite{litmasshpiztensor} a tensor product, let us here write it as $\hotimes$, which 
for bounded functions does not exhibit the above deficiencies
of the algebraic tensor product. Any idempotent semiring $S$ is a partially ordered
set with respect to the order relation $s\leq t$ if and only if $s+t=t$; $s,t\in S$.
Then the addition has the interpretation of a least upper bound,
$s+t = \sup \{ s,t \}$. The semiring $S$ is called \emph{boundedly complete}
($b$-complete) if every subset of $S$ which is bounded above has a supremum.
(The supremum of a subset, if it exists, is unique.) The above semiring $S(A,B)$
is $b$-complete.
Given a $b$-complete
commutative idempotent semiring $S$ and $b$-complete idempotent semimodules
$V,W$ over $S$, Litvinov, Maslov and Shpiz define a tensor product $V\hotimes W$,
which is again idempotent and $b$-complete. The fundamental difference to the
algebraic tensor product lies in allowing \emph{infinite} sums of elementary tensors.
A linear map $f:V\to W$ is called $b$-linear if $f(\sup V_0) =\sup f(V_0)$ for every
bounded subset $V_0 \subset V$. The canonical map $\pi: V\times W \to V\hotimes W$
is $b$-bilinear. For each $b$-bilinear map $f: V\times W \to U$ there exists a
unique $b$-linear map $f_{\hotimes}: V\hotimes W \to U$ such that
$f=f_{\hotimes} \pi$. Given any set $A$, let $\mathcal{B}_S (A)$ denote the
set of bounded functions $A\to S$. Then $\mathcal{B}_S (A)$ is a $b$-complete
idempotent $S$-semimodule. According to \cite[Prop. 5]{litmasshpiztensor},
$\mathcal{B}_S (A) \hotimes \mathcal{B}_S (B)$ and $\mathcal{B}_S (A\times B)$
are isomorphic for arbitrary sets $A$ and $B$. Note that when $S$ is complete,
then every $S$-valued function is bounded and thus 
\[ \fun_S (A) \hotimes \fun_S (B) = \mathcal{B}_S (A)\hotimes \mathcal{B}_S (B) \cong
    \mathcal{B}_S (A\times B) = \fun_S (A\times B). \]
In light of the above remarks and in order to make function semialgebras into
a monoidal category (cf. Proposition \ref{prop.funsmonoidal}), we adopt the following definition.
\begin{defn} \label{def.funtensorprod}
The \emph{functional tensor product} $\fun_S (A) \hotimes \fun_S (B)$
of two function semimodules $\fun_S (A)$ and $\fun_S (B)$ is given by
\[ \fun_S (A) \hotimes \fun_S (B) = \fun_S (A\times B). \]
\end{defn}
Diagram (\ref{equ.funatimesb}) can be rewritten as
\[ 
\xymatrix{
\fun (A)\times \fun (B) \ar[r]^>>>>>{\beta} \ar[d] & \fun (A) \hotimes \fun (B) \\
\fun (A)\otimes_S \fun (B) \ar[ru]_{\mu}
} \]
\begin{remark}  \label{rem.wrongfuncttensprod}
It might be tempting to define the functional tensor product of $\fun (A)$ and
$\fun (B)$ in a different way, namely as the subsemimodule of $\fun (A\times B)$
consisting of all functions $F:A\times B\to S$ that can be written in the form
\[ F(a,b) = \sum_{i=1}^k f_i (a)g_i (b) \]
for some $f_i \in \fun (A),$ $g_i \in \fun (B)$, in other words, let the functional tensor product be the
image of $\mu$. Such a definition
would be incorrect, however, since the resulting smaller semimodule would in
general be too small to contain the main invariant of a TFT, the state sum,
as constructed in Section \ref{sec.constft}. In Example \ref{exple.statesumdoesnotdecomp}, we
construct a system of fields on manifolds and action functionals such that the resulting state sum is
not in the image of $\mu$.
\end{remark}

Given two functions $f\in \fun (A),$ $g\in \fun (B)$ we call
$f\hotimes g = \beta (f,g)$ the (functional) \emph{tensor product} of $f$ and $g$.
We note that $f\hotimes g$ depends on the multiplication on $S$. If the
underlying additive monoid of $S$ is equipped with a different multiplication,
then $f\hotimes g$ will of course change.
For the functional tensor product with $S$ we have
\[ \fun (A)\hotimes S \cong \fun (A) \hotimes \fun (\{ * \}) =
  \fun (A\times \{ * \}) \cong \fun (A). \]

For fixed $S$, $\fun_S (-)$ is a contravariant functor $\fun_S: \mathbf{Sets}\to 
S\text{-}S\text{-}\mathbf{SAlgs}$ from the category of sets to the category of two-sided
$S$-semialgebras: A morphism $\phi:A\to B$ of sets induces a
morphism of two-sided $S$-semialgebras $\fun (\phi): \fun (B)\to \fun (A)$ by setting
$\fun (\phi)(f) = f\circ \phi$. 
Clearly, $\fun (\id_A) = \id_{\fun (A)}$ and
$\fun (\psi \circ \phi) = \fun (\phi)\circ \fun (\psi)$ for
$\psi: B\to C$.
\begin{prop}
The functor $\fun_S (-)$ is faithful. 
\end{prop}
\begin{proof}
Given functions $\phi, \psi: A\to B$
such that $\fun (\phi) = \fun (\psi),$ we know that
$f(\phi (a))=f(\psi (a))$ for all $f\in \fun (B)$ and $a\in A$. 
Taking $f$ to be the characteristic function $\chi_{\phi (a)}$,
given by $\chi_{\phi (a)} (b) = \delta_{\phi (a),b},$ we find that
$\chi_{\phi (a)}(\psi (a)) = \chi_{\phi (a)} (\phi (a)) = 1,$ whence
$\phi (a) = \psi (a)$.
\end{proof}
Let $\bfun_S$ be the category whose objects are $\fun_S (A)$ for all sets $A$,
and whose morphisms are those morphisms of two-sided $S$-semialgebras
$\fun_S (B)\to \fun_S (A)$ that have the form $\fun_S (\phi)$ for some
$\phi:A\to B$. The preceding remarks imply that this is indeed a category with
the obvious composition law.
For $\fun (\phi):\fun (B)\to \fun (A)$ and $\fun (\phi'): \fun (B')\to \fun (A'),$
we define
\[ \fun (\phi) \hotimes \fun (\phi') :\fun (B)\hotimes \fun (B')  \longrightarrow
  \fun (A)\hotimes \fun (A') \]
to be
\[  \fun (\phi \times \phi'):
  \fun (B\times B') \longrightarrow
  \fun (A\times A'), \]
where $\phi \times \phi': A\times A' \to B\times B'$ is the cartesian product
$(\phi \times \phi')(a,a')=(\phi (a), \phi' (a'))$.
In this manner, 
the functional tensor product becomes a functor $\hotimes: \bfun_S \times \bfun_S \to \bfun_S$.
Define the unit object of $\bfun_S$ to be $I = \fun_S (\{ * \}) = S$ and
define associators $a$ by
\[ a_{A,B,C} = \fun (\alpha_{A,B,C}):
\fun (A)\hotimes (\fun (B)\hotimes \fun (C)) \stackrel{\cong}{\longrightarrow}
 (\fun (A) \hotimes \fun (B))\hotimes \fun (C), \]
where
$\alpha_{A,B,C}: (A\times B)\times C \stackrel{\cong}{\longrightarrow}
 A\times (B\times C)$
is the standard associator that makes $\mathbf{Sets}$ into a monoidal
category, given by $\alpha_{A,B,C} ((a,b),c)=(a,(b,c)),$
$a\in A,$ $b\in B,$ $c\in C$. 
Define left unitors $l$ by
\[ l_A = \fun (\lambda_A): I\hotimes \fun (A) = \fun (\{ * \} \times A) 
\stackrel{\cong}{\longrightarrow} \fun (A), \]
where $\lambda_A: A\to \{ * \}\times A$ is the bijection
$\lambda_A (a)=(*,a)$.
Similarly, right unitors $r$ are defined as
\[ r_A = \fun (\rho_A): \fun (A) \hotimes I = \fun (A \times \{ * \}) 
\stackrel{\cong}{\longrightarrow} \fun (A), \]
where $\rho_A: A\to A \times \{ * \}$ is the bijection
$\rho_A (a)=(a,*)$. A braiding $b$ is given by
\[ b_{A,B} = \fun (\beta_{B,A}): \fun (A)\hotimes \fun (B) 
\stackrel{\cong}{\longrightarrow} \fun (B)\hotimes \fun (A), \]
with $\beta_{B,A}: B\times A \stackrel{\cong}{\longrightarrow}
 A\times B,$ $\beta_{B,A} (b,a)=(a,b).$ Then it is a routine task to verify the
following assertion:
\begin{prop} \label{prop.funsmonoidal}
The septuple $(\bfun_S, \hotimes, I, a, l, r, b)$ is a symmetric monoidal category.
\end{prop}
Let $S$ be a complete semiring and let $A,B,C$ be sets.
We put $\fun (A)\hotimes \fun (B)\hotimes \fun (C) =
\fun (A\times B\times C)$ (triples), etc.
A \emph{contraction}
\[ \gamma: \fun (A)\hotimes \fun (B)\hotimes \fun (B)\hotimes \fun (C)
 \longrightarrow \fun (A)\hotimes \fun (C) \]
can then be defined, using the summation law in $S$, by
\[ \gamma (f)(a,c) = \sum_{b\in B} f(a,b,b,c), \]
$f: A\times B\times B\times C \to S,$ $(a,c)\in A\times C$.
Given $f\in \fun (A)\hotimes \fun (B)$ and $g\in \fun (B)\hotimes \fun (C)$,
we shall also write
$\langle f,g \rangle = \gamma (f\hotimes g).$
This contraction appears in describing the behavior of our state sum
invariant under gluing of bordisms. The proof of the following two statements is
straightforward.
\begin{prop} \label{prop.contractbilinear}
The contraction
\[ \langle -,- \rangle: (\fun (A)\hotimes \fun (B))\times
 (\fun (B)\hotimes \fun (C)) \longrightarrow
 \fun (A)\hotimes \fun (C) \]
is $S_S S$-linear.
\end{prop}
\begin{prop} \label{prop.contractassoc}
The contraction $\langle -,- \rangle$ is associative, that is, given
four sets $A,B,C,D$ and elements $f\in \fun (A)\hotimes \fun (B),$
$g\in \fun (B)\hotimes \fun (C)$ and $h\in \fun (C)\hotimes \fun (D)$,
the equation
\[ \langle \langle f,g \rangle, h \rangle = \langle f, \langle g,h \rangle \rangle \]
holds in $\fun (A)\hotimes \fun (D)$.
\end{prop}

\section{Convolution Semirings Associated to Monoidal Categories}
\label{sec.semiringmoncat}

Let $S$ be a semiring and
let $\catc$ be a small category. 
The symbol $\Mor (\catc)$ will denote the set of all morphisms of $\catc$.
Given a morphism $\alpha \in \Mor (\catc),$ $\dom (\alpha)$ is its domain
and $\cod (\alpha)$ its codomain.
We shall show that if $S$ is complete, then this data determines a (complete)
semiring $Q = Q_S (\catc)$. Our construction is motivated by similar ones in 
harmonic analysis. Suppose that $\catc$ is a group, i.e. a groupoid with one
object $*$. If the group $G = \Hom_\catc (*,*)$ is locally compact Hausdorff,
then one may consider $L^1 (G),$ the functions on $G$ that are integrable with
respect to Haar measure. A multiplication, called the convolution product, on 
$L^1 (G)$ is given by
\[ (f*g)(t)=\int_G f(s)g(s^{-1} t)ds. \]
The resulting algebra is called the group algebra or convolution algebra.
If the functions on $G$ take values in a complete semiring $S$, one can
drop the integrability assumption (and in fact the assumption that $G$ be
topological) and use the summation law in $S$ instead of integration against
Haar measure:
\[ (f*g)(t) = \sum_{ss' =t} f(s)g(s'). \]
Written like this, it now becomes clear that even the assumption that $\catc$
be a group is obsolete. There are also close connections of our construction
of $Q_S (\catc)$ to the categorical algebra $R[\catc]$ associated to a 
locally finite category $\catc$ and a ring $R$. A category is locally finite,
if every morphism can be factored in only finitely many ways as a product
of nonidentity morphisms. Elements of $R[\catc]$
are functions $\Mor (\catc)\to R$, i.e. $R[\catc] = \fun_R (\Mor (\catc))$.
The local finiteness of $\catc$ ensures that the convolution product in $R$
is well-defined. The difference to our construction is that we do not need
$\catc$ to be locally finite, but we need $R=S$ to be a complete semiring.\\

Let $S$ be a semiring and $\catc$ a small category.
For every pair $(X,Y)$ of objects
in $\catc$, we then have the commutative monoid
$(\fun_S (\Hom_\catc (X,Y)),+,0)$ and can form the product
\[ Q_S (\catc) = \prod_{X,Y \in \Ob \catc} \fun_S (\Hom_\catc (X,Y)). \]
Elements $f\in Q_S (\catc)$ are families $f=(f_{XY})$ of functions
$f_{XY}: \Hom_\catc (X,Y)\to S$. Such a family is of course the same thing
as a function $f: \Mor (\catc)\to S$ and $Q_S (\catc) = \fun_S (\Mor (\catc)),$
but we find it convenient to keep the bigrading of $Q_S (\catc)$ by pairs of objects.
The addition on $Q_S (\catc)$ preserves the bigrading, i.e.
$(f+ g)_{XY} = f_{XY} + g_{XY}$.
The neutral element $0\in Q_S (\catc)$ is given by
$0 = (0_{XY})$ with $0_{XY} =0: \Hom_\catc (X,Y)\to S$ the
constant map sending every morphism $X\to Y$ to $0\in S$.
Then $(Q_S (\catc),+,0)$ is a commutative monoid.
\begin{prop} \label{prop.qsccompletemonoid}
If $S$ is a complete semiring, then $Q_S (\catc)$ inherits a
summation law from $S,$ making $(Q_S (\catc), +,0)$ into
a complete monoid.
\end{prop}
\begin{proof}
Given an index set $J$ and a family $\{ f_j \}_{j\in J}$ of
elements $f_j \in Q_S (\catc),$ declare a summation law
on $Q_S (\catc)$ by
\[ (\sum_{j\in J} f_j)_{XY} (\gamma) = \sum_{j\in J}
 ((f_j)_{XY} (\gamma)) \in S, \]
using on the right hand side the summation law of $S$.
In this formula,
$\gamma: X\to Y$ is a morphism in $\catc$.
Using this summation law, all the axioms for a complete monoid
are satisfied. For example, given a partition
$J = \bigcup_{k\in K} J_k,$ one has
\begin{eqnarray*}
(\sum_{k\in K} (\sum_{j\in J_k} f_j))_{XY} (\gamma)
&=& \sum_{k\in K} ((\sum_{j\in J_k} f_j)_{XY} (\gamma))
= \sum_{k\in K} (\sum_{j\in J_k} ((f_j)_{XY} (\gamma))) \\
& = &
\sum_{j\in J} ((f_j)_{XY} (\gamma)) =
(\sum_{j\in J} f_j)_{XY} (\gamma), 
\end{eqnarray*}
using the partition axiom provided by the completeness of $S$.
\end{proof}
Assume that $S$ is complete. Then we can define a (generally noncommutative)
multiplication $\cdot: Q_S (\catc)\times Q_S (\catc) \to
Q_S (\catc)$ by $(f_{XY})\cdot (g_{XY})= (h_{XY}),$ with
$h_{XY}: \Hom_\catc (X,Y)\to S$ defined on a morphism
$\gamma: X\to Y$ by the convolution formula
\[ h_{XY} (\gamma)= \sum_{\beta \alpha = \gamma}
  g_{ZY} (\beta)\cdot f_{XZ} (\alpha), \]
where $\alpha, \beta$ range over all $\alpha \in \Hom_\catc (X,Z),$
$\beta \in \Hom_\catc (Z,Y)$ with $\gamma = \beta \circ \alpha$.
The right hand side of this formula uses the multiplication of the
semiring $S$. Note that the sum may well be infinite, but nevertheless
yields a well-defined element of $S$ by completeness.
An element $1\in Q_S (\catc)$ is given by $(f_{XY})$ with
\[ f_{XY} = \begin{cases} f_{XX},& X=Y \\ 0 & X\not= Y \end{cases},~
  f_{XX}(\alpha)=\begin{cases} 1,& \alpha =\id_X \\ 0,& \alpha \not= \id_X
  \end{cases}. \]

\begin{prop}
The quintuple $(Q_S (\catc), +, \cdot, 0, 1)$ is a complete semiring.
\end{prop}
\begin{proof}
 Let us
verify that $(Q_S (\catc),\cdot, 1)$ is a monoid.
Given $(f_{XY}),$ $(g_{XY})$ and $(h_{XY}) \in Q_S (\catc),$
let $(l_{XY}) = (f_{XY})\cdot (g_{XY})$ and 
$(r_{XY}) = (g_{XY})\cdot (h_{XY})$. Let $\gamma: X\to Y$
be a morphism in $\catc$. Then
\begin{eqnarray*}
((l_{XY})\cdot (h_{XY}))_{XY} (\gamma) & = &
  \sum_{\beta \alpha = \gamma} h_{ZY} (\beta)\cdot l_{XZ} (\alpha) 
 =   \sum_{\beta \alpha = \gamma} h_{ZY} (\beta)\cdot 
   \sum_{\sigma \tau =\alpha} g_{UZ} (\sigma)\cdot f_{XU} (\tau) \\
& = &  \sum_{\beta \sigma \tau = \gamma} h_{ZY} (\beta)\cdot 
    g_{UZ} (\sigma)\cdot f_{XU} (\tau) 
 = \sum_{\zeta\in L}  s(\zeta),
\end{eqnarray*}
where
$L = \{ \zeta = (\beta, \sigma, \tau) ~|~ \beta \sigma \tau = \gamma \},$
involving all possible factorizations of $\gamma$ into three factors
\[ X \stackrel{\tau}{\longrightarrow} U \stackrel{\sigma}{\longrightarrow}
 Z \stackrel{\beta}{\longrightarrow} Y, \]
and the function $s$ is given by
$s(\beta, \sigma, \tau) =  h_{ZY} (\beta)\cdot 
    g_{UZ} (\sigma)\cdot f_{XU} (\tau).$
On the other hand,
\begin{eqnarray*}
((f_{XY})\cdot (r_{XY}))_{XY} (\gamma) & = &
  \sum_{\beta \alpha = \gamma} r_{ZY} (\beta)\cdot f_{XZ} (\alpha) 
=   \sum_{\beta \alpha = \gamma} 
  (\sum_{\sigma \tau = \beta} h_{VY} (\sigma)\cdot 
   g_{ZV} (\tau)) \cdot f_{XZ} (\alpha) \\
& = &  \sum_{\sigma \tau \alpha = \gamma} h_{VY} (\sigma)\cdot 
    g_{ZV} (\tau)\cdot f_{XZ} (\alpha) 
 =  \sum_{\zeta\in R}  s(\zeta),
\end{eqnarray*}
where
$R = \{ \zeta = (\sigma, \tau, \alpha) ~|~ \sigma \tau \alpha = \gamma \},$
involving all possible factorizations of $\gamma$ into three factors
\[ X \stackrel{\alpha}{\longrightarrow} Z \stackrel{\tau}{\longrightarrow}
 V \stackrel{\sigma}{\longrightarrow} Y, \]
and the function $s$ is the same as above. As $L=R$, this shows that
the multiplication $\cdot$ on $Q_S (\catc)$ is associative.
The element $1\in Q_S (\catc)$ is neutral with respect to this
multiplication, for
\begin{eqnarray*}
((f_{XY})\cdot 1)_{XY} (\gamma) & = &
  \sum_{\beta \alpha = \gamma} 1_{ZY} (\beta)\cdot f_{XZ} (\alpha) \\
 & = & \sum_{\beta \alpha = \gamma,~ Z\not= Y} 1_{ZY} (\beta)\cdot f_{XZ} (\alpha) +
  \sum_{\beta \alpha = \gamma,~ Z=Y} 1_{YY} (\beta)\cdot f_{XY} (\alpha) \\
& = & \sum_{\beta \alpha = \gamma} 1_{YY} (\beta)\cdot f_{XY} (\alpha) \\
 & = & \sum_{\beta \alpha = \gamma,~ \beta \not= \id_Y} 1_{YY} (\beta)\cdot f_{XY} (\alpha) +
  \sum_{\beta \alpha = \gamma,~ \beta=\id_Y} 1_{YY} (\beta)\cdot f_{XY} (\alpha) \\
& = &  \sum_{\alpha = \gamma} 1_{YY} (\id_Y)\cdot f_{XY} (\gamma) \\
& = & f_{XY} (\gamma)
\end{eqnarray*}
and similarly $1\cdot (f_{XY}) = (f_{XY}).$ Therefore, $(Q_S (\catc),\cdot,1)$ is a
monoid. The distribution laws are readily verified and
the element $0\in Q_S (\catc)$ is obviously absorbing.

By Proposition \ref{prop.qsccompletemonoid}, $(Q_S (\catc),+,0)$ is a complete
monoid. It remains to be shown that the summation law satisfies
the infinite distributivity requirement with respect to $\cdot$ on $Q_S (\catc)$.
As above, $\gamma: X\to Y$ is a morphism in $\catc$.
Given an element $g\in Q_S (\catc),$ an index set $J$ and a family $\{ f_j \}_{j\in J}$ of
elements $f_j \in Q_S (\catc),$ let $\Gamma = \{ (\beta, \alpha) ~|~ \beta \alpha =\gamma \}$ 
and let $P$ be the cartesian product $P = J\times \Gamma$.
Note that $P$ comes with two natural partitions, namely into sets
$\{ j \} \times \Gamma,$ $j\in J,$ and into sets $J\times \{ (\beta, \alpha) \}$,
$(\beta,\alpha)\in \Gamma$. Using the infinite distribution axiom for $S$ and
the partition axiom, we have
\begin{eqnarray*}
(g\cdot \sum_{j\in J} f_j)_{XY} (\gamma) & = &
  \sum_{(\beta, \alpha)\in \Gamma} (\sum_{j\in J} f_j)_{ZY} (\beta)\cdot
   g_{XZ} (\alpha) 
 =  \sum_{(\beta, \alpha)\in \Gamma} (\sum_{j\in J}(( f_j)_{ZY} (\beta)))\cdot
   g_{XZ} (\alpha) \\
& = & \sum_{(j,(\beta,\alpha))\in P} (f_j)_{ZY} (\beta)\cdot g_{XZ} (\alpha) 
 =  \sum_{j\in J} (\sum_{(\beta,\alpha)\in \Gamma} (f_j)_{ZY} (\beta)\cdot
    g_{XZ} (\alpha)) \\
& = & \sum_{j\in J} ((g\cdot f_j)_{XY} (\gamma)) 
 =  (\sum_{j\in J} (g\cdot f_j))_{XY} (\gamma).
\end{eqnarray*}
Similarly $(\sum f_j)\cdot g = \sum (f_j \cdot g).$
\end{proof}  
The above proof shows that the (strict) associativity of the composition law $\circ$
of $\catc$ implies the associativity of the multiplication $\cdot$ in $Q_S (\catc)$.
Similarly, the presence of identity morphisms in $\catc$ implies the existence of
a unit element $1$ for the multiplication. It is clear that the multiplication $\cdot$
on $Q_S (\catc)$ is generally noncommutative, even if $S$ happens to be
commutative. For example, let $\catc$ be the category given by two distinct objects
$X,Y$ and morphisms
\[ \Hom_\catc (X,X)=\{ \id_X \},~
   \Hom_\catc (Y,Y)=\{ \id_Y \},~
   \Hom_\catc (X,Y)=\{ \gamma \},~
   \Hom_\catc (Y,X)= \varnothing. \]
The composition law is uniquely determined. Let $S=\bool$ be the Boolean
semiring, which is commutative. If
$f(\gamma)=1,$ $g(\id_X)=0,$ $g(\gamma)=0,$ $g(\id_Y)=1,$
then
\[ (f\cdot g)(\gamma) = g(\gamma)f(\id_X) + g(\id_Y)f(\gamma)=1, \]
but
\[ (g\cdot f)(\gamma)=f(\gamma)g(\id_X) + f(\id_Y)g(\gamma)=0. \]

Now suppose that $(\catc, \otimes, I)$ is a strict monoidal category.
Then, using the monoidal structure, we can define a different
multiplication $\times: Q_S (\catc)\times Q_S (\catc) \to
Q_S (\catc)$ by $(f_{XY})\times (g_{XY})= (h_{XY}),$ with
$h_{XY}: \Hom_\catc (X,Y)\to S$ defined on a morphism
$\gamma: X\to Y$ as the convolution
\[ h_{XY} (\gamma)= \sum_{\alpha \otimes \beta = \gamma}
  g_{X'' Y''} (\beta)\cdot f_{X' Y'} (\alpha), \]
where $\alpha, \beta$ range over all $\alpha \in \Hom_\catc (X',Y'),$
$\beta \in \Hom_\catc (X'',Y'')$ such that
$X = X' \otimes X'',$ $Y= Y' \otimes Y''$ and
 $\gamma = \alpha \otimes \beta$.
Again, the right hand side of this formula uses the multiplication of the
complete ground semiring $S$. An element $1^\times \in Q_S (\catc)$ is given by $(f_{XY})$ with
\[ f_{XY} = \begin{cases} f_{II},& X=Y=I \text{ (unit obj.)} \\ 0 & \text{otherwise} \end{cases},~
  f_{II}(\alpha)=\begin{cases} 1,& \alpha =\id_I \\ 0,& \alpha \not= \id_I
  \end{cases}. \]

\begin{prop}
The quintuple $(Q_S (\catc), +, \times, 0, 1^\times)$ is a complete semiring.
\end{prop}
\begin{proof}
We check that $(Q_S (\catc),\cdot, 1^\times)$ is a monoid, knowing already
that $(Q_S (\catc),+,0)$ is a commutative monoid.
Given $(f_{XY}),$ $(g_{XY})$ and $(h_{XY}) \in Q_S (\catc),$
let $(l_{XY}) = (f_{XY})\times (g_{XY})$ and 
$(r_{XY}) = (g_{XY})\times (h_{XY})$. Let $\gamma: X\to Y$
be a morphism in $\catc$. Then
\begin{eqnarray*}
((l_{XY})\times (h_{XY}))_{XY} (\gamma) & = &
  \sum_{\alpha \otimes \beta = \gamma} h_{X'' Y''} (\beta)\cdot l_{X' Y'} (\alpha) \\
& = &  \sum_{\alpha \otimes \beta = \gamma} h_{X'' Y''} (\beta)\cdot 
   \sum_{\sigma \otimes \tau =\alpha} g_{X_2 Y_2} (\tau)\cdot f_{X_1 Y_1} (\sigma) \\
& = &  \sum_{\sigma \otimes \tau \otimes \beta = \gamma} h_{X'' Y''} (\beta)\cdot 
    g_{X_2 Y_2} (\tau)\cdot f_{X_1 Y_1} (\sigma) \\
& = & \sum_{\zeta\in L}  s(\zeta),
\end{eqnarray*}
where
$L = \{ \zeta = (\sigma, \tau, \beta) ~|~ \sigma \otimes \tau \otimes \beta = \gamma \},$
involving all possible factorizations of $\gamma$ into three tensor factors
\[ X = X_1 \otimes X_2 \otimes X''
\stackrel{\sigma \otimes \tau \otimes \beta}{\longrightarrow} 
Y_1 \otimes Y_2 \otimes Y'' = Y, \]
and the function $s$ is given by
$s(\sigma, \tau, \beta) =  h_{X'' Y''} (\beta)\cdot 
    g_{X_2 Y_2} (\tau)\cdot f_{X_1 Y_1} (\sigma).$
(Note that since $(\catc, \otimes, I)$ is strict, we do not have to indicate
parentheses. As is customary in strict monoidal categories, we write
$X_1 \otimes X_2 \otimes X''$ for 
$(X_1 \otimes X_2) \otimes X'' = X_1 \otimes (X_2 \otimes X'')$. Similarly for morphisms.)
On the other hand,
\begin{eqnarray*}
((f_{XY})\times (r_{XY}))_{XY} (\gamma) & = &
  \sum_{\alpha \otimes \beta = \gamma} r_{X'' Y''} (\beta)\cdot f_{X' Y'} (\alpha) \\
& = &  \sum_{\alpha \otimes \beta = \gamma} 
  (\sum_{\sigma \otimes \tau = \beta} h_{X_2 Y_2} (\tau)\cdot 
   g_{X_1 Y_1} (\sigma)) \cdot f_{X' Y'} (\alpha) \\
& = &  \sum_{\alpha \otimes \sigma \otimes \tau = \gamma} h_{X_2 Y_2} (\tau)\cdot 
    g_{X_1 Y_1} (\sigma)\cdot f_{X' Y'} (\alpha) \\
& = & \sum_{\zeta\in R}  s(\zeta),
\end{eqnarray*}
where
$R = \{ \zeta = (\alpha, \sigma, \tau) ~|~ \alpha \otimes \sigma \otimes \tau = \gamma \},$
involving all possible factorizations of $\gamma$ into three tensor factors
\[ X = X' \otimes X_1 \otimes X_2 
 \stackrel{\alpha \otimes \sigma \otimes \tau}{\longrightarrow}
 Y' \otimes Y_1 \otimes Y_2 =Y, \]
and the function $s$ is the same as above. As $\catc$ is strict, we have
$L=R$, which shows that
the multiplication $\times$ on $Q_S (\catc)$ is associative.
The element $1^\times \in Q_S (\catc)$ is neutral with respect to this
multiplication, for
\begin{eqnarray*}
((f_{XY})\times 1^\times)_{XY} (\gamma) & = &
  \sum_{\alpha \otimes \beta = \gamma} 1^\times_{X'' Y''} (\beta)\cdot f_{X' Y'} (\alpha) \\
 & = & \sum_{\alpha \otimes \beta = \gamma,~ X''\not= I \text{ or } Y''\not= I} 
       1^\times_{X'' Y''} (\beta)\cdot f_{X' Y'} (\alpha) \\
& & \hspace{1cm}+
  \sum_{\alpha \otimes \beta= \gamma,~ X''=Y''=I} 1^\times_{X'' Y''} (\beta)\cdot f_{X' Y'} (\alpha) \\
& = & \sum_{\alpha \otimes \beta= \gamma} 1^\times_{II} (\beta)\cdot f_{X' Y'} (\alpha) \\
 & = & \sum_{\alpha \otimes \beta= \gamma,~ \beta \not= \id_I} 1^\times_{II} (\beta)\cdot f_{X' Y'} (\alpha) +
  \sum_{\alpha \otimes \beta= \gamma,~ \beta=\id_I} 1^\times_{II} (\beta)\cdot f_{X' Y'} (\alpha) \\
& = &  \sum_{\alpha \otimes \id_I = \gamma} f_{XY} (\gamma) \\
& = & f_{XY} (\gamma)
\end{eqnarray*}
and similarly $1^\times \times (f_{XY}) = (f_{XY}).$ 
(In this calculation, we have used $X' \otimes I = X',$ $\alpha \otimes \id_I = \alpha,$ valid in a
strict monoidal category such as $\catc$.)
Therefore, $(Q_S (\catc),\cdot,1^\times)$ is a
monoid. The distribution laws are satisfied and
the element $0\in Q_S (\catc)$ is absorbing.

By Proposition \ref{prop.qsccompletemonoid}, $(Q_S (\catc),+,0)$ is a complete
monoid. It remains to be shown that the summation law satisfies
the infinite distributivity requirement with respect to $\times$ on $Q_S (\catc)$.
As above $\gamma: X\to Y$ is a morphism in $\catc$.
Given an element $g\in Q_S (\catc),$ an index set $J$ and a family $\{ f_j \}_{j\in J}$ of
elements $f_j \in Q_S (\catc),$ let $\Gamma = \{ (\alpha, \beta) ~|~ \alpha \otimes \beta =\gamma \}$ 
and let $P$ be the cartesian product $P = J\times \Gamma$.
Note that $P$ comes with two natural partitions, namely into sets
$\{ j \} \times \Gamma,$ $j\in J,$ and into sets $J\times \{ (\alpha,\beta) \}$,
$(\alpha,\beta)\in \Gamma$. Using the infinite distribution axiom for $S$ and
the partition axiom, we have
\begin{eqnarray*}
(g\times \sum_{j\in J} f_j)_{XY} (\gamma) & = &
  \sum_{(\alpha, \beta)\in \Gamma} (\sum_{j\in J} f_j)_{X'' Y''} (\beta)\cdot
   g_{X' Y'} (\alpha) \\
& = & \sum_{(\alpha, \beta)\in \Gamma} (\sum_{j\in J}(( f_j)_{X'' Y''} (\beta)))\cdot
   g_{X' Y'} (\alpha) \\
& = & \sum_{(j,(\alpha,\beta))\in P} (f_j)_{X'' Y''} (\beta)\cdot g_{X' Y'} (\alpha) \\
& = & \sum_{j\in J} (\sum_{(\alpha,\beta)\in \Gamma} (f_j)_{X'' Y''} (\beta)\cdot
    g_{X' Y'} (\alpha)) \\
& = & \sum_{j\in J} ((g\times f_j)_{XY} (\gamma)) \\
& = & (\sum_{j\in J} (g\times f_j))_{XY} (\gamma).
\end{eqnarray*}
\end{proof}  
It is crucial in the above proof to know that $\catc$ is strict. In order for the 
multiplication $\times$ to be associative, one must know that the sets of
factorizations $L_\gamma = \{ (\sigma, \tau, \beta) ~|~ 
(\sigma \otimes \tau)\otimes \beta = \gamma \}$ and 
$R_\gamma = \{ (\sigma, \tau, \beta) ~|~ 
\sigma \otimes (\tau \otimes \beta) = \gamma \}$
are equal. This holds when $\catc$ is strict, but may fail when $\catc$ is not strict.
Similarly, we used the property $\alpha \otimes \id_I = \alpha$, which holds in
a strict category but may fail to do so in a nonstrict one, to prove that $1^\times$
is a unit element for the multiplication $\times$. \\

We shall refer to the semiring  $Q^c =(Q_S (\catc), +, \cdot, 0, 1)$ as the
\emph{composition semiring} of $\catc$ (with ground semiring $S$), and to 
 $Q^m =(Q_S (\catc), +, \times, 0, 1^\times)$ as the
\emph{monoidal semiring} of $\catc$. \\

Given morphisms
\[ X' \stackrel{\xi'}{\longrightarrow} Z' \stackrel{\eta'}{\longrightarrow} Y',~
  X'' \stackrel{\xi''}{\longrightarrow} Z'' \stackrel{\eta''}{\longrightarrow} Y'', \]
the identity
\begin{equation} \label{equ.ctcistct}
(\eta' \circ \xi')\otimes (\eta'' \circ \xi'') = 
 (\eta' \otimes \eta'')\circ (\xi' \otimes \xi'')
\end{equation}
holds. This shows that the 
\emph{composition-tensor-composition (CTC) set} of a morphism
$\gamma: X\to Y$ in $\catc$,
\[ CTC(\gamma) = \{ (\xi', \xi'', \eta', \eta'') \in \Mor (\catc)^4 ~|~
  (\eta' \circ \xi')\otimes (\eta'' \circ \xi'') = \gamma \} \]
is a subset of the \emph{tensor-composition-tensor (TCT) set} of $\gamma$,
\[ TCT (\gamma) = \{ (\xi', \xi'', \eta', \eta'') \in \Mor (\catc)^4 ~|~
   (\eta' \otimes \eta'')\circ (\xi' \otimes \xi'') = \gamma \}, \]
since the equation $(\eta' \circ \xi')\otimes (\eta'' \circ \xi'') = \gamma$
implies that $\cod \xi' = \dom \eta'$ and $\cod \xi'' = \dom \eta''$,
so that (\ref{equ.ctcistct}) is applicable. However, knowing only
$(\eta' \otimes \eta'')\circ (\xi' \otimes \xi'') = \gamma$, one can infer
$\cod (\xi' \otimes \xi'') = \dom (\eta' \otimes \eta''),$ but \emph{not}
the individual statements $\cod \xi' = \dom \eta'$ and $\cod \xi'' = \dom \eta''$.
Thus (\ref{equ.ctcistct}) is not necessarily applicable and $TCT (\gamma)$ is
in general strictly larger than $CTC (\gamma)$.
\begin{prop} \label{prop.tctctconeobject}
Let $\catc$ be a strict monoidal category. Then
$TCT(\gamma)=CTC(\gamma)$ for all morphisms $\gamma$ in $\catc$
if and only if $\catc$ is a monoid, i.e. has only one object.
\end{prop}
\begin{proof}
If $\catc$ has only one object, then this object must be the unit object $I$
and $\cod \xi' = I = \dom \eta'$ and $\cod \xi'' =I= \dom \eta''$ for all
$(\xi', \xi'', \eta', \eta'')\in TCT(\gamma)$. Thus $TCT(\gamma)=CTC(\gamma)$
for all morphisms $\gamma$. For the converse direction, let $X$ be any object
of $\catc$. Write $\gamma = \id_X = (\id_X \otimes \id_I)\circ (\id_I \otimes \id_X)$.
Then $(\id_I, \id_X, \id_X, \id_I)\in TCT(\id_X) = CTC (\id_X)$ and hence
$\id_X = (\id_X \circ \id_I)\otimes (\id_I \circ \id_X)$. It follows that
$X=\dom (\id_X)=\cod (\id_I)=I.$
\end{proof}
Let us translate the equivalent statements of the preceding proposition
into a statement about the algebraic structure
of $Q_S (\catc)$.
\begin{prop} \label{prop.abcd}
If $\catc$ is a monoid (i.e. has only one object), then for any elements $a,b,c,d \in Q_S (\catc)$ 
such that $b$ or $c$ maps entirely into the center of $S$, the
multiplicative compatibility relation
\[ (a\times b)\cdot (c\times d) = (a\cdot c)\times (b\cdot d) \]
holds.
\end{prop}
\begin{proof}
On an endomorphism $\gamma: I\to I$ in $\catc$,
\[ ((a\times b)\cdot (c\times d))_{II} (\gamma) =
 \sum_{\eta \circ \xi = \gamma} (c\times d)_{II} (\eta)\cdot
    (a\times b)_{II} (\xi) \hspace{4cm}\]
\begin{eqnarray*}
 & = & \sum_{\eta \circ \xi = \gamma} \Big\{ \sum_{\eta' \otimes \eta'' = \eta}
  d_{II} (\eta'')\cdot c_{II} (\eta') \Big\} \cdot
 \Big\{ \sum_{\xi' \otimes \xi'' = \xi}
  b_{II} (\xi'')\cdot a_{II} (\xi') \Big\} \\
& = & \sum_{(\xi', \xi'', \eta', \eta'')\in TCT (\gamma)}
  d_{II} (\eta'')\cdot c_{II} (\eta') \cdot
  b_{II} (\xi'')\cdot a_{II} (\xi') \\
& = & \sum_{(\xi', \xi'', \eta', \eta'')\in CTC (\gamma)}
  d_{II} (\eta'')\cdot b_{II} (\xi'') \cdot
  c_{II} (\eta')\cdot a_{II} (\xi') \\
& = & \sum_{\gamma' \otimes \gamma'' = \gamma} \Big\{
 \sum_{\eta'' \circ \xi'' = \gamma''}  d_{II} (\eta'')\cdot
  b_{II} (\xi'') \Big\} \cdot \Big\{
 \sum_{\eta' \circ \xi' = \gamma'}  c_{II} (\eta')\cdot
  a_{II} (\xi') \Big\} \\
& = & \sum_{\gamma' \otimes \gamma'' = \gamma}
  (b\cdot d)_{II} (\gamma'')\cdot (a\cdot c)_{II} (\gamma') \\
& = & ((a\cdot c)\times (b\cdot d))_{II} (\gamma).
\end{eqnarray*}
In this calculation, we were able to commute $c_{II}(\eta')$
and $b_{II}(\xi'')$ because one of these two commutes with every
element of $S$.
\end{proof}
When $\xi'$ and $\xi''$ are fixed, we shall also write
\[ CTC(\gamma; \xi', \xi'') = \{ (\eta', \eta'') \in \Mor (\catc)^2 ~|~
  (\eta' \circ \xi')\otimes (\eta'' \circ \xi'') = \gamma \}, \]
\[ TCT (\gamma; \xi', \xi'') = \{ (\eta', \eta'') \in \Mor (\catc)^2 ~|~
   (\eta' \otimes \eta'')\circ (\xi' \otimes \xi'') = \gamma \}. \]
For certain applications, let us record the following simple observation.
\begin{lemma} \label{lem.commmonoidstrictmoncat}
A commutative monoid $(C,\cdot,1_C)$ determines a small strict monoidal
category $\catc = \catc (C)$ by
\[ \operatorname{Ob} \catc = \{ I \},~ \operatorname{End}_\catc (I)=C,~ I\otimes I =I,~
 \alpha \circ \beta = \alpha \cdot \beta = \alpha \otimes \beta \]
for all $\alpha, \beta \in C$.
\end{lemma}
\begin{proof}
The composition law is associative and we have $\id_I =1_C$. The tensor product $\otimes$
is strictly associative and
\[ \alpha \otimes \id_I = \alpha \cdot 1_C = \alpha = 1_C \cdot \alpha = \id_I \otimes \alpha. \]
Furthermore, $\otimes$ is a bifunctor, as
\[ (\alpha' \otimes \beta')\circ (\alpha \otimes \beta)=
  (\alpha' \cdot \beta')\cdot (\alpha \cdot \beta)=
 (\alpha' \cdot \alpha)\cdot (\beta' \cdot \beta) =
 (\alpha' \circ \alpha)\otimes (\beta' \circ \beta) \]
and $\id_I \otimes \id_I = \id_I = \id_{I\otimes I}$.
\end{proof}

\section{Fields and Category-Valued Actions}
\label{sec.fieldsactions}

The two ingredients needed to form a field theory are the fields and an action
functional on these fields. Both have to satisfy certain natural axioms. Regarding the fields,
our axioms will not deviate essentially from the usual axioms as employed in \cite{kirktqft},  
\cite{freedlecnotes}, for example. We emphasize, however, that our axiomatization
assigns fields only in codimensions $0$ and $1$, and not in higher codimensions.
Closed $n$-dimensional
topological manifolds will be denoted by $M, N, P, M_0,$ etc. Our manifolds need not be orientable.
The empty set $\varnothing$ is a manifold of any dimension. The symbol $\sqcup$ denotes
the ordinary ordered disjoint union of manifolds. It is not commutative and not associative
(see the Remark on p. 72 of \cite{masbaumrourke}),
but there are obvious canonical homeomorphisms
$M\sqcup N \cong N \sqcup M,$ $(M\sqcup N)\sqcup P \cong M\sqcup (N\sqcup P),$
$M\sqcup \varnothing \cong M \cong \varnothing \sqcup M$.
Note that the triple union $M\sqcup N \sqcup P$ is well-defined and canonically
homeomorphic to both $(M\sqcup N)\sqcup P$ and $M \sqcup (N\sqcup P)$.
An $(n+1)$-dimensional \emph{bordism} (sometimes also called spacetime in the literature) 
is a triple $(W,M,N)$, where $W$ is a compact
$(n+1)$-dimensional topological manifold with boundary $\partial W = M\sqcup N$.
The closed $n$-manifold $M$ is called the \emph{incoming boundary} of $W$ and $N$
is called the \emph{outgoing boundary} of $W$. (Strictly speaking, recording the outgoing
boundary is redundant since $N=\partial W - M$; nevertheless we find it convenient to include
$N$ in the notation as well.) We shall also say that $W$ is a bordism from $M$ to $N$.
Setting $\partial W^{\operatorname{in}} =M,$ $\partial W^{\operatorname{out}}=N,$
we may simply write $W$ for the bordism $(W, \partial W^{\operatorname{in}},
\partial W^{\operatorname{out}})$. For example, the cylinder $M\times [0,1]$ on a 
connected $M$ gives rise to three distinct bordisms, namely
$(M\times [0,1], M\times \{ 0,1 \}, \varnothing),$
$(M\times [0,1], M\times 0, M\times 1)$ and
$(M\times [0,1], \varnothing, M\times \{ 0,1 \})$.
The operation disjoint union is defined on bordisms by
\[ (W,M,N) \sqcup (W',M',N') = (W\sqcup W', M\sqcup M', N\sqcup N'). \]
If the outgoing boundary $N$ of $W$ is the incoming boundary of a bordism $W'$,
then we may glue along $N$ to obtain the bordism
\[ (W,M,N)\cup_N (W',N,P) = (W\sqcup_N W', M,P). \]
A \emph{homeomorphism} $\phi: (W,M,N)\to (W',M',N')$ \emph{of bordisms}
is a homeomorphism $W \to W'$, which preserves incoming boundaries and outgoing
boundaries, $\phi (M)=M',$ $\phi (N)=N'$.
The bordism $(W_0, M_0, N_0)$ is a \emph{subbordism} of $(W,M,N)$ if $W_0$
is a codimension $0$ submanifold of $W$ and the following two conditions are
satisfied: For every connected component $C$ of $M_0$ either $C\cap \partial W=\varnothing$
or $C\subset M$, and
for every connected component $C$ of $N_0$ either $C\cap \partial W=\varnothing$
or $C\subset N$. For instance, $(W,M,N)$ is a subbordism of
$(W,M,N)\sqcup (W',M',N')$ and it is a subbordism of
$(W,M,N)\cup_N (W',N,P)$.

\begin{defn} \label{def.fields}
A \emph{system $\Fa$ of fields} assigns to each $(n+1)$-dimensional bordism $W$ a
set $\Fa (W)$ (whose elements are called the \emph{fields} on $W$) and to every
closed $n$-manifold $M$ a set $\Fa (M)$ such that $\Fa (\varnothing)$ is a set with
one element and the following axioms are satisfied: \\

\noindent (FRES) \emph{Restrictions}: If $W_0 \subset W$ is a subbordism,
then there is a restriction map $\Fa (W)\to \Fa (W_0)$. If $M_0 \subset
M$ is a codimension $0$ submanifold, then there is a restriction map
$\Fa (M)\to \Fa (M_0)$. If $M\subset W$ is a closed (as a manifold) codimension $1$ submanifold, then
there is a restriction map $\Fa (W)\to \Fa (M)$. If $f\in \Fa (W)$ is a field, we will
write $f|_M$ for its restriction to $M$, and similarly for the other types of restriction.
All these restriction maps are required to commute with each other in the obvious way,
e.g. for $M_0 \subset M\subset W,$ the map $\Fa (W)\to \Fa (M_0)$ is the composition
$\Fa (W)\to \Fa (M)\to \Fa (M_0)$. Let $M$ be a closed (as a manifold) codimension $0$ submanifold
of $\partial W$. A given field $f\in \Fa (M)$ may be
imposed as a boundary condition by setting
\[ \Fa (W,f) = \{ F\in \Fa (W) ~|~ F|_{M} = f \}. \]
If $W$ is a bordism from $M$ to $N$ and $f\in \Fa (M),$ $g\in \Fa (N),$ we shall also
use the notation $\Fa (W,f,g) \subset \Fa (W)$ for the set of all fields on $W$ which
restrict to $f$ on the incoming boundary $M$ and to $g$ on the outgoing boundary $N$. \\

\noindent (FHOMEO) \emph{Action of homeomorphisms}: A homeomorphism
$\phi: W\to W'$ of bordisms induces contravariantly a bijection $\phi^\ast: \Fa (W')\to \Fa (W)$ such that
$(\id_W)^\ast = \id_{\Fa (W)}$ and $(\psi \circ \phi)^\ast = \phi^\ast \circ \psi^\ast$
for a homeomorphism $\psi: W' \to W''$. Similarly for $n$-dimensional homeomorphisms
$M\to N$. These induced maps are required to commute with the restriction maps of
(FRES). For example, if $M\subset W$ and $M' \subset W'$ are codimension $1$
submanifolds and $\phi: W\to W'$ restricts to a homeomorphism 
$\phi|:M\to M',$ then the diagram
\[ \xymatrix{
\Fa (W') \ar[r]^{\phi^\ast} \ar[d]_{\operatorname{res}} & \Fa (W) 
  \ar[d]^{\operatorname{res}} \\
\Fa (M') \ar[r]^{(\phi|)^\ast} & \Fa (M)
} \]
is to commute. \\

\noindent (FDISJ) \emph{Disjoint Unions}: The product of restrictions
\[ \Fa (W\sqcup W') \longrightarrow \Fa (W)\times \Fa (W') \]
is a bijection, that is, a field on the disjoint union $W\sqcup W'$ is uniquely
determined by its restrictions to $W$ and $W'$, and a field on $W$ and a field on $W'$
together give rise to a field on $W\sqcup W'$. 
Similarly, $\Fa (M\sqcup N) \longrightarrow \Fa (M)\times \Fa (N)$ 
must be a bijection in dimension $n$.\\

\noindent (FGLUE) \emph{Gluing}: Let $W'$ be a bordism from $M$ to $N$ and
let $W''$ be a bordism from $N$ to $P$. Let $W = W' \cup_N W''$ be the bordism
from $M$ to $P$ obtained by gluing $W'$ and $W''$ along $N$. Let
$\Fa (W',W'')$ be the pullback $\Fa (W')\times_{\Fa (N)} \Fa (W'')$ fitting into
a cartesian square
\[ \xymatrix{
\Fa (W',W'') \ar[r] \ar[d] & \Fa (W') \ar[d]^{\operatorname{res}} \\
\Fa (W'') \ar[r]^{\operatorname{res}} & \Fa (N). 
} \]
Since
\[ \xymatrix{
\Fa (W) \ar[r]^{\operatorname{res}} \ar[d]_{\operatorname{res}} 
  & \Fa (W') \ar[d]^{\operatorname{res}} \\
\Fa (W'') \ar[r]^{\operatorname{res}} & \Fa (N)
} \]
commutes, there exists a unique map $\Fa (W)\to \Fa (W',W'')$ such that
$\Fa (W)\to \Fa (W',W'') \to \Fa (W')$ is the restriction to $W'$ and
$\Fa (W)\to \Fa (W',W'') \to \Fa (W'')$ is the restriction to $W''$.
We require that $\Fa (W)\to \Fa (W',W'')$ is a bijection.
\end{defn}

Note that the convention for $\Fa (\varnothing)$ is consistent with axioms
(FHOMEO) and (FDISJ): The homeomorphism $\phi: M\sqcup \varnothing
\stackrel{\cong}{\longrightarrow} M$ induces a bijection
$\phi^\ast: \Fa (M) \stackrel{\cong}{\longrightarrow} \Fa (M\sqcup \varnothing)$.
Thus by (FDISJ),
\[ \Fa (M) \cong \Fa (M\sqcup \varnothing) \cong \Fa (M)\times \Fa (\varnothing)
= \Fa (M) \times \{ \pt \}. \]
For bordisms $W'$ with empty outgoing boundary and bordisms $W''$ with empty
incoming boundary, axiom (FDISJ) follows from (FGLUE) by taking $N=\varnothing$.
For then $W' \cup_{\varnothing} W'' = W' \sqcup W''$ and 
$\Fa (W',W'') = \Fa (W')\times \Fa (W'')$. However, since not all bordism
are of this type, the $(n+1)$-dimensional part of axiom (FDISJ) is not redundant.
\begin{lemma}  \label{lem.fdisjbndryconds}
Let $\Fa$ be a system of fields and
$M\subset \partial W,$ $M'\subset \partial W'$ closed codimension 0 submanifolds.
Then axiom (FDISJ) continues to hold in the presence
of boundary conditions.
More precisely: If $f\in \Fa (M\sqcup M')$ is a field, then the bijection
$\Fa (W\sqcup W') \to \Fa (W)\times \Fa (W')$ restricts to a bijection
\[ \Fa (W\sqcup W', f) \longrightarrow \Fa (W, f|_M)\times \Fa (W', f_{M'}). \]
\end{lemma}
\begin{proof}
The bijection $\sigma: \Fa (W\sqcup W')\to \Fa (W)\times \Fa (W')$ is given by
$\sigma (G)= (G|_W, G|_{W'})$. To show that it restricts as claimed,
let $G\in \Fa (W\sqcup W')$ be a field with $G|_{M\sqcup M'} =f$.
Using the diagram of restrictions
\begin{equation} \label{equ.wwprimem}
\xymatrix{
\Fa (W\sqcup W') \ar[r] \ar[d] & \Fa (W) \ar[d] \\
\Fa (M\sqcup M') \ar[r] & \Fa (M),
} \end{equation}
which commutes by axiom (FRES), (and using also the analogous diagram for $M'$), we have
\[ (G|_W)|_M = (G|_{M\sqcup M'})|_M = f|_M,~
 (G|_{W'})|_{M'} = (G|_{M\sqcup M'})|_{M'} = f|_{M'}. \]
Thus $\sigma (G) \in \Fa (W,f|_M)\times \Fa (W',f|_{M'})$ and $\sigma$ restricts
preserving boundary conditions. \\

This restriction is injective as the restriction of the injective map $\sigma$.
To show that the restriction is surjective, let $F\in \Fa (W)$ and $F' \in \Fa (W')$
be fields with $F|_M = f|_M$ and $F'|_{M'} = f|_{M'}$. Since $\sigma$ is surjective,
there exists a field $G\in \Fa (W\sqcup W')$ such that $G|_W =F$ and $G|_{W'} =F'$.
Using again diagram (\ref{equ.wwprimem}), we find
\[ ((G|_{M\sqcup M'})|_M, (G|_{M\sqcup M'})|_{M'}) =
  ((G|_W)|_M, (G|_{W'})|_{M'}) = (F|_M, F'|_{M'}) = (f|_M, f|_{M'}). \]
Since
$\Fa (M\sqcup M') \rightarrow \Fa (M)\times \Fa (M')$
is a bijection by axiom (FDISJ), we conclude that $G|_{M\sqcup M'}=f$, that is,
$G \in \Fa (W\sqcup W',f)$.
\end{proof}

\begin{lemma}  \label{lem.fgluebndryconds}
Let $\Fa$ be a system of fields,
let $W'$ be a bordism from $M$ to $N$ and
let $W''$ be a bordism from $N$ to $P$. Let $W = W' \cup_N W''$ be the bordism
from $M$ to $P$ obtained by gluing $W'$ and $W''$ along $N$. 
Then axiom (FGLUE) continues to hold in the presence of boundary conditions.
More precisely: 
Given fields $g'\in \Fa (M),$ $g''\in \Fa (P),$ let
$\Fa (W',W'', g', g'')$ be the pullback
$\Fa (W',g') \times_{\Fa (N)} \Fa (W'',g'')$. Then,
given a field $f\in \Fa (\partial W)$, the unique map $\rho$ such that
\[ \xymatrix{
\Fa (W,f) \ar@/^1pc/[rrd] \ar[rd]^{\rho} \ar@/_1pc/[ddr] & & \\
& \Fa (W',W'',f|_M, f|_P) \ar[r] \ar[d] & \Fa (W',f|_M) \ar[d] \\
& \Fa (W'',f|_P) \ar[r] & \Fa (N)
} \]
commutes is a bijection.
\end{lemma}
\begin{proof}
The bijection $\sigma: \Fa (W)\to \Fa (W',W'')$ is given by
$\sigma (G)= (G|_{W'}, G|_{W''})$. 
Let $G\in \Fa (W)$ be a field with $G|_{\partial W} =f$.
Using the diagram of restrictions
\begin{equation} \label{equ.wwprimempm}
\xymatrix{
\Fa (W) \ar[r] \ar[d] & \Fa (W') \ar[d] \\
\Fa (M\sqcup P) \ar[r] & \Fa (M),
} \end{equation}
which commutes by axiom (FRES), (and using also the analogous diagram for $P$), we have
\[ (G|_{W'})|_M = (G|_{M\sqcup P})|_M = f|_M,~
 (G|_{W''})|_{P} = (G|_{M\sqcup P})|_P = f|_P. \]
Since in addition $(G|_{W'})|_N = (G|_{W''})|_N$, we conclude that 
$\sigma (G) \in \Fa (W', W'', f|_M, f|_P)$ and thus $\rho$ is the restriction
of $\sigma$ to $\Fa (W,f)\subset \Fa (W)$. 

This restriction $\rho$ is injective as the restriction of the injective map $\sigma$.
To show that $\rho$ is surjective, let 
\[ (F',F'') \in \Fa (W', W'', f|_M, f|_P) \subset \Fa (W', W''). \]
Since $\sigma$ is surjective,
there exists a field $G\in \Fa (W)$ such that $G|_{W'} =F'$ and $G|_{W''} =F''$.
Using again diagram (\ref{equ.wwprimempm}), we find
\[ ((G|_{\partial W})|_M, (G|_{\partial W})|_P) =
  ((G|_{W'})|_M, (G|_{W''})|_P) = (F'|_M, F''|_P) = (f|_M, f|_P). \]
Since
$\Fa (\partial W)=\Fa (M\sqcup P) \to \Fa (M)\times \Fa (P)$
is a bijection by axiom (FDISJ), we conclude that $G|_{\partial W}=f$, that is,
$G \in \Fa (W,f)$.
\end{proof}

\begin{remark}
As with all axiomatic systems, the above axioms may need to be
appropriately adapted to concrete situations. For instance, the manifolds
to be considered may be decorated with additional structure,
for instance orientations. If the fields interact with the additional structure, then
the restrictions in axiom (FRES) will in general only be available for
inclusions that preserve the additional structure. In (FHOMEO), only
those homeomorphisms that preserve the structure will act on the fields.
For example, in an equivariant context, one may wish to impose (FHOMEO)
only on equivariant homeomorphisms.
In (FDISJ), the disjoint union will be assumed to be equipped with
the structure compatible to the structures on the component manifolds.
Analogous provisos apply to (FGLUE). 
The axioms can be adapted to the category of smooth manifolds
and smooth maps. The main issue there is to arrive at a correct version
of (FGLUE), since gluing two smooth maps that agree on the common
boundary component $N$ only yields a map which is continuous but
usually not smooth. This can be achieved by not only requiring equality
of the function values (as we have done in (FGLUE)), but also equality
of all higher partial derivatives. Another possibility is to require the
functions to be equal on collar neighborhoods of $N$ and then to glue the collars.
\end{remark}

\begin{example}
Let $B$ be a fixed space. Taking $\Fa (W)$ and $\Fa (M)$ to be the set of all
continuous maps $W\to B$, $M\to B$, respectively, and using the ordinary restrictions
of such maps to subspaces in (FRES), one obtains a system $\Fa$ of fields in the
sense of Definition \ref{def.fields}. The action of homeomorphisms on fields
is given by composition of fields with a given homeomorphism.
In practice, $B$ is often the classifying space
$BG$ of some topological group $G$ (which may be discrete), so that fields in that 
case have the interpretation of
principal $G$-bundles over $W$ and $M$. This has been considered for finite
groups $G$ in work of Freed and Quinn \cite{quinnfreed}, 
\cite{quinnlectaxiomtqft}, \cite{freedlecnotes}, see also \cite{dijkwitten}.
Fields of this kind are also used in the construction of the twisted
signature TFT given in Section \ref{ssec.twistedsigntft}.
Let us note in passing that taking manifolds endowed with maps to a fixed
space $B$ as \emph{objects} (and not as fields on objects), one arrives
at the notion of a \emph{homotopy quantum field theory} (HQFT),
\cite{turaevhqft}. Taking $B$ to be a point, HQFTs are seen to be generalizations
of TQFTs.
In the smooth category, one may fix $B$ to be a smooth 
manifold and consider $\Fa (W)=C^\infty (W,B)$, the space of smooth maps $W\to B$. 
This is roughly the setting for Chern-Simons theory.
\end{example}

\begin{remark}
Walker's axiomatization of fields, \cite{walkertqfts}, differs from ours
(and from \cite{kirktqft}, \cite{freedlecnotes}) in that he does not allow for
codimension $0$ restrictions and he
requires the existence
of an injection $\Fa (W', W'') \hookrightarrow \Fa (W)$
in the context of the gluing axiom.
However, he does assume any field on $W$ to be close to a field in the image of
the injection in the sense that the field on $W$ can be moved by a homeomorphism, which is
isotopic to the identity and supported in a small neighborhood of $N\subset W$,
to a field coming from $\Fa (W', W'')$ under gluing. 
Walker does not require a bijection because he wants to allow for the following
application: Fields could be embedded submanifolds, or even more intricate
``designs'' on manifolds, which are transverse to the boundary. Given any
submanifold of $W$, there is no way of guaranteeing that it is transverse to
$N$ (though it can be made so by an arbitrarily small movement).
Thus there is no restriction map from such fields on $W$ to fields on $W'$, and
not every field on $W$ comes from one on $W'$ and one on $W''$ by gluing. 
\end{remark}

Given a system $\Fa$ of fields, the second ingredient necessary for a field theory
is an action functional defined on $\Fa (W)$ for bordisms $W$.
In classical quantum field theory, the action is usually a system of real-valued functions
$S_W: \Fa (W)\to \real$ such that the additivity axiom
\begin{equation} \label{equ.sadddisj}
S_{W\sqcup W'} (f) = S_W (f|_W) + S_{W'} (f|_{W'}),~
f\in \Fa (W\sqcup W'),
\end{equation}
is satisfied for disjoint unions, and the additivity axiom
\begin{equation} \label{equ.saddglue}
S_W (f) = S_{W'} (f|_{W'}) + S_{W''} (f|_{W''}),~
f\in \Fa (W),
\end{equation}
is satisfied for $W=W' \cup_N W''$, the result of gluing a bordism $W'$ with
outgoing boundary $N$ to a bordism $W''$ with incoming boundary $N$.
Moreover, the action should be topologically invariant: if $\phi: W \to W'$ is
a homeomorphism, then for any field $f\in \Fa (W'),$ one requires that
under the bijection $\phi^\ast: \Fa (W')\to \Fa (W)$ of (FHOMEO), 
the action is preserved,
\begin{equation} \label{equ.snatural} 
S_W (\phi^\ast f) = S_{W'} (f). 
\end{equation}
Sometimes, for example in Chern-Simons theory, the action is only well-defined
up to an integer, that is, takes values in $\real /\intg$. Thus it is better to
exponentiate and consider the
complex-valued function $T_W = e^{2\pi iS_W}:
\Fa (W)\to \cplx$ whose image lies in the unit circle. The above two additivity
axioms are then transformed into the multiplicativity axioms
\begin{equation} \label{equ.smultdisj}
T_{W\sqcup W'} (f) = T_W (f|_W) \cdot T_{W'} (f|_{W'}),
\end{equation}
and
\begin{equation} \label{equ.smultglue}
T_W (f) = T_{W'} (f|_{W'}) \cdot T_{W''} (f|_{W''}).
\end{equation}
These axioms express that the action should be local to a certain extent.
\begin{example}
In the smooth oriented category, for the system of fields $\Fa (W) = C^\infty (W,B)$, $B$ a fixed smooth
manifold, fix a differential $(n+1)$-form $\omega \in \Omega^{n+1} (B)$ on $B$.
Setting 
\[ S_W (f) = \int_W f^\ast \omega,\]
the axioms (\ref{equ.sadddisj})
and (\ref{equ.saddglue}) are satisfied.
For an orientation preserving diffeomorphism $\phi:W\to W'$, (\ref{equ.snatural}) holds.
If $\omega$ is closed and $W$ has no boundary, then $S_W (f)$ only depends on the
homotopy class of $f$. For if $f,g:W\to B$ are homotopic, then a homotopy between
them gives rise to a homotopy operator $h: \Omega^\ast (B)\to \Omega^{\ast -1} (W),$
$dh + hd = f^\ast - g^\ast,$ so that for closed $\omega$ one has
$f^\ast (\omega) - g^\ast (\omega) = dh(\omega)$. By Stokes theorem,
\[ \int f^\ast \omega - \int g^\ast \omega = \int dh(\omega)=0. \]
The Chern-Simons action is roughly of this type.
\end{example}
For a number of purposes, remembering only a real number for a given field is
too restrictive and it is desirable to retain more information about the field.
The present paper thus introduces category valued actions. We will in fact directly
axiomatize the analog of the exponential $T$ of an action.
Let $(\catc, \otimes, I)$ be a strict monoidal category. (The strictness is not a very
serious assumption, as a well-known process turns any monoidal category into
a monoidally equivalent strict one, see \cite{kassel}). 
Since in a monoidal context of bordisms, disjoint union corresponds to the
tensor product, while gluing of bordisms corresponds to the composition of
morphisms, it is natural to modify the classical axioms
(\ref{equ.smultdisj}) and (\ref{equ.smultglue}) as follows:

\begin{defn} \label{def.actions}
Given a system $\Fa$ of fields, a \emph{system $\TT$ of $\catc$-valued action
exponentials} consists of functions $\TT_W: \Fa (W)\to \Mor (\catc)$, for all
bordisms $W$, such that for the empty manifold,
$\TT_{\varnothing} (p) =\id_I,$ where $p$ is the unique element of
$\Fa (\varnothing),$ and the following three axioms are
satisfied: \\

\noindent (TDISJ) If $W\sqcup W'$ is the ordered disjoint union of two
bordisms $W,W'$, then 
\[ \TT_{W\sqcup W'} (f) = \TT_W (f|_W) \otimes \TT_{W'} (f|_{W'}) \]
for all $f\in \Fa (W\sqcup W'),$ \\

\noindent (TGLUE) If $W=W'\cup_N W''$ is obtained by gluing a bordism
$W'$ with outgoing boundary $N$ to a bordism $W''$ with incoming boundary
$N$, then
\[ \TT_W (f) = \TT_{W''} (f|_{W''}) \circ \TT_{W'} (f|_{W'}) \]
for all $f\in \Fa (W)$, and \\

\noindent (THOMEO) If $\phi: W \to W'$ is
a homeomorphism of bordisms, then for any field $f\in \Fa (W'),$ we require that
under the bijection $\phi^\ast: \Fa (W')\to \Fa (W)$ of (FHOMEO), 
\[ \TT_W (\phi^\ast f) = \TT_{W'} (f). \]
\end{defn}

There is a simple test that shows that both the tensor product and the composition
product of $\catc$ must enter into the axioms for a category-valued action, underlining
the correctness of the above definition:
\begin{example}
\emph{(The tautological action.)}
Let $\catc = \mathbf{Bord}(n+1)^{\operatorname{str}}$ a strict version of the
$(n+1)$-dimensional bordism category. A bordism $W$ defines a morphism
$[W] \in \Mor (\mathbf{Bord}(n+1)^{\operatorname{str}})$. Then one has the
tautological action exponential $\TT_W (f)=[W]$. It satisfies
\[ \TT_{W\sqcup W'} (f) = [W\sqcup W'] = [W]\otimes [W'] =
  \TT_W (f|) \otimes \TT_{W'} (f|) \]
and 
\[ \TT_{W'\cup_N W''} (f)= [W' \cup_N W''] = [W''] \circ [W'] =
  \TT_{W''} (f|)\circ \TT_{W'} (f|). \]
This forces the above axioms (TDISJ) and (TGLUE).
\end{example}

If the manifolds $W$ are equipped with some extra structure, then one will
in practice usually modify (THOMEO) to apply only to those homeomorphisms
that preserve the extra structure. For example, if the $W$ are oriented, one
will usually require $\phi$ to preserve orientations.
Note that under the canonical homeomorphism $\phi: W\cong W\sqcup \varnothing,$
\[ \TT_W (\phi^\ast f) = \TT_{W\sqcup \varnothing} (f) =
  \TT_W (f|_W)\otimes \TT_{\varnothing} (f|_{\varnothing}) =
 \TT_W (f|_W)\otimes \id_I = \TT_W (f|_W), \]
using axioms (TDISJ) and (THOMEO). If $W = M\times I$ is the cylindrical bordism
from $M$ to $M$,
then we do \emph{not} require that $\TT_{M\times I}$ is an identity
morphism.
If $W'$ has empty outgoing boundary and $W''$ empty incoming boundary, then
we can ``glue'' along the empty set and get $W' \cup_{\varnothing} W'' =
W' \sqcup W''$. Thus (TGLUE) and (TDISJ) apply simultaneously and yield
\[  \TT_{W'} (f) \otimes \TT_{W''} (g) =
   \TT_{W''} (g) \circ \TT_{W'} (f). \]
In particular, the domain of any $\TT_{W'}(f)$ must be the tensor product
of the domain of $\TT_{W'}(f)$ with the domain of any $\TT_{W''}(g)$.
In practice, this usually means that the domains of all $\TT_{W''}(g)$ are
the unit object $I$ of $\catc$. But the domain of $\TT_{W''}(g)$ equals
the codomain of $\TT_{W'}(f)$. So in practice, the codomains of the
$\TT_{W'}(f)$ are usually $I$ as well. We would like to emphasize again
that these remarks apply only to bordisms whose incoming or outgoing
boundary is empty.

Let $W'_1$ be a bordism with empty outgoing boundary, 
$W'_2$ a bordism from $\varnothing$ to $N$, $W''_2$ a bordism
from $N$ to $\varnothing$ and let $W''_1$ be a bordism with empty
incoming boundary. Then we can form the bordism
\[ W = (W'_2 \sqcup W'_1) \cup_N (W''_2 \sqcup W''_1), \]
which we can also think of as
\[ W = (W'_2 \cup_N W''_2) \sqcup (W'_1 \cup_{\varnothing} W''_1). \]
These two representations of $W$ allow us to calculate the action associated with
$W$ in two different ways:
\[ \TT_W (f) = \TT_{W''_2 \sqcup W''_1}(f|) \circ \TT_{W'_2 \sqcup W'_1}(f|) =
  (\TT_{W''_2}(f|) \otimes \TT_{W''_1} (f|))\circ
  (\TT_{W'_2}(f|) \otimes \TT_{W'_1}(f|)) \]
and
\[ \TT_W (f) = \TT_{W'_2 \cup_N W''_2}(f|) \otimes \TT_{W'_1 \cup_{\varnothing} W''_1}(f|) =
  (\TT_{W''_2}(f|) \circ \TT_{W'_2}(f|))\otimes
   (\TT_{W''_1}(f|) \circ \TT_{W'_1}(f|)). \]
This implies the equation
\[  (\TT_{W''_2}(f|) \circ \TT_{W'_2}(f|))\otimes
   (\TT_{W''_1}(f|) \circ \TT_{W'_1}(f|)) =
   (\TT_{W''_2}(f|) \otimes \TT_{W''_1} (f|))\circ
  (\TT_{W'_2}(f|) \otimes \TT_{W'_1}(f|)), \]
which indeed holds automatically in any monoidal category $\catc$. 

The result of gluing two copies $W' = M\times [0,1]$ and $W'' = M\times [0,1]$
of the unit cylinder on $M$, identifying $M\times 1 \subset W'$ with
$M\times 0 \subset W''$, is $W=M\times [0,2].$ For
$(F',F'')\in \Fa (W', W''),$ axioms (TGLUE) and (THOMEO) imply the formula
\begin{equation} \label{equ.ttoncyl}
\TT_{M\times [0,1]} (2^\ast \sigma^{-1} (F',F''))=\TT_{M\times [0,1]} (F'')
  \circ \TT_{M\times [0,1]} (F'), 
\end{equation}
where $\sigma: \Fa (W)\to \Fa (W',W'')$ is the bijection of axiom (FGLUE)
and $2:M\times [0,1]\to M\times [0,2]$ is the stretching homeomorphism
$2(x,t)=(x,2t),$ $x\in M,$ $t\in [0,1]$.

\begin{remark}
The classical axioms (\ref{equ.smultdisj}) and (\ref{equ.smultglue})
do fit into the framework of Definition \ref{def.actions}.
From the perspective of this definition, the fact that in both
(\ref{equ.smultdisj}) and (\ref{equ.smultglue}) the ordinary multiplication
of complex numbers appears is just a reflection of the coincidence that
under the standard isomorphism $\cplx \otimes \cplx \cong \cplx$, the
tensor product of two $\cplx$-linear maps $\alpha, \beta: \cplx \to \cplx$
is given by multiplication $\alpha \cdot \beta$, and the composition of
two linear maps  $\alpha, \beta: \cplx \to \cplx$ \emph{also} happens to be
given by multiplication, $\alpha \cdot \beta$.
More precisely, let $\widehat{\cplx}$ be the category which has $\cplx$ as its
single object and $\Hom_{\widehat{\cplx}} (\cplx, \cplx) =
\{ \alpha: \cplx \to \cplx ~|~ \alpha \text{ is $\cplx$-linear} \}$.
Such an $\alpha$ is of course determined by $\alpha (1),$ whence
$\Hom_{\widehat{\cplx}} (\cplx, \cplx) \cong \cplx.$ In $\widehat{\cplx}$,
define $\cplx \otimes \cplx := \cplx$ and define $\alpha \otimes \beta:
\cplx \otimes \cplx = \cplx \to \cplx = \cplx \otimes \cplx$ by
$(\alpha \otimes \beta)(1) = \alpha (1)\cdot \beta (1)$. 
Taking $I=\cplx,$ $(\widehat{\cplx}, \otimes, I)$ is a strict monoidal
category. If a classical action exponential $T_W: \Fa (W)\to \cplx$ is
interpreted as a $\widehat{\cplx}$-valued action exponential
$\TT_W: \Fa (W)\to \cplx \cong \Mor (\widehat{\cplx}),$ then
(TDISJ) translates to (\ref{equ.smultdisj}) and (TGLUE) translates
to (\ref{equ.smultglue}).
\end{remark}

\begin{remark}
\emph{(On cutting.)}
Suppose that $W$ is an oriented bordism and $M\hookrightarrow W$ a
closed oriented codimension $1$ submanifold situated in the interior of $W$.
Let $W^\cut$ be the compact manifold with boundary
$\partial W^\cut = \partial W \sqcup M \sqcup M$ obtained from $W$ by
cutting along $M$. The problem is that, contrary to the operations of
disjoint union and gluing, this construction is \emph{not}
well-defined on bordisms because there is no canonical way to define the
incoming and outgoing boundary of $W^\cut$.
(Should $M\sqcup M$ belong to the incoming or outgoing boundary?
Should one of them belong to the incoming and the other to the outgoing
boundary? If so, which of the two copies is incoming and which outgoing?)
Our field axioms do not provide for an equalizer diagram
$\Fa (W)\to \Fa (W^\cut)\rightrightarrows \Fa (M)$ and our action axioms do \emph{not} stipulate 
\begin{equation} \label{equ.ttcut}
\TT_{W^\cut} (f^\cut) = \TT_W (f),
\end{equation} 
where $f^\cut$ is the
image of $f$ under a putative $\Fa (W)\to \Fa (W^\cut)$.
Such axioms are classically sometimes adopted, for example in \cite{freedlecnotes},
and are strongly motivated by thinking of actions as being given by integrals
of pullbacks of differential forms. In the setting of the present paper,
we wish to think of actions in much more general terms. For instance,
actions might be certain subspaces of $W$ associated to fields. But if $W$
is cut, then these subspaces are also cut and consequently
(\ref{equ.ttcut}) cannot hold.
\end{remark}

The next definition will be used in Section \ref{sec.cylidemproj}, when
we discuss the behavior of a certain projection operator on tensor products of states. 
The projection is associated to the state sum of cylinders.
\begin{defn}
A system $\TT$ of $\catc$-valued action exponentials is called
\emph{cylindrically firm}, if
\[ CTC (\gamma; \TT_{\mti} (F_M), \TT_{\nti} (F_N)) =
   TCT (\gamma; \TT_{\mti} (F_M), \TT_{\nti} (F_N)) \]
for all morphisms $\gamma$ in $\catc$, closed $n$-manifolds $M,N$
and fields $F_M \in \Fa (\mti),$ $F_N \in \Fa (\nti)$.
Here, $\mti$ is to be read as the bordism from $M\times 0$ to $M\times 1$,
similarly for $\nti$.
\end{defn}
For instance, by Proposition \ref{prop.tctctconeobject},
$\TT$ is cylindrically firm if $\catc$ is a monoid.
\begin{prop}
Let $\TT$ be cylindrically firm. Then the codomain of every
$\TT_{M\times [0,1]} (F)$ is the unit object $I$ of $\catc$.
Furthermore, if $F|_{M\times 0} = F|_{M\times 1},$ then
$\TT_{M\times [0,1]}(F)$ is an endomorphism of the unit object.
\end{prop}
\begin{proof}
For $M=\varnothing$, we have $\TT_{M\times [0,1]}(F_M)=\TT_{\varnothing}(p)=\id_I$.
Taking $\gamma = \TT_{N\times [0,1]}(F_N)$, the equation
\[ (\id_X \otimes \id_I)\circ (\TT_{\varnothing} (p) \otimes \TT_{N\times [0,1]}(F_N))=
   \TT_{N\times [0,1]}(F_N) \]
holds, where $X$ is the codomain of $\TT_{N\times [0,1]}(F_N)$.
This places $(\eta', \eta'')=(\id_X, \id_I)$ into
\[ TCT (\TT_{N\times [0,1]}(F_N); \TT_{\varnothing}(p), \TT_{N\times [0,1]}(F_N))
 =  CTC (\TT_{N\times [0,1]}(F_N); \TT_{\varnothing}(p), \TT_{N\times [0,1]}(F_N)). \]
Therefore,
\[ (\id_X \circ \id_I)\otimes (\id_I \circ \TT_{N\times [0,1]}(F_N)) =
  \TT_{N\times [0,1]}(F_N) \]
so that in particular $X=I$. Let $M$ be any closed $n$-manifold.
If $F\in \Fa (M\times [0,1])$ satisfies $F|_{M\times 0} = F|_{M\times 1},$ then
the diagonal element $(F,F)$ lies in the pullback
$\Fa (M\times [0,1], M\times [0,1])$ and thus Equation (\ref{equ.ttoncyl}) shows that
$\dom \TT_{M\times [0,1]}(F) = \cod \TT_{M\times [0,1]}(F) =I$.
\end{proof}
It follows from the proposition that for a cylindrically firm system of action exponentials,
the above $TCT$ and $CTC$ sets can be nonempty only for $\gamma$ that factor
through the unit object.

\section{Quantization}
\label{sec.constft}

We shall define our positive TFT $Z$ in this section. We will specify the state module $Z(M)$ for a closed
$n$-manifold $M$ as well as an element
$Z_W \in Z(\partial W)$, the \emph{Zustandssumme}, for a bordism $W$,
which may have a nonempty boundary $\partial W$.
Neither $M$ nor $W$ have to be oriented; thus $Z$ will be a 
nonunitary theory. In \cite{wittenqftjones}, Witten starts out with
the phase space $\mathcal{M}_0$ of all connections on a trivial $G$-bundle
over $\Sigma \times \real^1$, where $\Sigma$ is a Riemann surface and $G$
a compact simple gauge group. In Section 3 of \emph{loc. cit.}, he carries
out the quantization of Chern-Simons theory on $\Sigma$ in two steps:
First, constraint equations are imposed, which reduce $\mathcal{M}_0$ to
the finite dimensional moduli space $\mathcal{M}$ of flat connections, where two flat connections
are identified if they differ by a gauge transformation. Second, Witten's
quantum Hilbert space (state module) $\mathcal{H}_\Sigma$ is obtained by
taking functions on $\mathcal{M},$ more precisely, global holomorphic sections
of a certain line bundle on $\mathcal{M}$. This provides a model for our construction
the state module $Z(M)$.\\ 

Fix a complete semiring $S$, which will play the role of a ground semiring for the
theory to be constructed. To any given system $\Fa$ of fields, strict monoidal small
category $(\catc, \otimes, I)$, and system $\TT$ of $\catc$-valued action
exponentials, we shall now associate a positive topological field theory $Z$.
In Section \ref{sec.semiringmoncat}, we have seen that $\catc$ determines
two complete semirings: the composition semiring
$Q^c = (Q_S (\catc),+,\cdot,0,1),$ whose multiplication $\cdot$ encodes the
composition law of $\catc$, and the monoidal semiring
$Q^m = (Q_S (\catc),+,\times,0,1^\times),$ whose multiplication $\times$
encodes the monoidal structure on $\catc$, i.e. the tensor functor $\otimes$.
Both of these semirings have the same underlying (complete) additive monoid
$(Q_S (\catc),+,0)$. For a closed $n$-dimensional manifold $M$, we define
its \emph{pre-state module} to be
\[ E(M) = \fun_Q (\Fa (M)), \]
where we have abbreviated $Q = Q_S (\catc)$. 
By Proposition \ref{prop.funsemimod},
$E(M)$ is a two-sided $Q^c$-semialgebra and a two-sided
$Q^m$-semialgebra.
We observe that for the empty manifold, 
\[ E(\varnothing) = \fun_Q (\Fa (\varnothing)) =
\fun_Q (\{ \pt \}) \cong Q. \]
We now impose the constraint equation
\begin{equation} \label{equ.constraint}
 z(\phi^\ast f)=z(f) 
\end{equation}
on pre-states $z\in E(M)$, where $f\in \Fa (M)$ and $\phi:M\to M$ is
a homeomorphism, which is isotopic to the identity.
In other words, call two fields $f,g\in \Fa (M)$ equivalent, if there is a
$\phi$, isotopic to the identity, such that $g=\phi^\ast (f)$.
Let $\ovf (M)$ be the set of equivalence classes. If $M$ is empty,
$\ovf (M)$ consists of a single element.
We define the \emph{state module} (or \emph{quantum Hilbert space}) of $M$
to be
\[ Z(M) = \fun_Q (\ovf (M)) = \{ z: \Fa (M)\to Q ~|~ z(\phi^\ast f)=z(f) \} \subset E(M). \]
Then $Z(M)$ is
a two-sided $Q^c$-semialgebra and a two-sided
$Q^m$-semialgebra. It is in general infinitely generated as a semimodule.
Again, for the empty manifold $Z(\varnothing)\cong Q$.

\begin{remark} \label{rem.wrongstatemodule}
Attempting to define the (pre-)state module of a manifold $M$ with
connected components $M_1,\ldots, M_k$ as the subsemimodule
of $\fun_Q (\Fa (M_1 \sqcup \cdots \sqcup M_k))$ consisting
of all $z$ that can be written as
\[ z(f) = \sum_{i=1}^l z_{i1}(f|_{M_1})\times
  z_{i2}(f|_{M_2})\times \cdots \times z_{ik}(f|_{M_k}) \]
for suitable functions $z_{ij} \in \fun_Q (\Fa (M_j))$, 
leads to an incorrect state module. The reason is that it will
generally not contain the state sum of an $(n+1)$-manifold $W$ with
$\partial W=M$, as Example \ref{exple.statesumdoesnotdecomp} below shows.
See also Remark \ref{rem.wrongfuncttensprod}. 
\end{remark}

\begin{prop} \label{prop.zmndisjtensdecomp}
If $M$ and $N$ are closed $n$-manifolds and $M\sqcup N$ their ordered disjoint
union, then the restrictions to $M$ and $N$ induce an isomorphism
\[ Z(M\sqcup N) \cong Z(M) \hotimes Z(N) \]
of two-sided $Q^c$-semialgebras and of two-sided $Q^m$-semialgebras.
\end{prop}
\begin{proof}
We define a map
\[
\overline{\rho}: \ovf (M\sqcup N) \longrightarrow \ovf (M)\times \ovf (N)
\]
by $[f] \mapsto ([f|_M], [f|_N])$. We need to prove that this is well-defined.
Suppose that $\phi \in \Homeo (M\sqcup N)$ is isotopic to the identity, so that
$[\phi^\ast f]=[f]$.
Then $\phi$ induces the identity map on $\pi_0 (M\sqcup N)$ and thus restricts to
homeomorphisms $\phi|_M \in \Homeo (M),$ $\phi|_N \in \Homeo (N).$
Similarly an isotopy $H: (M\sqcup N)\times [0,1] \to M\sqcup N$ from $\phi$
to the identity restricts to isotopies $H|:M\times [0,1]\to M$, $H|:N\times [0,1]\to N$.
Hence $\phi|_M$ is isotopic to $\id_M$ and $\phi|_N$ is isotopic to $\id_N$.
Using the commutative diagram
\[ \xymatrix{
\Fa (M\sqcup N) \ar[r]^{\phi^\ast}_{\cong} \ar[d]_{\operatorname{res}} & \Fa (M\sqcup N) 
  \ar[d]^{\operatorname{res}} \\
\Fa (M) \ar[r]^{(\phi|_M)^\ast}_{\cong} & \Fa (M)
} \]
provided by axiom (FHOMEO), we arrive at
\[ [(\phi^\ast f)|_M] = [(\phi|_M)^\ast (f|_M)] = [f|_M], \]
and similarly $[(\phi^\ast f)|_N] = [f|_N].$ A map in the other direction
\begin{equation} \label{equ.resovfprodtomn}
\ovf (M)\times \ovf (N) \longrightarrow \ovf (M\sqcup N)
\end{equation}
is given as follows: 
By axiom (FDISJ), the product of restrictions $\rho: \Fa (M\sqcup N)\to
\Fa (M)\times \Fa (N)$ is a bijection. Thus given a pair of fields
$(f,f')\in \Fa (M)\times \Fa (N),$ there exists a unique field $F\in \Fa (M\sqcup N)$
such that $F|_M =f$ and $F|_N =f'$. Then (\ref{equ.resovfprodtomn})
is defined as $([f], [f'])\mapsto [F]$. Again, it must be checked that this is
well-defined. Suppose $\phi \in \Homeo (M), \psi \in \Homeo (N)$ are 
isotopic to $\id_M,$ $\id_N$, respectively, so that $[\phi^\ast f]=[f]$ and
$[\psi^\ast f']=[f']$. Let $G\in \Fa (M\sqcup N)$ be the unique field with
$G|_M = \phi^\ast f$ and $G|_N = \psi^\ast f'$. We have to show that
$[G]=[F]$ on the disjoint union. The disjoint union
$\Phi = \phi \sqcup \psi$ defines a homeomorphism $\Phi: M\sqcup N \to
M\sqcup N$. It is isotopic to the identity via the disjoint union
$H\sqcup H': (M\sqcup N)\times [0,1] \to M\sqcup N$ of isotopies
$H$ from $\phi$ to $\id_M$ and $H'$ from $\psi$ to $\id_N$. 
Since
\[ (\Phi^\ast F)|_M = \phi^\ast (F|_M) = \phi^\ast f= G|_M \]
and similarly $(\Phi^\ast F)|_N = G|_N,$ we have $\Phi^\ast F = G$ by uniqueness.
Therefore, $[G] = [\Phi^\ast F]=[F]$ as required.

The two maps $\overline{\rho}$ and (\ref{equ.resovfprodtomn}) are
inverse to each other, and thus are both bijections.
As discussed in Section \ref{sec.funsemimods},
$\overline{\rho}$ induces a morphism of two-sided semialgebras (over $Q^c$ and over $Q^m$)
\[ \fun (\overline{\rho}): Z(M) \hotimes Z(N) = \fun_Q (\ovf (M)\times \ovf (N))
 \longrightarrow \fun_Q (\ovf (M\sqcup N)) = Z(M\sqcup N) \]
and $\overline{\rho}^{-1}$ induces a morphism of two-sided semialgebras
\[ \fun (\overline{\rho}^{-1}): Z (M\sqcup N) \longrightarrow Z(M) \hotimes Z(N). \]
These two morphisms are inverse to each other by the functoriality of $\fun_Q$.
\end{proof}
Let $W$ be a bordism. A field $F\in \Fa (W)$ determines an
element $T_W (F)\in Q_S (\catc)$ by
\begin{equation} \label{equ.deftwf}
T_W  (F)_{XY} (\gamma) = \begin{cases}
1, & \text{ if } \gamma = \TT_W (F) \\
0, & \text{ otherwise,} \end{cases}
\end{equation}
where $\gamma$ ranges over morphisms $\gamma: X\to Y$ of $\catc$ and
$1$ is the $1$-element of $S$. 
In other words, $T_W (F)= \chi_{\TT_W (F)}$ is the characteristic function 
of $\TT_W (F)$.
Suppose that $W$ is a bordism from $M$ to $N$
and $f\in \Fa (M\sqcup N)=\Fa (\partial W)$. Then we define the state sum
(or partition function) $Z_W$ of $W$ on $f$ by
\[ Z_W (f) = \sum_{F\in \Fa (W,f)} T_W (F) \in Q_S (\catc), \]
using the summation law of the complete monoid $(Q_S (\catc),+,0).$ This is
a well-defined element of $Q_S (\catc)$ that only depends on $W$ and $f$.
\begin{remark}
This sum replaces in our context the notional path integral
\[ \int_{F\in \Fa (W,f)} e^{iS_W (F)} d\mu_W \]
used in classical quantum field theory. As a mathematical object,
this path integral is problematic, since in many situations of interest, an
appropriate measure $\mu_W$ has not been defined or is known not to exist.
The present paper utilizes the notion of completeness in semirings to bypass
measure theoretic questions on spaces of fields. The appearance of the $1$-element
of $S$ in formula (\ref{equ.deftwf}) can be interpreted as a reflection of the fact that
the amplitude of the integrand in the Feynman path integral is always $1$,
$|e^{iS_W (F)}|=1$.
\end{remark}

If $\partial W = \varnothing,$ then $Z_W \in Q \cong Z(\varnothing)$.
If $W$ is empty, then $\Fa (\partial W)=\Fa (\varnothing)=\{ p \}$ is a 
singleton and $\Fa (W,p)=\Fa (W)=\{ p \}$ so that
\[ Z_{\varnothing} (p)=T_{\varnothing} (p) = \chi_{\TT_{\varnothing} (p)} =
  \chi_{\id_I} = 1^\times, \]
the unit element of the semiring $Q^m$. This accords with Atiyah's
requirement (4b) \cite[p. 179]{atiyahtqft}.
Given a morphism $\gamma: X\to Y$ in $\catc$, we have the formula
\[ Z_W (f)_{XY} (\gamma) = \sum_{F\in \Fa (W,f),~ \TT_W (F)=\gamma} 1, \]
which exhibits $Z_W (f)$ as an elaborate counting device: On a morphism
$\gamma$, it ``counts'' for how many fields $F$ on $W$, which restrict to $f$ on the
boundary, $\gamma$ appears as the action exponential of $F$. 
This is a hint that certain kinds of counting functions in number theory and
combinatorics may be expressible as state sums of suitable positive TFTs.
Our theorems, such as e.g. the gluing theorem, will then yield identities for
such functions. In Section \ref{sec.polya}, we illustrate this by deriving
P\'olya's counting theory using positive TFT methods.
In Example \ref{ssec.numbertheory}, we consider arithmetic functions arising
in number theory.
In practice, it
is often most important to know whether $Z_W (f)_{XY} (\gamma)$ is zero or
nonzero. If the ground semiring is the Boolean semiring, then
$Z_W (f)_{XY} (\gamma) =1$ if and only if there exists a field $F\in \Fa (W,f)$
such that $\TT_W (F)=\gamma$. In this case, $Z_W (f)$ admits the interpretation
as a subset $Z_W (f)\subset \Mor (\catc),$ namely
$Z_W (f) = \{ \TT_W (F) ~|~ F\in \Fa (W,f) \}$.
The main results below then show that this system of subsets transforms like a
topological quantum field theory. \\

Returning to the general discussion and letting $f$ vary, we have a pre-state vector
\[ Z_W \in E(\partial W) = E(M\sqcup N)\cong E(M)\hotimes E(N). \]
Let us discuss the topological invariance of the state sum.
A homeomorphism $\phi: M\to N$ induces as follows covariantly a pre-state map
\[ \phi_\ast: E(M) \to E(N), \]
which is an isomorphism of both two-sided $Q^c$-semialgebras and two-sided $Q^m$-semialgebras:
By axiom (FHOMEO), $\phi$ induces a bijection
$\phi^\ast: \Fa (N)\to \Fa (M)$. As shown in Section \ref{sec.funsemimods}, 
this bijection in turn induces a morphism 
\[ \phi_\ast = \fun_Q (\phi^\ast): E(M) = \fun_Q (\Fa (M)) \longrightarrow
  \fun_Q (\Fa (N)) = E(N) \]
of two-sided $Q^c$- and $Q^m$-semialgebras. Since $\phi^\ast$ is a bijection,
$\phi_\ast$ is indeed an isomorphism. 
Moreover, if $\psi: N\to P$ is another homeomorphism, then
$\psi_\ast \circ \phi_\ast = (\psi \circ \phi)_\ast: E(M)\to E(P)$
and $(\id_M)_\ast = \id_{E(M)}: E(M)\to E(M)$, that is, the pre-state module $E(-)$ is a
functor on the category of closed $n$-manifolds and homeomorphisms. 
In particular the group $\operatorname{Homeo}(M)$ of self-homeomorphisms
$M\to M$ acts on $E(M)$.

Let $\phi: W\to W'$ be a homeomorphism of bordisms. Then $\phi$ restricts to a
homeomorphism $\phi_\partial =\phi|: \partial W \to \partial W'$ which induces
an isomorphism $\phi_{\partial \ast}: E(\partial W) \to E(\partial W').$
\begin{thm} \label{thm.statetopinv}
(Topological Invariance.)
If $\phi: W\to W'$ is a homeomorphism of bordisms, then $\phi_{\partial \ast} (Z_W) = Z_{W'}$.
If $W$ and $W'$ are closed, then $\phi_{\partial \ast}=\id:Q\to Q$ and thus
$Z_W = Z_{W'}$.
\end{thm}
\begin{proof}
Let $f\in \Fa (\partial W')$ be a field.
We claim first that the bijection $\phi^\ast: \Fa (W')\to \Fa (W)$ restricts to a 
bijection $\phi^\ast_\rel: \Fa (W',f) \to \Fa (W,\phi^\ast_\partial f).$ To see
this, suppose that $F' \in \Fa (W',f)$. Then the field $\phi^\ast F' \in \Fa (W)$
satisfies
\[ (\phi^\ast F')|_{\partial W} = \phi^\ast_\partial (F'|_{\partial W'})
  = \phi^\ast_\partial (f), \]
where we have used the commutative diagram
\[ \xymatrix{
\Fa (W') \ar[r]^{\phi^\ast}_{\cong} \ar[d]_{\operatorname{res}} & \Fa (W) 
  \ar[d]^{\operatorname{res}} \\
\Fa (\partial W') \ar[r]^{\phi^\ast_\partial}_{\cong} & \Fa (\partial W)
} \]
provided by axiom (FHOMEO).
Thus $\phi^\ast F' \in \Fa (W,\phi^\ast_\partial f)$ and the desired restriction
$\phi^\ast_\rel$ exists. As $\phi^\ast$ is injective, $\phi^\ast_\rel$ is injective
as well. Given $F\in \Fa (W,\phi^\ast_\partial f),$ the field
$(\phi^{-1})^\ast (F)$ lies in $\Fa (W',f)$ and
\[ \phi^\ast_\rel ((\phi^{-1})^\ast (F)) = \phi^\ast (\phi^{-1})^\ast (F)
  = (\phi^{-1} \phi)^\ast (F) = F. \]
This shows that $\phi^\ast_\rel$ is surjective, too, and proves the claim.

Axiom (THOMEO) for the system $\TT$ of action exponentials asserts
that $\TT_W (\phi^\ast F') = \TT_{W'} (F').$ Consequently, for a morphism
$\gamma: X\to Y$ in $\catc$,
\[ T_{W'} (F')_{XY} (\gamma) = \begin{cases}
1,& \text{ if } \gamma = \TT_{W'} (F') \\
0,& \text{ otherwise} \end{cases} =
 \begin{cases}
1,& \text{ if } \gamma = \TT_{W} (\phi^\ast F') \\
0,& \text{ otherwise} \end{cases} =
 T_W (\phi^\ast_\rel (F'))_{XY} (\gamma), \]
that is,
\[ T_W (\phi^\ast_\rel (F')) = T_{W'} (F'). \]
The bijection $\phi^\ast_\rel: \Fa (W',f) \to \Fa (W,\phi^\ast_\partial f)$
implies the identity
\[ \sum_{F\in \Fa (W,\phi^\ast_\partial f)} T_W (F) =
  \sum_{F'\in \Fa (W',f)} T_W (\phi^\ast_\rel F'). \]
Hence, the pushforward of the state sum of $W$, subject to the boundary 
condition $f$, can be calculated as 
\begin{eqnarray*}
(\phi_{\partial \ast} (Z_W))(f)
& = & (\fun_Q (\phi^\ast_\partial)(Z_W))(f) =
 Z_W (\phi^\ast_\partial (f)) \\
& = & \sum_{F\in \Fa (W,\phi^\ast_\partial f)} T_W (F) =
 \sum_{F'\in \Fa (W',f)} T_W (\phi^\ast_\rel F') \\
& = & \sum_{F'\in \Fa (W',f)}  T_{W'} (F') 
=  Z_{W'} (f).
\end{eqnarray*}
\end{proof}

We will now use topological invariance to show that the state sum is really
a state, not just a pre-state.
\begin{prop}
The state sum $Z_W \in E(\partial W)$ solves the constraint equation 
(\ref{equ.constraint}). Thus $Z_W$ lies in the state module $Z(\partial W)\subset E(\partial W).$
\end{prop}
\begin{proof}
Given a field $f\in \Fa (\partial W)$ and $\phi \in \Homeo (\partial W)$ isotopic to the
identity, we need to show that $Z_W (\phi^\ast f)=Z_W (f).$
Let $H: \partial W \times [0,1] \to \partial W\times [0,1]$ be a level preserving isotopy,
$H(-,t)=(-,t),$ $H(x,0)=x,$ $H(x,1)=\phi (x)$, for all $x\in \partial W$.
This isotopy fits into a commutative diagram
\[ \xymatrix@C=40pt{
W \ar@{=}[d] & \partial W \ar@{^{(}->}[l] \ar@{^{(}->}[r]^{(\id_{\partial W},0)} \ar@{=}[d] &
  \partial W\times [0,1] \ar[d]^{H}_{\cong} \\
W & \partial W \ar@{^{(}->}[l] \ar@{^{(}->}[r]^{(\id_{\partial W},0)} & \partial W \times [0,1] 
} \]
The pushout of the rows 
is homeomorphic to $W$ (using collars, which are available in the topological
category by
Marston Brown's collar neighborhood theorem \cite{browntopcollars}) via a
homeomorphism which is the identity on the boundary.
Thus the universal property of pushouts applied to the above diagram
yields a homeomorphism $\Phi: W \to W$ of bordisms, whose restriction
to the boundary $\Phi_\partial$ is $\phi$.
By Theorem \ref{thm.statetopinv} on topological invariance, we have 
$\Phi_{\partial \ast} (Z_W) = Z_W$. Consequently,
\[ Z_W (f) = \Phi_{\partial \ast} (Z_W)(f) = \phi_\ast (Z_W)(f) = Z_W (\phi^\ast f). \]
\end{proof}

Homeomorphisms between closed $n$-manifolds induce isomorphisms on
the associated state modules, as we will now explain.
\begin{lemma}
Let $\phi: M\to N$ be any homeomorphism of closed $n$-manifolds.
Then the induced isomorphism $\phi_\ast: E(M)\to E(N)$ of pre-state modules restricts 
to an isomorphism $\phi_\ast: Z(M)\to Z(N)$ of state modules.
\end{lemma}
\begin{proof}
Let $z\in Z(M)$ be a state, that is, $z:\Fa (M)\to Q$ is a function with
$z(\psi^\ast f)=z(f)$ for all $\psi \in \Homeo (M)$ isotopic to the identity.
Let $g\in \Fa (N)$ be any field on $N$ and $\xi \in \Homeo (N)$ isotopic to $\id_N$.
Let $\psi \in \Homeo (M)$ be the homeomorphism $\psi = \phi^{-1} \xi \phi$.
If $H: N\times [0,1]\to N$ is an isotopy from $\xi$ to $\id_N$, then
\[ M\times [0,1] \stackrel{\phi \times \id_{[0,1]}}{\longrightarrow}
 N\times [0,1] \stackrel{H}{\longrightarrow} N 
  \stackrel{\phi^{-1}}{\longrightarrow} M \]
is an isotopy from $\psi$ to $\id_M$. Thus
\begin{eqnarray*}
\phi_\ast (z)(\xi^\ast g) & = & (z\circ \phi^\ast)(\xi^\ast g) =
  z((\xi \phi)^\ast g) = z((\phi \psi)^\ast g) \\
& = & z(\psi^\ast (\phi^\ast g)) = z(\phi^\ast g)= \phi_\ast (z)(g).
\end{eqnarray*}
Hence $\phi_\ast (z)$ solves the constraint equation on $N$ and so
$\phi_\ast (z)\in Z(N)$. Its inverse is given by
$(\phi^{-1})_\ast: Z(N)\to Z(M)$.
\end{proof}
By the above lemma, any homeomorphism $\phi: M\to N$ induces an isomorphism
$\phi_\ast: Z(M)\to Z(N)$.
If $\psi: N\to P$ is another homeomorphism, then
$\psi_\ast \circ \phi_\ast = (\psi \circ \phi)_\ast: Z(M)\to Z(P)$
and $(\id_M)_\ast = \id_{Z(M)}: Z(M)\to Z(M)$, that is, the state module $Z(-)$ is a
functor on the category of closed $n$-manifolds and homeomorphisms. 
In particular the group $\operatorname{Homeo}(M)$ of self-homeomorphisms
$M\to M$ acts on $Z(M)$.
\begin{thm} \label{thm.isotopyinv}
(Isotopy Invariance.)
Isotopic homeomorphisms $\phi, \psi: M\to N$ induce equal isomorphisms
$\phi_\ast = \psi_\ast: Z(M)\to Z(N)$ on state modules.
In particular, the action of $\Homeo (M)$ on $Z(M)$ factors through the
mapping class group.
\end{thm}
\begin{proof}
Let $H: M\times [0,1]\to N$ be an isotopy from $\phi$ to $\psi$.
Then $\psi^{-1} \circ H: \mti \to M$ is an isotopy from $\psi^{-1} \phi$ to $\id_M$.
Hence for $z\in Z(M),$
\[ (\psi^{-1} \phi)_\ast (z)(f) = z((\psi^{-1} \phi)^\ast f)=z(f). \]
It follows that
\[ (\psi_\ast)^{-1} \circ \phi_\ast = (\psi^{-1})_\ast \circ \phi_\ast
 = (\psi^{-1} \phi)_\ast = \id_{Z(M)} \]
and therefore $\phi_\ast = \psi_\ast$.
\end{proof}

As there are two multiplications available on $Q_S (\catc),$ there are also two
corresponding $Q_Q Q$-linear maps $\beta^c, \beta^m: E(M)\times E(N) \to
E(M)\hotimes E(N),$ given by
$\beta^c (z,z')(f,g) = z(f)\cdot z(g),$ using the multiplication $\cdot$ of the
composition semiring $Q^c$, and $\beta^m (z,z')(f,g) = z(f)\times z(g),$ 
using the multiplication $\times$ of the monoidal semiring $Q^m$.
If $(z,z')\in Z(M)\times Z(N),$ and $\phi \in \Homeo(M),$ $\psi \in \Homeo (N)$
are isotopic to the respective identities, then
\[ \beta^c (z,z')(\phi^\ast f, \psi^\ast g) =
  z(\phi^\ast f)\cdot z'(\psi^\ast g)= z(f)\cdot z'(g) = \beta^c (z,z')(f,g). \]
Thefore, $\beta^c (z,z')\in Z(M)\hotimes Z(N)$. Similarly, we have a
$Q_Q Q$-linear map $\beta^m: Z(M)\times Z(N)\to Z(M)\hotimes Z(N).$
Thus a pair of vectors $(z,z')\in Z(M)\times Z(N)$ determines two generally
different tensor products in $Z(M)\hotimes Z(N)$, namely
$z \hotimes_m z' = \beta^m (z,z')$ and $z \hotimes_c z' = \beta^c (z,z').$

\begin{thm} \label{thm.zdisjunion}
The state sum $Z_{W\sqcup W'} \in Z(\partial W \sqcup \partial W')$ of an
ordered disjoint union of bordisms $W$ and $W'$ is the tensor
product
\[ Z_{W\sqcup W'} = Z_W \hotimes_m Z_{W'} \in
  Z(\partial W)\hotimes Z(\partial W') \cong Z(\partial W \sqcup \partial W'). \]
\end{thm}
\begin{proof}
Let $F \in \Fa (W\sqcup W')$ be a field on the disjoint union. Then, regarding
$Q_S (\catc)$ as the monoidal semiring $Q^m$, we have on a morphism
$\gamma: X\to Y$ of $\catc$,
\[ (T_W (F|_W)\times T_{W'} (F|_{W'}))_{XY} (\gamma) =
 \sum_{\alpha \otimes \beta = \gamma} T_{W'} (F|_{W'})_{X'' Y''} (\beta)
  \cdot T_{W} (F|_{W})_{X' Y'} (\alpha). \]
The element $T_W (F|_W)_{X' Y'} (\alpha)\in S$ is $1$ when $\alpha =
\TT_W (F|_W)$ and $0$ otherwise. Similarly,
$T_{W'} (F|_{W'})_{X'' Y''} (\beta)\in S$ is $1$ when $\beta =
\TT_{W'} (F|_{W'})$ and $0$ otherwise. Thus
\begin{eqnarray*}
(T_W (F|_W)\times T_{W'} (F|_{W'}))_{XY} (\gamma) & = &
  \sum_{\TT_W (F|_W)\otimes \TT_{W'} (F|_{W'})=\gamma} 1\cdot 1 \\
& = & \begin{cases} 1,& \text{ if } \gamma = \TT_W (F|_W)\otimes \TT_{W'} (F|_{W'}) \\
  0,& \text{ otherwise.} \end{cases}
\end{eqnarray*}
Now by axiom (TDISJ), $\TT_W (F|_W)\otimes \TT_{W'} (F|_{W'}) =
\TT_{W\sqcup W'} (F),$ which implies
\[ T_W (F|_W)\times T_{W'} (F|_{W'}) =
  T_{W\sqcup W'} (F) \]
in $Q^m$. We apply this identity in calculating the state sum of the disjoint
union subject to the boundary condition $f\in \Fa (\partial W \sqcup \partial W')$:
\[ Z_{W\sqcup W'} (f) = \sum_{F \in \Fa (W\sqcup W', f)} T_{W\sqcup W'} (F) =
  \sum_{F \in \Fa (W\sqcup W', f)} T_W (F|_W)\times T_{W'} (F|_{W'}). \]
As pointed out in Section \ref{sec.monssemirings}, if $\{ m_i \}_{i\in I},$
$\{ n_j \}_{j\in J}$ are families of elements in a complete monoid and
$\sigma: J\to I$ is a bijection such that $m_{\sigma (j)} = n_j,$ then
$\sum_{i\in I} m_i = \sum_{j\in J} n_j.$ Applying this principle to the families
\[ \{ T_W (G)\times T_{W'} (G') \}_{(G,G') \in \Fa (W, f|_{\partial W})\times
   \Fa (W', f|_{\partial W'})}, \]
\[ \{ T_W (F|_W)\times T_{W'} (F|_{W'}) \}_{F\in \Fa (W\sqcup W',f)} \]
and to the bijection
\[ \sigma: \Fa (W\sqcup W',f)\longrightarrow
 \Fa (W, f|_{\partial W})\times \Fa (W', f|_{\partial W'}) \]
given by Lemma \ref{lem.fdisjbndryconds},
we obtain that
\begin{eqnarray*}
Z_{W\sqcup W'} (f) & = & \sum_{(G,G') \in \Fa (W, f|_{\partial W})\times
   \Fa (W', f|_{\partial W'})} T_W (G)\times T_{W'} (G') \\
& = & \left\{ \sum_{G\in \Fa (W,f|_{\partial W})} T_W (G) \right\} \times
  \left\{ \sum_{G'\in \Fa (W',f|_{\partial W'})} T_{W'} (G') \right\} \\
& = & Z_W (f|_{\partial W}) \times Z_{W'} (f|_{\partial W'}) \\
& = & \beta^m (Z_W, Z_{W'}) (f|_{\partial W}, f|_{\partial W'}) \\
& = & (Z_W \hotimes_m Z_{W'})(f|_{\partial W}, f|_{\partial W'}).
\end{eqnarray*}
\end{proof}
Especially in light of the decomposition of the state sum of a disjoint union provided
by Theorem \ref{thm.zdisjunion}, one
may wonder whether the state sum $Z_W \in Z(M_1)\hotimes \cdots \hotimes
Z(M_k)$ of any compact manifold $W$ whose boundary has the connected components
$M_1,\ldots, M_k,$ can always be written as
\[ Z_W (f) = \sum_{i=1}^l z_{i1}(f|_{M_1})\times
  z_{i2}(f|_{M_2})\times \cdots \times z_{ik}(f|_{M_k}) \]
for suitable functions $z_{ij} \in \fun_Q (\Fa (M_j))$, see
Remarks \ref{rem.wrongfuncttensprod} and \ref{rem.wrongstatemodule}. 
The following example shows that this is not the case.

\begin{example} \label{exple.statesumdoesnotdecomp}
Take $n=0$ and the Boolean semiring $\bool$ as the ground semiring.
Let $\catc$ be the category with one object $I$ and one morphism,
$\id_I: I\to I$. This category has a unique monoidal structure, which is strict.
Then
\[ Q_{\bool} (\catc) = \fun_{\bool} (\Hom_{\catc} (I,I)) =
  \fun_{\bool} (\{ \id_I \}) \cong \bool \]
is the Boolean semiring. For a $1$-dimensional bordism $W$,
let $\Fa (W)$ be the locally constant functions on $W$ with values in the
natural numbers $\nat = \{ 0,1,2,\ldots \}$, that is, if $W$ has $k$
connected components, then $\Fa (W) = \nat^k$.
Fields $\Fa (M)$ on a closed $n$-manifold $M$ are defined in the same way.
The restrictions are given by the ordinary restriction of functions to subspaces.
Homeomorphisms $W\cong W'$ and $M\cong N$ act on fields by permuting
function values in a manner consistent with the permutation which the
homeomorphism induces on the connected components.
Since two locally
constant functions that agree on a common boundary component glue to give
a locally constant function again, $\Fa$ is a system of fields.
The action exponential $\TT_W: \Fa (W)\to \Mor (\catc),$ 
$\TT_W (F)=\id_I$ is uniquely determined. Then $\TT$ is indeed a valid
system of action exponentials. 
Now let $W$ be the unit interval $[0,1]$, whose boundary $\partial W = M\sqcup N$
is given by the incoming boundary $M= \{ 0 \}$ and the outgoing boundary $N=\{ 1 \}$.
Let $f = (m,n)\in \Fa (\partial W) \cong \Fa (M)\times \Fa (N) = \nat \times \nat$ be
a boundary condition. If $m\not= n$, then $\Fa (W,f)$ is empty, for there is no
constant function on the unit interval that restricts to two distinct numbers on the
endpoints. Consequently,
\[ Z_W (m,n) = \sum_{F\in \Fa (W,(m,n))} T_W (F) =0\in Q \cong \bool. \]
On the other hand, if $m=n$, then $\Fa (W,(m,n))$ consists of a single element,
namely $m=n$, so that
$Z_W (m,n) = T_W (m) = 1\in \bool.$
Therefore, $Z_W (m,n) = \delta_{mn}$, that is, viewed as a countably
infinite Boolean value matrix, $Z_W$ is the identity matrix.
But it follows from results of \cite{banagl-tensor} that the identity matrix is not in the image of
\[ \mu:\fun_{\bool} (\nat)\otimes \fun_{\bool} (\nat) \longrightarrow
  \fun_{\bool} (\nat \times \nat). \] 
This can also be seen on foot by observing that a representation
$\delta_{mn} = \sum_{i=1}^l z_i (m)z'_i (n)$ (where we assume that $\delta_{mn}$
cannot be expressed in such a form with fewer than $l$ terms), implies that
$z_i (m) = \delta_{mn_i},$ $z'_i (n)=\delta_{n_i n}$ for certain $n_1,\ldots, n_l$.
But the finite sum $\sum_{i=1}^l \delta_{mn_i} \delta_{n_i n}$ can certainly
not equal $\delta_{mn}$ for all $m,n$ as follows by taking $m=n\not\in
\{ n_1,\ldots, n_l \}$.
\end{example}

Let $M,N,P$ be closed $n$-manifolds.
The contraction $\gamma$ of Section \ref{sec.funsemimods} defines a map
\[ \gamma: E(M) \hotimes E(N) \hotimes E(N) \hotimes E(P) \longrightarrow
  E(M) \hotimes E(P). \]
This contraction allows us to define an inner product
$\langle z,z' \rangle$ of a vector $z\in E(M)\hotimes E(N)$ and a vector
$z' \in E(N)\hotimes E(P)$ by setting
\[ \langle z,z' \rangle = \gamma (z\hotimes_c z') \in E(M)\hotimes E(P). \]
By Proposition \ref{prop.contractbilinear},
this inner product
\[ \langle -,- \rangle: (E(M)\hotimes E(N))\times (E(N)\hotimes E(P)) 
  \longrightarrow E(M)\hotimes E(P) \]
is $Q^c_{Q^c} Q^c$-linear.
By Proposition \ref{prop.contractassoc}, the inner product is associative,
i.e. for a fourth $n$-manifold $R$ and a vector $z''\in E(P)\hotimes E(R),$
one has $\langle \langle z,z' \rangle, z'' \rangle = \langle z, \langle z',z'' \rangle \rangle$
in $E(M)\hotimes E(R)$.
An injection $Z(M)\hotimes Z(N)\hookrightarrow E(M)\hotimes E(N)$ is given by
sending $z:\ovf (M)\times \ovf (N)\to Q$ to
\[ \xymatrix@C=40pt{ \Fa (M)\times \Fa (N) 
  \ar@{->>}[r]^{\operatorname{quot}\times \operatorname{quot}}
  & \ovf (M)\times \ovf (N) \ar[r]^z & Q.} \] 
Then the inner product $\langle z,z' \rangle$ of a state vector
$z\in Z(M)\hotimes Z(N)$ and a state vector
$z' \in Z(N)\hotimes Z(P)$ satisfies the constraint equation and is thus a
state $\langle z,z' \rangle \in Z(M)\hotimes Z(P).$
Therefore, restriction defines an inner product
\[ \langle -,- \rangle: (Z(M)\hotimes Z(N))\times (Z(N)\hotimes Z(P)) 
  \longrightarrow Z(M)\hotimes Z(P). \]
\begin{thm} \label{thm.statesumgluing}
(Gluing Formula.)
Let $W'$ be a bordism from $M$ to $N$ and let $W''$ be a bordism from
$N$ to $P$. Let $W = W' \cup_N W''$ be the bordism from $M$ to $P$ obtained
by gluing $W'$ and $W''$ along $N$. Then the state sum of $W$ can be
calculated as the contraction inner product
\[ Z_W = \langle Z_{W'}, Z_{W''} \rangle \in Z(M)\hotimes Z(P)\cong Z(M\sqcup P). \]
\end{thm}
\begin{proof}
Let $F \in \Fa (W)$ be a field on the glued bordism. Then, regarding
$Q_S (\catc)$ as the composition semiring $Q^c$, we have on a morphism
$\gamma: X\to Y$ of $\catc$,
\[ (T_{W'} (F|_{W'})\cdot T_{W''} (F|_{W''}))_{XY} (\gamma) =
 \sum_{\beta \circ \alpha= \gamma} T_{W''} (F|_{W''})_{ZY} (\beta)
  \cdot T_{W'} (F|_{W'})_{XZ} (\alpha). \]
(Here, $Z$ is of course an object of $\catc$ and not a state sum.)
The element $T_{W'} (F|_{W'})_{XZ} (\alpha)\in S$ is $1$ when $\alpha =
\TT_{W'} (F|_{W'})$ and $0$ otherwise. Similarly,
$T_{W''} (F|_{W''})_{ZY} (\beta)\in S$ is $1$ when $\beta =
\TT_{W''} (F|_{W''})$ and $0$ otherwise. Thus
\begin{eqnarray*}
(T_{W'} (F|_{W'})\cdot T_{W''} (F|_{W''}))_{XY} (\gamma) & = &
  \sum_{\TT_{W''} (F|_{W''})\circ \TT_{W'} (F|_{W'})=\gamma} 1\cdot 1 \\
& = & \begin{cases} 1,& \text{ if } \gamma = \TT_{W''} (F|_{W''})\circ \TT_{W'} (F|_{W'}) \\
  0,& \text{ otherwise.} \end{cases}
\end{eqnarray*}
Now by axiom (TGLUE), $\TT_{W''} (F|_{W''})\circ \TT_{W'} (F|_{W'}) =
\TT_{W} (F),$ which implies
\[ T_{W'} (F|_{W'})\cdot T_{W''} (F|_{W''}) =
  T_{W} (F) \]
in $Q^c$. Let $f\in \Fa (\partial W) = \Fa (M\sqcup P).$
Lemma \ref{lem.fgluebndryconds} asserts that the unique map $\rho$ such that
\[ \xymatrix{
\Fa (W,f) \ar@/^1pc/[rrd] \ar[rd]^{\rho} \ar@/_1pc/[ddr] & & \\
& \Fa (W',W'',f|_M, f|_P) \ar[r] \ar[d] & \Fa (W',f|_M) \ar[d] \\
& \Fa (W'',f|_P) \ar[r] & \Fa (N)
} \]
commutes, is a bijection. Therefore,
\[ \sum_{F\in \Fa (W,f)} T_{W'} (F|_{W'})\cdot T_{W''} (F|_{W''}) =
 \sum_{(G',G'') \in \Fa (W',W'',f|_M,f|_P)}
  T_{W'} (G')\cdot T_{W''} (G''). \]
The pullback $\Fa (W',W'',f|_M,f|_P)$ possesses the natural partition
\[ \Fa (W',W'',f|_M, f|_P) = \bigcup_{g\in \Fa (N)}
  \Fa (W',f|_M,g) \times \Fa (W'',g,f|_P), \]
which enables us to rewrite the sum over the pullback as
\[ \sum_{(G',G'') \in \Fa (W',W'',f|_M,f|_P)}
  T_{W'} (G')\cdot T_{W''} (G'') = \hspace{4cm} \]
\[ \hspace{2cm} = \sum_{g\in \Fa (N)}
   \sum_{~(G',G'') \in \Fa (W',f|_M,g) \times \Fa (W'',g,f|_P)}
  T_{W'} (G')\cdot T_{W''} (G'')  \]
according to the partition property that the summation law in a complete
monoid satisfies. Thus for the state sum of the glued bordism, subject
to the boundary condition $f$,
\begin{eqnarray*}
Z_W (f) & = & \sum_{F\in \Fa (W,f)} T_W (F) \\
& = & \sum_{F\in \Fa (W,f)} T_{W'} (F|_{W'})\cdot T_{W''} (F|_{W''}) \\
& = & \sum_{(G',G'') \in \Fa (W',W'',f|_M,f|_P)}
  T_{W'} (G')\cdot T_{W''} (G'') \\
& = & \sum_{g\in \Fa (N)}
   \sum_{~G' \in \Fa (W',f|_M,g)} \sum_{~G'' \in \Fa (W'',g,f|_P)}
  T_{W'} (G')\cdot T_{W''} (G'') \\
& = & \sum_{g\in \Fa (N)}
  \big( \sum_{G' \in \Fa (W',f|_M,g)} T_{W'} (G') \big)\cdot 
  \big( \sum_{G'' \in \Fa (W'',g,f|_P)} T_{W''} (G'') \big) \\
& = & \sum_{g\in \Fa (N)} Z_{W'} (f|_M \sqcup g)\cdot 
   Z_{W''} (g\sqcup f|_P) \\
& = & \sum_{g\in \Fa (N)} \beta^c (Z_{W'}, Z_{W''})
   (f|_M \sqcup g, g\sqcup f|_P) \\
& = & \gamma (Z_{W'} \hotimes_c Z_{W''})(f) \\
& = & \langle Z_{W'}, Z_{W''} \rangle (f).
\end{eqnarray*}
(Note that via the identification $\Fa (M \sqcup P) \cong \Fa (M)\times \Fa (P),$
we may also think of the boundary condition $f$ as the pair $(f|_M, f|_P)$.)
\end{proof}
Taking bordisms with $N=\varnothing$ is allowed in the above gluing theorem.
In such a situation, $W= W' \cup_{\varnothing} W'' = W' \sqcup W''$ so that
by using the gluing theorem in conjunction with Theorem \ref{thm.zdisjunion}
on disjoint unions, we obtain
$\langle Z_{W'}, Z_{W''} \rangle = Z_{W'} \hotimes_m Z_{W''}.$\\

As an application of our theorems, let us derive the well-known
zig-zag equation. Given a closed $n$-manifold, we write
$M_2 = M\sqcup M$ and $M_3 = M\sqcup M\sqcup M$. The cylinder $M\times [0,1]$ can be interpreted as
a bordism in three different ways:
\[ C=(M\times [0,1],M,M),~
   C_\subset =(M\times [0,1],\varnothing,M_2),~
   C_\supset =(M\times [0,1],M_2,\varnothing). \]
Using these, we can form the bordisms
\[ W' = C_\subset \sqcup C,~ W'' = C\sqcup C_\supset,~
  W = W' \cup_{M_3} W''. \]
Then $W$ is homeomorphic to $C$ by a homeomorphism $\phi$ which is the identity
on the boundary, $\phi_\partial = \id_{M_2}$. Thus, by topological invariance of the
state sum, $Z_W = Z_C$. On the other hand, using Theorems
\ref{thm.zdisjunion} and \ref{thm.statesumgluing}, 
\[ Z_W = \langle Z_{W'}, Z_{W''} \rangle =
 \langle Z_{C_\subset \sqcup C}, Z_{C\sqcup C_\supset} \rangle =
 \langle Z_{C_\subset} \hotimes_m Z_C, Z_C \hotimes_m Z_{C_\supset} \rangle. \]
Thus we arrive at the zig-zag equation
\[ \langle Z_{C_\subset} \hotimes_m Z_C, Z_C \hotimes_m Z_{C_\supset} \rangle
  = Z_C, \]
relating the state sums of the three bordisms associated to $M\times [0,1]$. \\

Finally, we introduce a new kind of invariant $\gA (M)\in Z(M)$,
the \emph{coboundary aggregate} of a closed $n$-manifold $M$,
which has no counterpart in classical topological field theories.
Call two bordisms $(W_1, \varnothing, M)$ and $(W_2,\varnothing, M)$
equivalent, if there exists
a homeomorphism $W_1 \cong W_2$ whose restriction to the
boundary is the identity on $M$. (This is indeed an equivalence relation.)
Let $\Cob (M)$ be the collection of all equivalence classes of bordisms
(``$M$-coboundaries'') $(W,\varnothing, M)$.
If $W_1$ and $W_2$ are equivalent, then Theorem \ref{thm.statetopinv}
implies that they have equal state sums, $Z_{W_1} = Z_{W_2} \in Z(M)$.
Therefore, the state sum can be viewed as a well-defined function
$Z: \Cob (M)\to Z(M)$.
\begin{defn}
The \emph{coboundary aggregate} $\gA (M)\in Z(M)$ of a closed
topological $n$-manifold is the state vector
\[ \gA (M) = \sum_{[W]\in \Cob (M)} Z_{[W]}. \]
\end{defn}
Note that this is a well-defined element of $Z(M)$, since the completeness of $Q$
together with Proposition \ref{prop.funsemimod} implies that
$Z(M)$ is complete as well.
\begin{thm} \label{thm.aggregatetopinv}
If $\phi: M\to N$ is a homeomorphism, then $\phi_\ast \gA (M) = \gA (N)$.
That is, the coboundary aggregate is a topological invariant.
\end{thm}
\begin{proof}
Given a bordism $W$ with $\partial W^{\operatorname{in}} =\varnothing$ and
$\partial W^{\operatorname{out}}=M,$ let 
\[ \psi (W) = W \cup_{\phi} N\times [0,1], \]
that is, let $\psi (W)$ be the mapping cylinder of $\phi^{-1}$ followed by
the inclusion $M\subset W$. Then $\psi (W)$ is a bordism
with $\partial \psi (W)^{\operatorname{in}} =\varnothing$ and
$\partial \psi (W)^{\operatorname{out}}=N.$
Suppose that $W_1$ and $W_2$ represent
the same element of $\Cob (M)$. Then there exists a homeomorphism
$\Phi: W_1 \to W_2$ whose restriction to the boundary is the identity on $M$.
The pushout of the top row of the commutative diagram
\[ \xymatrix@C=40pt{
W_1 \ar[d]_{\Phi}^{\cong} & M \ar@{^{(}->}[l] \ar@{^{(}->}[r]^{(\phi,0)} \ar@{=}[d] &
  N\times [0,1] \ar@{=}[d] \\
W_2 & M \ar@{^{(}->}[l] \ar@{^{(}->}[r]^{(\phi,0)} & N\times [0,1] 
} \]
is $\psi (W_1)$, while the pushout of the bottom row is $\psi (W_2)$.
Thus the universal property of pushouts applied to the above diagram
yields a homeomorphism $\psi (W_1) \cong \psi (W_2),$ whose restriction
to the boundary is the identity on $N$. Hence $\psi (W_1)$ and $\psi (W_2)$
are equivalent and represent the same element of $\Cob (N)$.
Consequently, $\psi$ induces a well-defined map
$\psi: \Cob (M)\to \Cob (N)$. Reversing the roles of $M$ and $N$
(using $\phi^{-1}$), we also
obtain a map $\psi': \Cob (N)\to \Cob (M)$. We claim that $\psi$ and
$\psi'$ are inverse to each other:
The colimit of the top row of the commutative diagram
\[ \xymatrix@C=40pt{
W \ar@{=}[d] & M \ar@{^{(}->}[l] \ar@{^{(}->}[r]^{(\id_M,0)} \ar@{=}[d] &
  M\times [0,1] \ar[d]^{\phi \times \id_{[0,1]}}_{\cong} 
 & M \ar@{^{(}->}[l]_{\operatorname{at } 1}  
  \ar@{^{(}->}[r]^{(\id_M,0)} 
 \ar[d]^{\phi}_{\cong} & M\times [0,1] \ar@{=}[d] \\
W & M \ar@{^{(}->}[l] \ar@{^{(}->}[r]^{(\phi,0)} & N\times [0,1]
& N \ar@{^{(}->}[l]_{\operatorname{at } 1} 
  \ar@{^{(}->}[r]^{(\phi^{-1},0)} & M\times [0,1] 
} \]
is homeomorphic to $W$ (using 
the collar neighborhood theorem \cite{browntopcollars}) via a
homeomorphism which is the identity on the boundary,
while the colimit of the bottom row is $\psi' (\psi (W))$.
Thus the universal property of colimits applied to the above diagram
yields a homeomorphism $W \cong \psi' (\psi (W)),$ whose restriction
to the boundary is the identity. We conclude that
$[W] = \psi' \psi [W] \in \Cob (M)$. By symmetry we also have
$[W'] = \psi \psi' [W'] \in \Cob (N)$. This shows that $\psi' = \psi^{-1}$
and $\psi$ is a bijection.

Given a field $f\in \Fa (N),$ let us show that
\begin{equation} \label{equ.zwphifiszpsiwf}
Z_W (\phi^\ast f) = Z_{\psi (W)} (f). 
\end{equation}
The pushout of the top row of the commutative diagram
\[ \xymatrix@C=40pt{
W \ar@{=}[d] & M \ar@{^{(}->}[l] \ar@{^{(}->}[r]^{(\id_M,0)} \ar@{=}[d] &
  M\times [0,1] \ar[d]^{\phi \times \id_{[0,1]}}_{\cong} \\
W & M \ar@{^{(}->}[l] \ar@{^{(}->}[r]^{(\phi,0)} & N\times [0,1] 
} \]
is homeomorphic to $W$ (using the collar neighborhood theorem) via a
homeomorphism which is the identity on the boundary,
while the pushout of the bottom row is $\psi (W)$.
Thus the universal property of pushouts applied to the above diagram
yields a homeomorphism of bordisms $W \cong \psi (W),$ whose restriction
to the boundary is $\phi$. By topological invariance of the state sum
(Theorem \ref{thm.statetopinv}), $\phi_\ast (Z_W) = Z_{\psi (W)}$, proving
equation (\ref{equ.zwphifiszpsiwf}). Consequently,
\begin{eqnarray*}
(\phi_\ast \gA (M))(f) & = & \gA (M)(\phi^\ast f) =
  \sum_{[W]\in \Cob (M)} Z_{[W]} (\phi^\ast f) \\
& = & \sum_{[W]\in \Cob (M)} Z_{\psi [W]} (f) =
  \sum_{[W'] \in \psi(\Cob (M))} Z_{[W']} (f) \\
& = & \gA (N)(f).
\end{eqnarray*}
\end{proof}

\section{The Frobenius Structure}
\label{sec.frobenius}

We shall show that the state modules $Z(M)$, in any dimension $n$, 
come naturally equipped with the structure of a Frobenius semialgebra.
Let $S$ be any semiring, not necessarily commutative.

\begin{defn}
A \emph{Frobenius semialgebra} over $S$ is a two-sided $S$-semialgebra $A$
together with an $S$-bisemimodule homomorphism
$\epsilon: A\to S,$ called the \emph{counit} functional (or sometimes \emph{trace}), such that the
$S_S S$-linear form
\[ A\times A\to S,~ (a,b)\mapsto \epsilon (a\cdot b) \]
is nondegenerate. A \emph{morphism $\phi: A\to B$ of Frobenius
$S$-semialgebras} is a morphism of two-sided $S$-semialgebras which
commutes with the counits, i.e. $\epsilon_A = \epsilon_B \circ \phi$.
\end{defn}

\begin{prop}
The state module $Z(M)$ of a closed, $n$-dimensional manifold $M$
becomes a Frobenius semialgebra over $Q^m$ and over $Q^c$ when
endowed with the counit functional
\[ \epsilon = \epsilon_M: Z(M)\to Q,~ \epsilon (z) = \sum_{f\in \Fa (M)} z(f). \]
\end{prop}
\begin{proof}
The functional is clearly a $Q^c$-bisemimodule homomorphism and
a $Q^m$-bisemimodule homomorphism.
To show that $Z(M)$ is indeed Frobenius over the monoidal semiring  
$Q^m =(Q_S (\catc), +, \times, 0, 1^\times)$,
we shall, to a given nonzero $z\in Z(M),$ construct a nonzero
$z'\in Z(M)$ such that $\epsilon (z\times z')$ is not zero.
If $z\not =0$, then there exists a class $[f_0] \in \ovf (M)$,
represented by a field $f_0$, such that
$z[f_0]\not= 0$. Setting
\[ z' [f] = \begin{cases} 1^\times, & \text{ if } [f]=[f_0] \\
0, & \text{ otherwise,} \end{cases} \]
we obtain a state $z'\in Z(M)$. Using the partition
$\Fa (M) = \bigcup_{c\in \ovf (M)} c,$ the trace of the product is
\begin{eqnarray*}
\epsilon (z\times z') & = & \sum_{f\in \Fa (M)} z(f)\times z' (f) =
 \sum_{c\in \ovf (M)} \sum_{f\in c} z(f)\times z'(f) \\
& = & \sum_{f\in [f_0]} z(f)\times z'(f) = \sum_{f\in [f_0]} z(f)\times z'(f_0) \\
& = & \sum_{f\in [f_0]} z(f) = \sum_{f\in [f_0]} z(f_0).
\end{eqnarray*}
Now if $R$ is a complete semiring, $r\in R$ a nonzero element and
$J$ an arbitrary nonempty index set, then $\sum_{j\in J} r$ cannot be zero.
For if it were zero, then
\[ 0 = \sum_{j\in J} r = r + \sum_{j\in J - \{ j_0 \}} r. \]
But as $R$ is complete, it is in particular zerosumfree, which would imply $r=0$,
a contradiction. Thus, since $z(f_0)\not=0$, we have
$\sum_{f\in [f_0]} z(f_0)\not= 0$.
Similarly, to prove that $Z(M)$ is Frobenius over the
composition semiring $Q^c =(Q_S (\catc), +, \cdot, 0, 1)$,
take 
\[ z' [f] = \begin{cases} 1, & \text{ if } [f]=[f_0] \\
0, & \text{ otherwise.} \end{cases} \]
\end{proof}
\begin{remark}
As every state is constant over classes $[f]\in \ovf (M),$ it might at first
seem more natural to construct the counit functional by summing over classes,
rather than individual fields. However, summing over classes would invalidate
Proposition \ref{prop.contreps}, and thus would be less compatible with
the existing formalism.
\end{remark}

More generally, we can also define counits of the form
\[ \epsilon_{M,-}: Z(M) \hotimes Z(N) \longrightarrow
  Z(N) \]
on the tensor product of the state modules of two closed $n$-manifolds
$M$ and $N$ by
\[ \epsilon_{M,-}(z)(g) = \sum_{f\in \Fa (M)} z(f,g), \]
$z: \ovf (M)\times \ovf (N)\to Q,$ $g\in \Fa (N)$; similarly for
$\epsilon_{-,N}$.
The next proposition
shows that the Frobenius counit interacts multiplicatively with the
tensor product of vectors.
\begin{prop} \label{prop.epstensmult}
Given two vectors $z\in Z(M),$ $z' \in Z(N),$ we have the formula
\[ \epsilon_{M\sqcup N} (z\hotimes_c z') = \epsilon_M (z)\cdot
  \epsilon_N (z') \]
in the composition semiring $Q^c$, and
\[ \epsilon_{M\sqcup N} (z\hotimes_m z') = \epsilon_M (z)\times
  \epsilon_N (z') \]
in the monoidal semiring $Q^m$.
\end{prop}
\begin{proof}
Viewing $Q$ as $Q^c$,
\begin{eqnarray*}
\epsilon_{M\sqcup N} (z\hotimes_c z') & = &
  \sum_{f\in \Fa (M\sqcup N)} (z\hotimes_c z')(f) =
 \sum_{f\in \Fa (M\sqcup N)} \beta^c (z,z')(f|_M, f|_N) \\
& = & \sum_{(g,g')\in \Fa (M)\times \Fa (N)} \beta^c (z,z')(g,g') =
  \sum_{g\in \Fa (M)} \sum_{g' \in \Fa (N)} z(g)\cdot z' (g') \\
& = & \left\{ \sum_{g\in \Fa (M)} z(g) \right\} \cdot \left\{
  \sum_{g'\in \Fa (N)} z' (g') \right\} = \epsilon_M (z)\cdot \epsilon_N (z'),
\end{eqnarray*}
using axiom (FDISJ) for systems of fields. If one replaces $\cdot$ by $\times$
and $c$ by $m$ in this calculation, one obtains the corresponding formula
in $Q^m$.
\end{proof}
Let $W\sqcup W'$ be a disjoint union of two bordisms
$W,W'$. By Theorem \ref{thm.zdisjunion},
the state sum of the disjoint union decomposes as
$Z_{W\sqcup W'} = Z_W \hotimes_m Z_{W'}$ and thus by the above Proposition
\[ \epsilon_{\partial (W\sqcup W')} (Z_{W\sqcup W'}) =
  \epsilon_{\partial W} (Z_W)\times \epsilon_{\partial W'} (Z_{W'}). \]

Given three $n$-manifolds $M,N,P$ and two vectors
$z\in Z(M)\hotimes Z(N),$ $z' \in Z(N)\hotimes Z(P),$ we may form their
contraction product $\langle z,z' \rangle \in Z(M)\hotimes Z(P)$.
The following proposition expresses the counit image of 
$\langle z,z' \rangle$ in terms of the counits
$\epsilon_{M,-}: Z(M)\hotimes Z(N)\to Z(N),$
$\epsilon_{-,P}: Z(N)\hotimes Z(P)\to Z(N)$ and
$\epsilon_N: Z(N)\to Q.$
\begin{prop} \label{prop.contreps}
Given vectors $z\in Z(M)\hotimes Z(N),$ $z' \in Z(N)\hotimes Z(P),$
the identity
\[ \epsilon_{M\sqcup P} \langle z,z' \rangle = \epsilon_N
(\epsilon_{M,-}(z)\cdot \epsilon_{-,P}(z')) \]
holds in $Q$, using on the right hand side the semiring-multiplication on $Z(N)$.
\end{prop}
\begin{proof}
The desired formula is established by the calculation
\begin{eqnarray*}
\epsilon_{M\sqcup P} \langle z,z' \rangle & = &
 \epsilon_{M\sqcup P} \gamma (z\hotimes_c z') =
 \sum_{f\in \Fa (M\sqcup P)} \gamma (\beta^c (z,z'))(f|_M, f|_P) \\
& = & \sum_{f\in \Fa (M)} \sum_{h\in \Fa (P)}
  \gamma (\beta^c (z,z'))(f,h) \\
& = & \sum_{f\in \Fa (M)} \sum_{h\in \Fa (P)} \sum_{g\in \Fa (N)}
  \beta^c (z,z')(f,g,g,h) \\
& = & \sum_{g\in \Fa (N)} \sum_{f\in \Fa (M)} \sum_{h\in \Fa (P)}
  z(f,g)\cdot z'(g,h) \\
& = & \sum_{g\in \Fa (N)} \left\{ \sum_{f\in \Fa (M)} z(f,g) \right\} \cdot
 \left\{ \sum_{h\in \Fa (P)} z'(g,h) \right\} \\
& = & \sum_{g\in \Fa (N)} \epsilon_{M,-}(z)(g)\cdot
   \epsilon_{-,P}(z')(g) \\
& = & \sum_{g\in \Fa (N)} (\epsilon_{M,-}(z)\cdot
   \epsilon_{-,P}(z'))(g) \\
& = &  \epsilon_N
(\epsilon_{M,-}(z)\cdot \epsilon_{-,P}(z')).
\end{eqnarray*}
\end{proof}
Let $W'$ be a bordism from $M$ to $N$ and let $W''$ be a bordism from
$N$ to $P$. By Theorem \ref{thm.statesumgluing}, the state sum $Z_W$
of the bordism $W=W' \cup_N W''$ obtained by gluing $W'$ and $W''$
along $N$ is given by the contraction product
$Z_W = \langle Z_{W'}, Z_{W''} \rangle$. Thus by Proposition
\ref{prop.contreps},
\[ \epsilon_{M\sqcup P} (Z_W) = \epsilon_N
(\epsilon_{M,-}(Z_{W'})\cdot \epsilon_{-,P}(Z_{W''})). \]

A homeomorphism $\phi: M\to N$ induces covariantly an isomorphism
$\phi_\ast: Z(M) \to Z(N)$
of both two-sided $Q^c$-semialgebras and $Q^m$-semialgebras.
This is even an isomorphism of Frobenius semialgebras, as the calculation
\begin{eqnarray*}
\epsilon_N \phi_\ast (z) & = &
\sum_{g\in \Fa (N)} \phi_\ast (z)(g) = \sum_{g\in \Fa (N)} \fun_Q (\phi^\ast)(z)(g) \\
& = & \sum_{g\in \Fa (N)} z(\phi^\ast g) = \sum_{f\in \Fa (M)} z(f) = 
 \epsilon_M (z),
\end{eqnarray*}
$z\in Z(M),$ shows, where we have used the bijection
$\phi^\ast: \Fa (N)\to \Fa (M)$ provided by axiom (FHOMEO).
If $\phi:W\to W'$ is a homeomorphism of bordisms with restriction
$\phi_\partial:\partial W \to \partial W',$ then by Theorem \ref{thm.statetopinv},
\[ \epsilon_{\partial W'} (Z_{W'}) = \epsilon_{\partial W'}(\phi_{\partial \ast} Z_W)
  = \epsilon_{\partial W} (Z_W). \]
Thus state sums of homeomorphic bordisms have equal Frobenius traces.

\section{Linear Representations}
\label{sec.linearreps}

Let $\vect$ denote the category of vector spaces over some fixed field, with morphisms the
linear maps. While category-valued systems $\TT$ of action exponentials, as
formulated in Definition \ref{def.actions}, provide a lot of flexibility, one is often
ultimately interested in linear categories, as those are thoroughly understood and
possess a rich, well developed theory of associated invariants. Thus in practice, the 
process of constructing a useful positive TFT will consist of two steps: First, find
a (small) strict monoidal category $\catc$ and a system of fields that possesses an
interesting system of action exponentials into $\catc$. The morphisms of $\catc$ may
still be geometric or topological objects associated to the fields in a monoidal way.
In the second step, construct a linear representation of $\catc$, that is,
construct a monoidal functor $R: \catc \to \vect$. This converts the morphisms of 
$\catc$ into linear maps, which can then be analyzed using tools from linear algebra.
From this perspective, the category $\catc$ plays an intermediate role in the
construction of a TFT: it ought to be large enough to be able to record interesting
information of the fields on a manifold, but small enough so as to allow for
manageable linear representations. We will verify in this section
(Proposition \ref{prop.lvectvalactexp}) that the composition of the $\catc$-valued action exponentials
with the representation $R$ yields a system of $\vect$-valued action exponentials,
which then have their associated positive TFT $Z$, provided that $\vect$ is
endowed with the structure of a strict monoidal category. At this point, we face
a formal problem: The ordinary tensor product of vector spaces is not
strictly associative and the unit is not strict either. This problem can be solved
by endowing $\vect$, without changing its objects and morphisms, with a strict (symmetric)
monoidal structure, which is monoidally equivalent to the usual monoidal structure
on $\vect$.

\subsection{The Schauenburg Tensor Product}

The ordinary tensor product of vector spaces is well-known
not to be associative, though it is associative up to natural isomorphism. Thus,
if we endowed $\vect$ with the ordinary tensor product
and took the unit object $I$ to be the one-dimensional
vector space given by the ground field, then, using obvious associators and unitors,
$\vect$ would become a monoidal category, but not a strict one.
There is an abstract process of turning a monoidal category $\catc$ into 
a monoidally equivalent strict monoidal category $\catc^{\operatorname{str}}$. 
However, this process
changes the category considerably and is thus not always practical.
Instead, we base our monoidal structure on the Schauenburg tensor product
$\odot$ introduced in \cite{schauenburg}, which does not change the
category $\vect$ at all. 
The product $\odot$ satisfies the strict associativity
\[ (U\odot V)\odot W = U\odot (V\odot W). \]
We shall thus simply write $U\odot V \odot W$ for this vector space.
The unit object $I$ remains the same as in the usual nonstrict monoidal structure,
and one has
\[  V\odot I = V,~ I\odot V =V. \]
The strict
monoidal category $(\vect, \odot, I)$ thus obtained is monoidally equivalent to the usual nonstrict
monoidal category of vector spaces. The underlying functor of this monoidal
equivalence is the identity.
In particular, there is a natural isomorphism $\xi: \otimes \to \odot,$
$\xi_{VW}: V\otimes W \to V\odot W$, where $\otimes$ denotes the standard
tensor product of vector spaces.
Naturality means that for every pair of linear
maps $f: V\to V'$, $g:W\to W',$ the square
\[ \xymatrix{
V\otimes W \ar[r]^{\xi_{VW}}_{\cong} \ar[d]_{f\otimes g} &
  V\odot W \ar[d]^{f\odot g} \\
V'\otimes W' \ar[r]^{\xi_{V'W'}}_{\cong} & V'\odot W'
} \]
commutes. Note that via $\xi$ we can speak of elements
$v\odot w\in V\odot W,$ $v\odot w = \xi_{VW} (v\otimes w),$ $v\in V,$ $w\in W$.
As the diagram 
\[ \xymatrix{
(U\otimes V)\otimes W \ar[r]^{\xi \otimes \id} \ar[d]_a &
(U\odot V)\otimes W \ar[r]^{\xi} &
(U\odot V)\odot W \ar@{=}[d] \\
U\otimes (V\otimes W) \ar[r]^{\id \otimes \xi} &
U\otimes (V\odot W) \ar[r]^{\xi} &
U\odot (V\odot W)
} \]
commutes, the identity
\[ (u\odot v)\odot w = u\odot (v\odot w) \]
holds for elements $u\in U, v\in V$ and $w\in W$. \\

The basic idea behind the construction of $\odot$ is to build a specific new
equivalence of categories 
$L: \vect \rightleftarrows \vect^{\operatorname{str}}: R$ such that $LR$ is the
identity and then setting $V\odot W = R(LV*LW),$ where $*$ is the strictly associative 
tensor product in $\vect^{\operatorname{str}}$. Then
\begin{eqnarray*}
(U\odot V)\odot W & = & R(L(U\odot V)*LW) =
  R(LR (LU*LV)*LW) = R((LU*LV)*LW) \\
& = & R(LU*(LV*LW)) = R(LU*LR(LV*LW)) =
   R(LU*L(V\odot W)) \\
& = & U\odot (V\odot W).
\end{eqnarray*}
For the rest of this section we will always use the Schauenburg tensor product
on $\vect$ and thus will from now on write $\otimes_\vect$ or simply $\otimes$ for $\odot$. \\

A braiding on the strict monoidal category $\vect$ is defined by
taking $b: V\otimes W \cong W\otimes V$ to be $v\otimes w\mapsto
w\otimes v$. The hexagon equations are satisfied. As $b^2 =1$,
$\vect$ endowed with $b$ is thus a symmetric strict monoidal category.

\subsection{Linear Positive TFTs}

Let $(\catc, \otimes, I)$ be a strict monoidal small category and
$R: \catc \to \vect$ a strict monoidal functor, that is, a linear representation
of $\catc$. Let $\Fa$ be a system of fields and $\TT$ a system of $\catc$-valued
action exponentials. The \emph{$R$-linearization} $\LL$ of $\TT$ is given on a
bordism $W$ by the composition
\[ \LL_W: \Fa (W) \stackrel{\TT_W}{\longrightarrow} \Mor (\catc) 
 \stackrel{R}{\longrightarrow} \Mor (\vect). \]

\begin{prop} \label{prop.lvectvalactexp}
The $R$-linearization $\LL$ of $\TT$ is a system of $\vect$-valued action 
exponentials.
\end{prop}
\begin{proof}
Using axiom (TDISJ) for $\TT$, we have for a disjoint union $W\sqcup W'$
and a field $f\in \Fa (W\sqcup W')$:
\begin{eqnarray*}
\LL_{W\sqcup W'} (f) & = &
 R(\TT_{W\sqcup W'} (f)) 
 =  R(\TT_W (f|_W) \otimes_{\catc} \TT_{W'} (f|_{W'})) \\
& = & R\TT_W (f|_W) \otimes_{\vect} R\TT_{W'} (f|_{W'}) 
 =  \LL_W (f|_W) \otimes_{\vect} \LL_{W'} (f|_{W'}).
\end{eqnarray*}
This proves (TDISJ) for $\LL$.
If $W=W'\cup_N W''$ is obtained by gluing a bordism
$W'$ with outgoing boundary $N$ to a bordism $W''$ with incoming boundary
$N$, then on $f\in \Fa (W),$
\begin{eqnarray*}
\LL_{W} (f) & = &
 R(\TT_{W} (f)) = R(\TT_{W''} (f|_{W''}) \circ_{\catc} \TT_{W'} (f|_{W'})) \\
& = & R\TT_{W''} (f|_{W''}) \circ_{\vect} R\TT_{W'} (f|_{W'}) 
 =  \LL_{W''} (f|_{W''}) \circ_{\vect} \LL_{W'} (f|_{W'}),
\end{eqnarray*}
using (TGLUE) for $\TT$. This establishes (TGLUE) for $\LL$.
Lastly, for a homeomorphism $\phi:W\to W'$ and a field $f\in \Fa (W'),$
\[ \LL_W (\phi^\ast f) = R\TT_W (\phi^\ast f) = R\TT_{W'} (f)
  = \LL_{W'} (f), \]
using (THOMEO) for $\TT$ and proving this axiom for $\LL$.
\end{proof}
Using the quantization of Section \ref{sec.constft}, the linearization
$\LL$ thus determines a positive TFT with state modules
$Z(M)$, which are Frobenius semialgebras over the semiring
$Q_S (\vect)^c$ and over the semiring $Q_S (\vect)^m$.
If $W$ is a closed $(n+1)$-manifold, then its state sum
$Z_W$ lies in $Z(\varnothing)=Q_S (\vect).$ Thus given any
two vector spaces $V,V'$ and a linear operator $A: V\to V'$,
the state sum yields topologically invariant values $(Z_W)_{VV'} (A)$ in $S$.
If $W$ is not closed, then we may apply the Frobenius counit
to $Z_W \in Z(\partial W)$ to get topological invariants
$(\epsilon_W (Z_W))_{VV'} (A)\in S$.

\section{The Cylinder, Idempotency, and Projections}
\label{sec.cylidemproj}

Let $M$ be a closed $n$-manifold. The cylinder $W=M\times [0,1]$, viewed as a
bordism from $M=M\times \{ 0 \}$ to $M= M\times \{ 1 \},$ plays a special role
in any topological quantum field theory, since its homeomorphism class 
functions as the identity morphism
$M\to M$ in cobordism categories. Thus in such categories, the cylinder is in
particular idempotent and in light of the Gluing Theorem \ref{thm.statesumgluing}
it is reasonable to expect the state sum $Z_{\mti}$ to be idempotent
as well. We shall show below that this can indeed be deduced from our axioms.
It does not, however, follow from these axioms that $Z_{M\times [0,1]}$ must be
a predetermined universal element in $Z(M)\hotimes Z(M)$ which only depends on $M$ and not
on the action. This creates problems for any attempt to recast positive TFTs
as functors on cobordism categories. One could of course add axioms to the
definition of category valued action exponentials $\TT$ that would force
$Z_{\mti}$ to be a ``canonical'' element. Practical experience indicates
that this is undesirable, as it would suppress a range of naturally arising,
interesting actions. We shall write $M_2 = M\sqcup M$.

\begin{prop}  \label{prop.statecylidempot}
The state sum $Z_{\mti}$ of a cylinder is idempotent, that is,
\[ \langle Z_{\mti}, Z_{\mti} \rangle = Z_{\mti} \in Z(M_2). \]
\end{prop}
\begin{proof}
Let $\phi: \mti \to M\times [0,\smlhf]$ be the homeomorphism
$\phi (x,t) = (x,t/2),$ $x\in M,$ $t\in [0,1]$. Then the restriction
$\phi_\partial: M_2 \to M_2$ of $\phi$ to the boundary is the identity
map, $\phi_\partial = \id_{M_2}$. Consequently, $\phi_{\partial \ast}: Z(M_2)\to
Z(M_2)$ is the identity as well. By Theorem \ref{thm.statetopinv},
\[ Z_{\mti} = \phi_{\partial \ast} (Z_{\mti}) = Z_{M\times [0,\smlhf]}. \]
Similarly, $Z_{\mti} = Z_{M\times [\smlhf, 1]}.$
Let $W' = M\times [0,\smlhf],$ $W'' = M\times [\smlhf, 1]$ and
$N=M\times \{ \smlhf \}.$ Then $\mti = W' \cup_N W''$ and we
deduce from the Gluing Theorem \ref{thm.statesumgluing} that
\[  \langle Z_{\mti}, Z_{\mti} \rangle = 
 \langle Z_{W'}, Z_{W''} \rangle =
Z_{\mti}. \]
\end{proof}

Let $M$ and $N$ be two closed $n$-manifolds. Define a map
\[ \pi_{M,N}: Z(M\sqcup N) \longrightarrow Z(M\sqcup N) \]
by $\pi_{M,N} (z) = \langle Z_{\mti}, z \rangle.$
By Proposition \ref{prop.contractbilinear},
$\pi_{M,N}$ is right $Q^c$-linear. The contraction involved here is
\[ \gamma: E(M\times 0)\hotimes E(M\times 1) \hotimes E(M\times 1)
  \hotimes E(N) \longrightarrow E(M\times 0)\hotimes E(N). \]
Thus, technically, $\pi_{M,N}$ is a map
\[ \pi_{M,N}: Z(M\times 1)\hotimes Z(N) \longrightarrow
  Z(M\times 0)\hotimes Z(N) \]
given explicitly by
\[ \pi_{M,N} (z)(f,g) = \sum_{h\in \Fa (M\times 1)}
  Z_{\mti} (f,h)\cdot z(h,g), \]
$f\in \Fa (M\times 0),$ $g\in \Fa (N)$.
\begin{prop} \text{ }\\ \label{prop.pimnprojection}
(1) The state sum $Z_W$ of any bordism $W$ from $M$ to $N$ is
in the image of $\pi_{M,N}$. In fact $Z_W = \pi_{M,N} (Z_W)$, i.e.
$\pi_{M,N}$ acts as the identity on the set of all state sums. \\
(2) The map $\pi_{M,N}$ is a projection, that is, $\pi^2_{M,N} = \pi_{M,N}$.
\end{prop}
\begin{proof}
We prove (1): Let $\widehat{W} = \mti \cup_{M\times \{ 1 \}} W$
be the topological manifold obtained from attaching the cylinder along $M\times \{ 1 \}$
to the incoming boundary $M$ of $W$. By Brown's collar neighborhood theorem \cite{browntopcollars},
the boundary $\partial W$ of the topological manifold $W$ possesses a collar.
Using this collar, there exists a homeomorphism $\phi: \widehat{W} 
\stackrel{\cong}{\longrightarrow} W,$ which is the identity on the boundary,
$\phi_\partial = \id_{M\sqcup N}: \partial \widehat{W} =
M\sqcup N \to M \sqcup N=\partial W$.
By Theorem \ref{thm.statetopinv},
$Z_{\widehat{W}} = \phi_{\partial \ast} (Z_{\widehat{W}}) = Z_W.$
By the Gluing Theorem \ref{thm.statesumgluing},
\[ Z_W = Z_{\widehat{W}} = \langle Z_{\mti}, Z_W \rangle = 
  \pi_{M,N} (Z_W). \]
We prove (2): By Proposition \ref{prop.contractassoc}, the contraction
$\langle -,- \rangle$ is associative. Therefore, using the idempotency of
$Z_{\mti}$ (Proposition \ref{prop.statecylidempot}),
\begin{eqnarray*}
\pi^2_{M,N} (z) & = &
 \langle Z_{\mti}, \langle Z_{\mti}, z \rangle \rangle 
= \langle \langle Z_{\mti}, Z_{\mti} \rangle, z \rangle \\
& = & \langle Z_{\mti}, z \rangle = \pi_{M,N} (z).
\end{eqnarray*}
\end{proof}
Taking $N$ to be the empty manifold, we obtain a projection
\[ \pi_M = \pi_{M,\varnothing}: Z(M) = Z(M\times 1) \longrightarrow Z(M)=Z(M\times 0). \]
(In the case $N=\varnothing$, the map $\gamma$
involved in the inner product defining $\pi_{M,N}$ becomes
\[ \gamma: E(M\times 0)\hotimes E(M\times 1) \hotimes E(M\times 1)
  \longrightarrow E(M\times 0) \]
under the identifications $E(\varnothing)\cong Q,$ 
$E\hotimes Q \cong E$.)
For $z\in Z(M),$ this projection is given by
the explicit formula
\[ \pi_M (z)(f) = \sum_{h\in \Fa (M\times 1)}
  Z_{\mti} (f,h)\cdot z(h), \]
$f\in \Fa (M\times 0)$.
In passing, let us observe the formal analogy to integral transforms 
\[ (Tg)(x) = \int K(x,\xi)g(\xi) d\xi, \]
given by an integral kernel $K$. Thus the state sum $Z_{\mti}$ of the cylinder
can be interpreted as such a kernel. We shall now pursue the question
how to compute the projection $\pi_{M\sqcup N}$ of a tensor product
of state vectors. Is the image again a tensor product?
\begin{lemma} \label{lem.cylmcmequcmc}
Let $\TT$ be a cylindrically firm system of $\catc$-valued action exponentials.
Given fields $f_M \in \Fa (M\times 0),$ $g_M \in \Fa (M\times 1),$
$f_N \in \Fa (N\times 0),$ $g_N \in \Fa (N\times 1)$ and
state vectors $z_M \in Z(M\times 1),$ $z_N \in Z(N\times 1),$ the identity
\[ \big( Z_{\mti} (f_M, g_M)\times Z_{\nti} (f_N, g_N) \big) \cdot
   \big( z_M (g_M) \times z_N (g_N) \big) \hspace{1cm}\]
\[ \hspace{1cm} =\big( Z_{\mti} (f_M, g_M)\cdot z_M (g_M) \big) \times
   \big( Z_{\nti} (f_N, g_N)\cdot z_N (g_N) \big) \]
holds in $Q_S (\catc)$.
\end{lemma}
\begin{proof}
We put $a=Z_{\mti} (f_M, g_M),$ $b=Z_{\nti} (f_N, g_N),$
$c = z_M (g_M)$ and $d = z_N (g_N)$. On a morphism $\xi',$
\[ a(\xi') = \sum_{F_M \in \Fa (\mti, f_M \sqcup g_M)} 
  T_{\mti} (F_M)(\xi'), \]
with
\[ T_{\mti} (F_M)(\xi') = \begin{cases} 1,& \xi' = \TT_{\mti} (F_M) \\
    0,& \text{otherwise.} \end{cases} \]
Similarly, on a morphism $\xi'',$
\[ b(\xi'') = \sum_{F_N \in \Fa (\nti, f_N \sqcup g_N)} 
  T_{\nti} (F_N)(\xi''), \]
with
\[ T_{\nti} (F_N)(\xi'') = \begin{cases} 1,& \xi'' = \TT_{\nti} (F_N) \\
    0,& \text{otherwise.} \end{cases} \]
Thus (see also the proof of Proposition \ref{prop.abcd}),
\[
((a\times b)\cdot (c\times d))_{XY} (\gamma)  =
 \sum_{(\xi', \xi'', \eta', \eta'')\in TCT(\gamma)}
  d(\eta'') c(\eta') b(\xi'') a(\xi') \hspace{3cm} \]
\begin{eqnarray*}
& = & \sum_{F_M, F_N} \sum_{(\xi', \xi'', \eta', \eta'')\in TCT(\gamma)}
  d(\eta'') c(\eta') T_{\nti} (F_N)(\xi'') T_{\mti} (F_M)(\xi') \\
& = & \sum_{F_M, F_N} \sum_{(\eta', \eta'')\in 
    TCT(\gamma; \TT_{\mti} (F_M), \TT_{\nti} (F_N))}
  d(\eta'') c(\eta') \\
& = & \sum_{F_M, F_N} \sum_{(\eta', \eta'')\in 
    CTC(\gamma; \TT_{\mti} (F_M), \TT_{\nti} (F_N))}
  d(\eta'') c(\eta') \\
& = & \sum_{F_M, F_N} \sum_{(\xi', \xi'', \eta', \eta'')\in CTC(\gamma)}
  d(\eta'') T_{\nti} (F_N)(\xi'') c(\eta') T_{\mti} (F_M)(\xi') \\
& = & \sum_{(\xi', \xi'', \eta', \eta'')\in CTC(\gamma)}
  d(\eta'') b(\xi'') c(\eta') a(\xi') \\
& = & ((a\cdot c)\times (b\cdot d))_{XY} (\gamma).
\end{eqnarray*}
Note that the element $T_{\nti} (F_N)(\xi'')$ commutes with any element of $S$,
since it is either $1$ or $0$. 
\end{proof}

\begin{thm} \label{thm.projmdisjn}
Let $M, N$ be closed $n$-manifolds and $z_M \in Z(M),$ $z_N \in Z(N)$.
If $\TT$ is a cylindrically firm system of $\catc$-valued action exponentials, then
\[ \pi_{M\sqcup N} (z_M \hotimes_m z_N) =
  \pi_M (z_M) \hotimes_m \pi_N (z_N). \]
\end{thm}
\begin{proof}
Using Theorem \ref{thm.zdisjunion} to decompose
$Z_{(M\sqcup N)\times [0,1]} = Z_{\mti} \hotimes_m Z_{\nti},$
axiom (FDISJ) to decompose $\Fa ((M\sqcup N)\times 1) \cong
\Fa (M\times 1)\times \Fa (N\times 1)$ and
Lemma \ref{lem.cylmcmequcmc}, we compute the value of
$\pi_{M\sqcup N} (z_M \hotimes_m z_N)$ on a field
$f\in \Fa ((M\sqcup N)\times 0)$:
\[ \pi_{M\sqcup N} (z_M \hotimes_m z_N)(f) =
  \sum_{g\in \Fa ((M\sqcup N)\times 1)} Z_{(M\sqcup N)\times [0,1]} (f,g)
  \cdot (z_M \hotimes_m z_N)(g) \hspace{2cm} \]
\begin{eqnarray*}
& = & \sum_g (Z_{\mti} \hotimes_m Z_{\nti})(f,g)\cdot (z_M \hotimes_m z_N)(g) \\
& = & \sum_g \Big( Z_{\mti} (f|_{M\times 0}, g|_{M\times 1}) \times 
    Z_{\nti} (f|_{N\times 0}, g|_{N\times 1}) \Big) \cdot 
  \Big( z_M (g|_{M\times 1}) \times z_N (g|_{N\times 1}) \Big) \\
& = & \sum_{(g_M,g_N)\in \Fa (M)\times \Fa (N)} 
  \Big( Z_{\mti} (f|_{M}, g_M) \times 
    Z_{\nti} (f|_{N}, g_N) \Big) \cdot 
  \Big( z_M (g_M) \times z_N (g_N) \Big) \\
& = & \sum_{g_M \in \Fa (M)} \sum_{g_N \in \Fa (N)} 
  \Big( Z_{\mti} (f|_{M}, g_M) \cdot 
    z_M (g_M) \Big) \times 
  \Big( Z_{\nti} (f|_N, g_N) \cdot z_N (g_N) \Big) \\
& = & \Big\{ \sum_{g_M \in \Fa (M)}  
   Z_{\mti} (f|_{M}, g_M) \cdot z_M (g_M) \Big\} \times 
  \Big\{ \sum_{g_N \in \Fa (N)} Z_{\nti} (f|_N, g_N) \cdot z_N (g_N) \Big\} \\
& = & \pi_M (z_M)(f|_{M\times 0})\times \pi_N (z_N)(f|_{N\times 0}) \\
& = & (\pi_M (z_M)\hotimes_m \pi_N (z_N))(f).
\end{eqnarray*}
\end{proof}

\begin{cor}
If $\catc$ is a monoid, then
\[ \pi_{M\sqcup N} (z_M \hotimes_m z_N) =
  \pi_M (z_M) \hotimes_m \pi_N (z_N). \]
\end{cor}

\begin{remark}
Following classical topological field theory, and informed by Proposition \ref{prop.pimnprojection},
one might now attempt to set
$Z' (M) = \pi_M Z(M)$.
In the present framework of positive topological field theory, this smaller
state module is not serviceable, for at least the following reason:
In order to obtain an analog of Proposition \ref{prop.zmndisjtensdecomp} for $Z'$,
i.e. a decomposition $Z'(M\sqcup N) \cong Z'(M) \hotimes Z'(N)$,
one would have to rely on results such as Theorem \ref{thm.projmdisjn}, which
breaks up a projection of a tensor product on a disjoint union into a tensor product of projections
on the components. But for many interesting actions that
one wishes to apply the present framework to, the assumption of cylindrical firmness is
not germane. So we refrain from
passing to the images of the projections $\pi_M$.
\end{remark}

\section{Application: P\'olya Counting Theory}
\label{sec.polya}

We will demonstrate that P\'olya's theory \cite{polyacounting}
of counting colored configurations
modulo symmetries can be recreated within the framework of positive topological
field theories by making a judicious choice of monoidal category $\catc$, fields $\Fa$
and action functional $\TT$. We will derive P\'olya's enumeration formula by
interpreting its terms as state sums and applying our theorems on state sums.
P\'olya theory has a wide range of applications, among them chemical isomer
enumeration, investigation of crystal structure, applications in graph theory and
statistical mechanics. The method of this section will also prove Burnside's lemma on
orbit counting from TFT formulae. \\

The set $C=\{ 1,\chi,\mu \}$ becomes a commutative monoid $(C,\cdot, 1)$ by setting
\[ \chi^2 =\mu,~ \chi \cdot \mu =\mu,~ \mu^2 = \mu. \]
(Associativity is readily verified.) By Lemma \ref{lem.commmonoidstrictmoncat},
$(C,\cdot,1)$ determines a small strict monoidal category $\catc = \catc (C)$.
Suppose that a finite group $G$ acts on a finite set $X$ (from the left).
Let $Y$ be another finite set, whose elements we think of as ``colors''.
Then $G$ acts on the set $W = \fun_Y (X)$ of functions $w: X\to Y$ from the right
by $(wg)(x)=w(gx)$. We interpret the elements of $W$ as colorings of $X$.
To form our positive TFT, we also interpret $W$ and its subsets as 
$0=(-1+1)$-dimensional bordisms (whose boundary is necessarily empty).
Given a subset $W' \subset W,$ i.e. a codimension $0$ submanifold, we define
a set $\Ea (W')$ over it by
\[ \Ea (W') = \{ (w,g)\in W'\times G ~|~ g\in G_w \}, \]
where $G_w \subset G$ denotes the isotropy subgroup of $G$ at $w$.
There is an inclusion $\Ea (W')\subset \Ea (W)$. We define the fields on $W'$ to
be
$\Fa (W')= \fun_{\bool} (\Ea (W')),$
where $\bool$ is the Boolean semiring. The restriction $\Fa (W')\to \Fa (W'')$
associated to a subspace $W'' \subset W'$ is given in the natural way, that is,
by restricting a function $F: \Ea (W')\to \bool$ to $\Ea (W'')\subset \Ea (W')$.
Then axiom (FRES) is satisfied. For disjoint subsets $W', W'' \subset W$, we have
$\Ea (W' \sqcup W'') = \Ea (W') \sqcup \Ea (W'')$ and hence
\begin{eqnarray*}
\Fa (W' \sqcup W'') & = & \fun_\bool (\Ea (W' \sqcup W'')) = \fun_\bool (\Ea (W')\sqcup \Ea (W'')) \\
& \cong & \fun_\bool (\Ea (W'))\times \fun_\bool (\Ea (W'')) = \Fa (W')\times \Fa (W'').
\end{eqnarray*}  
Axiom (FDISJ) is thus satisfied. Axiom (FGLUE) is vacuously satisfied, as $W'$ and $W''$
cannot share a nonempty boundary component $N$, and we have already discussed
$N=\varnothing$. (Note that $\otimes = \circ$ in $\catc$.) Since we are working in an
equivariant context, axiom (THOMEO) has to be restricted to equivariant homeomorphisms.
Such a homeomorphism $\phi: W' \to W''$ induces a bijection
$\Ea (\phi): \Ea (W') \to \Ea (W'')$ by $\Ea (\phi)(w,g)=(\phi (w),g),$ since
$\phi (w)\cdot g = \phi (w\cdot g)=\phi (w).$ Obviously
$\Ea (\psi \circ \phi) = \Ea (\psi)\circ \Ea (\phi)$ and $\Ea (\id_{W'}) = \id_{\Ea (W')}$.
Define the action of $\phi$ on fields by
\[ \phi^\ast = \fun_\bool (\Ea (\phi)): \Fa (W'') = \fun_\bool (\Ea (W''))\longrightarrow
  \fun_\bool (\Ea (W')) = \Fa (W'). \]
Then $(\psi \circ \phi)^\ast = \phi^\ast \circ \psi^\ast,$ 
$\id^\ast_{W'} = \id_{\Fa (W')},$ and a commutative square
\[ \xymatrix{
W' \ar[r]^{\phi} & W'' \\
W'_0 \ar@{^{(}->}[u] \ar[r]_{\phi_0} & W''_0 \ar@{^{(}->}[u]
} \]
with $\phi, \phi_0$ equivariant homeomorphisms, induces a commutative square
\[ \xymatrix{
\Fa (W'') \ar[d]_{\operatorname{res}} \ar[r]^{\phi^\ast} & \Fa (W') \ar[d]^{\operatorname{res}} \\
\Fa (W''_0)  \ar[r]_{\phi^\ast_0} & \Fa (W'_0)
} \]
on fields. Thus (FHOMEO) is satisfied and $\Fa$ is a system of fields on subsets of $W$.
The action functional on a field $F\in \Fa (W')$ is by definition
\[ \TT_{W'} (F) = \begin{cases}
 \id_I,& \text{ if } F \text{ is identically } 0, \\
 \chi,& \text{ if } F \text{ is the characteristic function of some element}, \\
 \mu,& \text{ otherwise.}
\end{cases} \]
To verify that $\TT$ is a valid system of action exponentials, let $W', W'' \subset W$ be
disjoint subsets and let $F: \Ea (W'\sqcup W'') = \Ea (W')\sqcup \Ea (W'')\to \bool$ be a field
on their union. If $\TT_{W' \sqcup W''} (F)=\id_I,$ then $F=0$ and so
$F|_{\Ea (W')=0}$ and $F|_{\Ea (W'')}=0.$ Consequently, $\TT_{W'} (F|_{W'})=\id_I$ and
$\TT_{W''} (F|_{W''})=\id_I$. The tensor product of these two morphisms is
$\id_I \otimes \id_I = \id_I = \TT_{W' \sqcup W''}(F).$ If
$\TT_{W' \sqcup W''}(F)=\chi,$ then $F$ is a characteristic function and thus either
$F|_{\Ea (W')} =0$ and $F|_{\Ea (W'')}$ is characteristic, or $F|_{\Ea (W')}$ is characteristic and
$F|_{\Ea (W'')}=0$. Thus precisely one of $\TT_{W'} (F|_{W'}),$ $\TT_{W''} (F|_{W''})$
is $\chi$, the other $\id_I$. Hence their tensor product is $\chi \otimes \id_I = \chi =
\TT_{W' \sqcup W''} (F).$ Finally, if $\TT_{W' \sqcup W''}(F)=\mu,$ then either
$\TT_{W'}(F|_{W'})=\mu$ or $\TT_{W''} (F|_{W''})=\mu$ or else
$\TT_{W'}(F|_{W'})=\chi = \TT_{W''}(F|_{W''}).$
Therefore,
\[ \TT_{W'}(F|_{W'}) \otimes \TT_{W''}(F|_{W''}) \in \{ \chi \otimes \chi,
  \mu \otimes \alpha \} = \{ \mu \}, \]
for $\alpha \in \Mor (\catc),$ which agrees with $\TT_{W' \sqcup W''} (F)=\mu$.
We conclude that (TDISJ) holds. Thus (TGLUE) holds also, since $N=\varnothing$ and
$\circ = \otimes$ in $\catc$. 
Let $\phi: W' \to W''$ be an equivariant homeomorphism and $G: \Ea (W'')\to \bool$ a
field on $W''$. If $G=0$, then $\phi^\ast G = G\circ \Ea (\phi)=0$ and
$\TT_{W'} (\phi^\ast G)=\id_I = \TT_{W''} (G)$. If $G$ is the characteristic function
of some element $(w,g)\in \Ea (W'')$, then $\phi^\ast G$ is the characteristic function
of $(\phi^{-1} (w),g)$ and $\TT_{W'} (\phi^\ast G)=\chi = \TT_{W''} (G)$.
If there are two distinct elements of $\Ea (W'')$ on which $G$ takes the value $1$,
then the $\phi$-preimages of these are two distinct elements on which $\phi^\ast G$
has value $1$, and so $\TT_{W'} (\phi^\ast G)=\mu = \TT_{W''} (G)$.
Thus (THOMEO) is verified and $\TT$ is a system of action exponentials. \\

Let $S = \nat^{\infty}$ be the complete semiring of Example \ref{exple.completesemirings}.
The cardinality of a set $A$ is denoted $|A|$.
Let $Z$ be the positive TFT associated with $S, \catc, \Fa$ and $\TT$.
Let $W/G$ denote the orbit space and let $\Oa \in W/G$ be an orbit.
Then the state sum of $\Oa$ evaluated at $\chi$ is
\begin{eqnarray*}
Z_{\Oa} (\chi) & = & \sum_{F\in \Fa (\Oa)} T_{\Oa} (F)(\chi) = 
   \sum_{F:\Ea (\Oa)\to \bool,~ \TT_{\Oa} (F)=\chi} 1 = |\Ea (\Oa)| \\
& = & \sum_{w\in \Oa} |G_w| = |\Oa|\cdot |G_w| = |G|.
\end{eqnarray*}
The space $W$ can be written as a disjoint union
\[ W = \bigsqcup_{\Oa \in W/G} \Oa. \]
Applying Theorem \ref{thm.zdisjunion} to this decomposition shows that
\[ Z_W = \widehat{\bigotimes}_m \{ Z_{\Oa} ~|~ \Oa \in W/G \}. \]
Evaluating this on $\chi,$ we get
\[ Z_W (\chi) = |W/G|\cdot |G|. \]
On the other hand,
\[ Z_W (\chi) = \sum_{F\in \Fa (W)} T_W (F)(\chi) = \sum_{w\in W} |G_w| =
    \sum_{g\in G} |W^g|, \]
where
$W^g = \{ w\in W ~|~ wg=w \}.$
Thinking of $g$ as a permutation of $X$, $g$ has a unique cycle decomposition.
Let $c(g)$ be the number of cycles. Then $|W^g| = |\fun_Y (X)^g| = |Y|^{c(g)}$
and we arrive at
\[ |\fun_Y (X)/G| = \frac{1}{|G|} \sum_{g\in G} |Y|^{c(g)}, \]
which is P\'olya's enumeration theorem. Applying the above method directly
to $W=X$ instead of $W=\fun_Y (X)$, one obtains Burnside's lemma
\[ |X/G| = \frac{1}{|G|} \sum_{g\in G} |X^g|. \]
These applications suggest that positive TFTs may be instrumental in solving other
types of combinatorial problems as well.

\section{Examples}
\label{sec.examples}

We describe five concrete positive TFTs, ranging from simple
toy examples to nontrivial theories. The latter involve the
Novikov signature of compact oriented manifolds. For these
theories, Novikov additivity is the key ingredient to prove
the gluing axiom for the action-exponentials. In the last example,
\ref{ssec.numbertheory}, we observe that certain arithmetic functions
used in number theory can be expressed as state sums of positive TFTs.

\subsection{The Max TFT}

A commutative monoid $M$ is given by $M= [0,1]$, the unit interval,
with operation $a\cdot b = \max (a,b)$. The unit element is $0$, as
$\max (a,0)=a=\max (0,a)$ for any $a\in [0,1]$. By
Lemma \ref{lem.commmonoidstrictmoncat}, $(M,\cdot, 0)$ determines
a small strict monoidal category $\catm$.
A system $\Fa$ of fields is given by taking $\Fa (W), \Fa (M)$ to be the
set of continuous functions on $W,M,$ respectively, with values in $[0,1]$.
The restrictions are the obvious restrictions of continuous maps to 
subspaces. The action of homeomorphisms on fields is given by composition.
A system $\TT$ of $\catm$-valued action exponentials is given by
\[ \TT_W (f) = \max_{x\in W} f(x),~ f\in \Fa (W), \]
if $W\not= \varnothing$ and $\TT_{\varnothing} (f)=0$. Axiom (TDISJ)
is satisfied, since
\begin{eqnarray*}
\TT_{W\sqcup W'} (f) & = & \max_{x\in W\sqcup W'} f(x) =
  \max (\max_W f(x), \max_{W'} f(x)) \\
& = & (\max_W f(x))\otimes (\max_{W'} f(x)) =
 \TT_W (f|_W) \otimes \TT_{W'} (f|_{W'}). 
\end{eqnarray*}
Axiom (TGLUE) holds, for
\begin{eqnarray*}
\TT_{W} (f) & = & \max_{x\in W} f(x) =
  \max (\max_{W'} f(x), \max_{W''} f(x)) \\
& = & (\max_{W''} f(x))\circ (\max_{W'} f(x)) =
 \TT_{W''} (f|_{W''}) \circ \TT_{W'} (f|_{W'}). 
\end{eqnarray*}
Let $\phi: W\to W'$ be a homeomorphism of bordisms. As 
\[ \TT_W (\phi^\ast f) = \max_{x\in W} f(\phi (x)) =
  \max_{y\in W'} f(y) = \TT_{W'} (f), \]
for any $f\in \Fa (W'),$ axiom (THOMEO) is satisfied as well.
Consequently, there is a positive TFT $Z^{\max}$ associated to $(\Fa, \TT)$.
For a boundary condition $f\in \Fa (\partial W),$ the state sum
$Z^{\max}_W$ yields a value $Z^{\max}_W (f)\in Q_S (\catm)$. Since
$\catm$ has only one object, the unit object $I$, the state sum subject
to the boundary condition is completely described by the values
$Z^{\max}_W (f)_{II} (a) \in S,$
where $S$ is the ground semiring and $a$ is a real numbers, $0 \leq a \leq 1$.
In fact,
\[ Z^{\max}_W (f)_{II} (a) = \sum_{F\in \Fa (W,f),~ \max_W F =a} 1. \]
For $a=0$,
\[ Z^{\max}_W (f)_{II} (0) = \sum_{F\in \Fa (W,f),~ F\equiv 0} 1 =
  \begin{cases} 1,& \text{ if } f\equiv 0, \\ 0,& \text{ else.} \end{cases} \]

\subsection{The Intermediate Value TFT}

Let $M = \{ p_+, p_- \}$ and $N = \{ q_+, q_- \}$ be two $0$-dimensional
manifolds, each consisting of two distinct points. If the ground semiring $S$
has more than one element, i.e. $0\not= 1$, then the following very simple positive TFT
is capable of distinguishing the two $(0+1)$-dimensional
bordisms $W_{=}$ and $W_{\supset \subset}$
\[ 
\xygraph{ !{0;/r1pc/:}
{p_+} [r] !{\xcaph@(0)}  !{\xcaph@(0)} [r] {q_+} [dllll]  {p_-} [r] !{\xcaph@(0)}  !{\xcaph@(0)} 
[r] {q_-}
}, \hspace{2cm}
\xygraph{ !{0;/r1pc/:}
 {p_+} [r] !{\xcaph@(0)} !{\hcap} [rr] !{\hcap-} !{\xcaph@(0)} [r] {q_+}
[dllllll] {p_-} [r] !{\xcaph@(0)} 
 [rr]  !{\xcaph@(0)} [r] {q_-}
}.
\] \\

\noindent Both are bordisms from $M$ to $N$ so that
$\partial W_{=} = M\sqcup N = \partial W_{\supset \subset}.$
Let $(F_2,\cdot,1)$ be the commutative monoid of the field with two elements.
That is, $F_2 = \{ 0,1 \}$ and
 $0\cdot 0 = 1\cdot 0 = 0\cdot 1 =0,$ $1\cdot 1 = 1$.
By Lemma \ref{lem.commmonoidstrictmoncat}, $(F_2,\cdot, 1)$ determines
a small strict monoidal category $\catf_2$.
A system $\Fa$ of fields is given by taking $\Fa (W), \Fa (M)$ to be the
set of continuous functions on $W,M,$ respectively, with values in $[-1,1]$.
The restrictions are the obvious restrictions of continuous maps to 
subspaces. The action of homeomorphisms on fields is given by composition.
A system $\TT$ of $\catf_2$-valued action exponentials is given by
\[ \TT_W (f) = \begin{cases} 1,& \text{ if } 0\not\in f(W) \\
      0, & \text{ if } 0\in f(W). \end{cases} \]
Note that $\TT_{\varnothing} (f)=1 = \id_I$. Let us verify that axiom (TDISJ)
is satisfied. Let $f\in \Fa (W\sqcup W')$ be a field on the disjoint union of
$W$ and $W'$. Suppose that $f$ takes on the value $0$ somewhere on
$W\sqcup W'$. Then either $0\in f(W)$ or $0\in f(W')$ (or both); say
$0\in f(W)$. Hence
\[ \TT_{W\sqcup W'} (f) = 0 = 0\cdot \TT_{W'} (f|_{W'}) =
\TT_W (f|_W) \otimes \TT_{W'} (f|_{W'}). 
\]
If $0\not\in f(W\sqcup W'),$ then $0\not\in f(W)$ and $0\not\in f(W')$.
So in this case
\[ \TT_{W\sqcup W'} (f) = 1 = 1\cdot 1 =
\TT_W (f|_W) \otimes \TT_{W'} (f|_{W'}). 
\]
We conclude that (TDISJ) holds. 
Axiom (TGLUE) is proven in a similar manner:
Let $f\in \Fa (W)$ be a field on $W= W'\cup_N W''$. 
If $f$ takes on the value $0$ somewhere on
$W$, then $0\in f(W')$ or $0\in f(W'')$; say
$0\in f(W')$. Hence
\[ \TT_{W} (f) = 0 = \TT_{W''} (f|_{W''}) \cdot 0 =
\TT_{W''} (f|_{W''}) \circ \TT_{W'} (f|_{W'}). 
\]
On the other hand, if $0\not\in f(W),$ then $0\not\in f(W')$ and $0\not\in f(W'')$.
So in this case
\[ \TT_{W} (f) = 1 = 1\cdot 1 =
\TT_{W''} (f|_{W''}) \circ \TT_{W'} (f|_{W'}). 
\]
This establishes (TGLUE).
Let $\phi: W\to W'$ be a homeomorphism. 
Since $\phi (W)=W'$, we have for any $f\in \Fa (W'),$ 
\[ (\phi^\ast f)(W)=f(\phi (W)) = f(W'). \]
Therefore, $0\in (\phi^\ast f)(W)$ if and only if $0\in f(W').$
This implies that $\TT_W (\phi^\ast f) = \TT_W' (f)$, which verifies
axiom (THOMEO).
Consequently, there is a positive TFT $Z^{\iv}$ associated to $(\Fa, \TT)$,
the \emph{intermediate value TFT}. We shall now show that the
invariant $Z^{\iv}_W \in Z(M\sqcup N)$ distinguishes the two bordisms
$W_{=}$ and $W_{\supset \subset}$ depicted above.
Let $f_0 \in \Fa (M\sqcup N)$ be the boundary condition
\[ f_0 (p_+)=1,~  f_0 (p_-)=-1,~  f_0 (q_+)=1,~  f_0 (q_-)=-1. \]
We calculate the state sum with boundary condition $f_0$ for a
bordism $W$ from $M$ to $N$. As $\catf_2$ has only one object $I$,
all the information is contained in $Z^{\iv}_W (f_0)_{II} (\gamma)$,
where $\gamma:I \to I$ is a morphism, that is, $\gamma \in \{ 0,1 \}$.
For $\gamma =1$ and $F\in \Fa (W,f_0)$, 
\[ T_W (F)_{II} (1) =
\begin{cases} 1,& \text{ if } \TT_W (F)=1 \\
 0, & \text{ if } \TT_W (F)\not=1
\end{cases} \]
Since $\TT_W (F) =1$ if and only if $0\not\in F(W)$, this translates to
\[ T_W (F)_{II} (1) =
\begin{cases} 1,& \text{ if } 0\not\in F(W) \\
 0, & \text{ if } 0\in F(W).
\end{cases} \]
Hence for the state sum,
\[ Z^{\iv}_W (f_0)_{II} (1) =
 \sum_{F\in \Fa (W,f_0)} T_W (F)_{II} (1) =
 \sum_{F\in \Fa (W,f_0),~ 0\not\in F(W)} 1. \]
On the bordism $W_{=},$ there is a continuous function 
$F: W_{=} \to [-1,1]$ such that $F|_{\partial W_{=}} = f_0$ and
$0\not\in F(W_{=})$, e.g. $F\equiv 1$ on the strand that joins $p_+$
and $q_+$ and $F\equiv -1$ on the strand that joins $p_-$ and $q_-$.
Thus the state sum is of the form
\[ Z^{\iv}_{W_{=}} (f_0)_{II} (1) = 1+s \]
for some $s\in S$. But on $W_{\supset \subset}$ every continuous
function $F: W_{\supset \subset} \to [-1,1]$, which agrees with $f_0$
on the boundary, \emph{must} take on the value $0$ somewhere by the
intermediate value theorem. This shows that
\[  Z^{\iv}_{W_{\supset \subset}} (f_0)_{II} (1) = 0. \]
Since $S$ is complete, it is zerosumfree. So if $1+s=0$, then $1=0$,
which contradicts our assumption $1\not= 0$. Hence
\[  Z^{\iv}_{W_{=}} (f_0)_{II} (1) \not=
   Z^{\iv}_{W_{\supset \subset}} (f_0)_{II} (1). \]

\subsection{The Signature TFT}

Suppose that $n+1$ is divisible by $4$, say $n+1 =4k$, and assume that 
bordisms $W$ are oriented.
The \emph{intersection form} of $W$ is the symmetric bilinear form
\[ H^{2k} (W,\partial W;\real)\times H^{2k} (W, \partial W;\real) 
  \longrightarrow \real \]
given by evaluation of the cup product on the fundamental class
$[W]\in H_{4k} (W,\partial W),$ $(x,y)\mapsto \langle x\cup y, [W] \rangle$.
(If the boundary of $W$ is empty, this form is nondegenerate.)
The \emph{signature} of $W$, $\sigma (W)$, is the signature of this
bilinear form. Suppose that $W=W' \cup_N W''$ is obtained by gluing along $N$ the
bordism $W'$ with outgoing boundary $N$ to the bordism $W''$ whose incoming
boundary is also $N$. The orientation of $W$ is to restrict to the orientations of 
$W'$ and $W''$. Then \emph{Novikov additivity} asserts that
\[ \sigma (W) = \sigma (W') + \sigma (W''), \]
see \cite[p. 154]{memgeorgesderham} and \cite{wallnonadd}.
(Novikov and Rohlin were actually interested in defining
Pontrjagin-Hirzebruch classes modulo a prime $p$. The additivity property
for the signature enabled them to find such a definition.) The proof of
Novikov additivity is provided in \cite[Prop. (7.1), p. 588]{ASieoiii}. 
It is important here that $N$ be closed as a manifold. If $W', W''$ are
allowed to have corners and one glues along a manifold with boundary
$(N,\partial N)$, then the signature is generally non-additive but can be calculated
using a formula of Wall, which contains a Maslov triple index correction term.
More recently, Novikov has pointed out that his additivity property is
equivalent to building a nontrivial topological quantum field theory.
We shall now proceed to give a precise construction of such a
\emph{signature TFT} $Z^{\sign}$ using the framework of the present paper. 
Since the signature can be negative, this example shows,
prima facie paradoxically, that invariants which
require additive inverses can also often be expressed by positive TFTs.
To do this, one exploits that the monoidal category $\catc$ can be quite arbitrary.\\ 

By Lemma \ref{lem.commmonoidstrictmoncat}, the additive monoid (group)
$(\intg, +, 0)$ of integers determines a small strict monoidal category $\catz$.
A system $\Fa$ of fields is given by admitting only a single unique field on each manifold,
that is, by taking $\Fa (W) = \{ \star \}$, 
$\Fa (M) = \{ \star \}$, where $\star$ denotes a single element, which we may
interpret as the unique map to a point.
A system $\TT$ of $\catz$-valued action exponentials on oriented bordisms $W$ is given by
\[ \TT_W (\star) = \sigma (W). \]
Note that $\TT_{\varnothing} (\star)=\sigma (\varnothing)=0 = \id_I$. 
In axiom (TDISJ) it is of course now assumed that $W\sqcup W'$ is oriented
in agreement with the orientations of $W$ and $W'$, similarly for axiom (TGLUE). 
Then (TDISJ) is satisfied, since
\[ \TT_{W\sqcup W'} (\star) = \sigma (W\sqcup W') = \sigma (W) + \sigma (W')
  = \TT_W (\star|_W) + \TT_{W'} (\star|_{W'}) =
 \TT_W (\star|_W) \otimes \TT_{W'} (\star|_{W'}). \]
By Novikov additivity, we have for $W= W' \cup_N W'',$
\[ \TT_{W} (\star) = \sigma (W) = \sigma (W'') + \sigma (W')
  = \TT_{W''} (\star|_{W''}) \circ \TT_{W'} (\star|_{W'}), \] 
which proves (TGLUE).
Let $\phi: W\to W'$ be an orientation-preserving homeomorphism.
Then, as the induced map on cohomology
$\phi^\ast: H^\ast (W',\partial W';\real)\to H^\ast (W,\partial W;\real)$
is a ring-isomorphism preserving the cup-product structure, we have
$\sigma (W) = \sigma (W')$. 
Therefore,
\[ \TT_W (\phi^\ast \star) = \sigma (W) = \sigma (W') = \TT_{W'} (\star), \] 
which verifies axiom (THOMEO).
Consequently, there is a positive TFT $Z^{\sign}$ associated to $(\Fa, \TT)$,
the \emph{signature TFT}. If $a$ is any integer (a morphism in $\catz$), then
\[ Z^{\sign}_W (\star)_{II} (a) = T_W (\star)_{II} (a) =
 \begin{cases} 1,& \text{ if } a=\TT_W (\star) = \sigma (W) \\ 
  0,& \text{ otherwise}. \end{cases} \]
So the signature state sum on a morphism is a Kronecker delta function,
\[ Z^{\sign}_W (\star)_{II} (a) = \delta_{a, \sigma (W)}. \]

\subsection{A Twisted Signature TFT}
\label{ssec.twistedsigntft}

The signature TFT can be twisted by allowing nontrivial fields, as we shall now
explain. Let $F$ be a closed, oriented, topological manifold whose dimension is such that
$n+1+\dim F$ is divisible by $4$. Let $G$ be a topological group acting
continuously on $F$ by orientation preserving homeomorphisms.
Let $EG \to BG$ be the universal principal $G$-bundle.
The associated fiber bundle $E\to BG$ with fiber $F$ is given by the total
space $E= EG \times_G F$, i.e. the quotient of $EG \times F$ by the diagonal
action of $G$. The projection is induced by $EG\times F \to EG \to BG$.
Let $W$ be an oriented $(n+1)$-dimensional bordism with boundary
$\partial W$.
Principal $G$-bundles over $W$ have the form of a pullback $f^\ast EG$ for some
continuous map $f: W\to BG$. The associated $F$-bundle  
$p: f^\ast E \to W$ is a compact manifold with boundary $(f|_{\partial W})^\ast E$.
(Note that the bundle $f^\ast E$ is canonically isomorphic as a bundle
to $(f^\ast EG) \times_G F$.)
This manifold receives a canonical orientation as follows: Cover $W$ by open
sets $U_\alpha$ over which there are trivializations $\phi_\alpha: p^{-1} (U_\alpha)
\cong U_\alpha \times F$. The orientation of $W$ induces an orientation of every
$U_\alpha$. Then $U_\alpha \times F$ receives the product orientation and
$p^{-1}(U_\alpha)$ receives the orientation such that $\phi_\alpha$ is
orientation preserving. Taking a trivialization $\phi_\beta$ over $U_\beta$
with nonempty $U_{\alpha \beta} = U_\alpha \cap U_\beta$, the homeomorphism
$\phi_\beta \phi_\alpha^{-1}: U_{\alpha \beta} \times F \to U_{\alpha \beta}\times F$
is orientation preserving as the structure group $G$ acts on $F$ in an orientation
preserving manner. Thus $p^{-1}(U_{\alpha \beta})$ receives the same orientation,
regardless of whether $\phi_\alpha$ or $\phi_\beta$ is used. Since the
$p^{-1}(U_\alpha)$ cover $f^\ast E$, the latter is oriented. 

We shall again use the strict monoidal category $\catz$ of integers, as introduced
in the previous example.
A system $\Fa$ of fields is given by 
\[ \Fa (W) = \{ f:W\to BG ~|~ f \text{ continuous} \}, \]
similarly for closed $n$-manifolds $M$.
(Again, the restrictions are the obvious ones and the action of homeomorphisms 
on fields is given by composition.)
A system $\TT$ of $\catz$-valued action exponentials on oriented bordisms $W$ is given by
the signature of the $F$-bundle pulled back from $BG$ under $f$,
\[ \TT_W (f) = \sigma (f^\ast E). \]
The signature is defined, since $f^\ast E$ is canonically oriented as discussed above.
Note that $\TT_{\varnothing} (f)=\sigma (\varnothing)=0 = \id_I$. 
Axiom (TDISJ) is satisfied, since for a map $f: W\sqcup W' \to BG$,
\begin{eqnarray*}
 \TT_{W\sqcup W'} (f) & = & \sigma (f^\ast E) = 
 \sigma ((f|_W)^\ast E \sqcup (f|_{W'})^\ast E) \\
& = &\sigma ((f|_W)^\ast E) + \sigma ((f|_{W'})^\ast E)
  = \TT_W (f|_W) + \TT_{W'} (f|_{W'}) \\
& = & \TT_W (f|_W) \otimes \TT_{W'} (f|_{W'}). 
\end{eqnarray*}
For $W= W' \cup_N W''$ and $f:W\to BG,$ the manifold $X = f^\ast E$
is a bordism $X= X' \cup_Y X''$ obtained from gluing along
$Y = (f|_N)^\ast E$ the bordism $X' = (f|_{W'})^\ast E$ with outgoing
boundary $Y$ to the bordism $X'' = (f|_{W''})^\ast E$ whose incoming
boundary is also $Y$. We observe that the orientation of $X$ restricts
to the orientations of $X'$ and $X''$. 
Thus by Novikov additivity, we have 
\[ \TT_{W} (f) = \sigma (X) = \sigma (X'') + \sigma (X')
  = \TT_{W''} (f|_{W''}) \circ \TT_{W'} (f|_{W'}), \] 
which proves (TGLUE).
Let $\phi: W\to W'$ be an orientation preserving homeomorphism
and let $f\in \Fa (W')$ be a field on $W'$.
Pulling $f^\ast E$ back to $W$, we obtain a cartesian square
\[ \xymatrix{
\phi^\ast f^\ast E \ar[r]^{\Phi} \ar[d] & f^\ast E \ar[d] \\
W \ar[r]^{\cong}_{\phi} & W'.
} \]
Then $\Phi$ is a homeomorphism and preserves orientations, since
$\phi$ does so. Consequently, $\sigma (\phi^\ast f^\ast E) = 
\sigma (f^\ast E)$ and
\[ \TT_W (\phi^\ast f) = \sigma ((f\phi)^\ast E) = \sigma (f^\ast E) = \TT_{W'} (f), \] 
which verifies axiom (THOMEO).
Therefore, there is a positive TFT $\widetilde{Z}^{\sign}$ associated to $(\Fa, \TT)$,
the \emph{$F$-twisted signature TFT}. 
If $f^\ast E \to \partial W$ is an $F$-bundle over the boundary, given by a map
$f:\partial W\to BG$, and
$a$ any integer (a morphism in $\catz$), then the $F$-twisted signature state sum
$\widetilde{Z}^{\sign} (f)$ is the counting function whose value on $a$ is 
\[ \widetilde{Z}^{\sign}_W (f)_{II} (a) = \sum_A 1, \]
where $A$ is the set of all $F:W\to BG$ extending $f$ such that 
$\sigma (F^\ast E)=a$.
So roughly, the signature state sum on a morphism $a$ ``counts'' those $F$-bundles
over $W$ that extend $f^\ast E\to \partial W$ and have signature $a$.
Naturally, since one is not summing over distinct isomorphism types of $F$-bundles,
that is, distinct homotopy classes of maps $W\to BG$, one generally picks up either
no summand or uncountably many. So, as pointed out before, what may primarily be
of interest is the zero/nonzero-pattern contained in the invariant 
$\widetilde{Z}^{\sign}_W$, not the actual value in $S$. 
The reason why this TFT is called \emph{twisted} is that, contrary to the Euler
characteristic, the signature is generally not multiplicative for fiber bundles:
If $F\to E\to B$ is any fiber bundle of closed, oriented manifolds, then in
general $\sigma (E)\not= \sigma (B)\sigma (F)$, see
\cite{atiyah}, \cite{meyer}, \cite{meyer2}. Instead, at least in the smooth context, $\sigma (E)$
is a product of a Chern-character associated to the $K$-theory signature of
the flat middle-degree cohomology bundle and the Hirzebruch $L$-class of the base.
Thus the present action exponential $\TT$ is a highly nontrivial invariant and
generally hard to compute.

\subsection{Relation to Number Theoretic Quantities}
\label{ssec.numbertheory}

Certain number theoretic quantities, such as arithmetic functions, can be rendered as
state sums of a positive TFT. The multiplicativity of such functions on coprime integers
can then be deduced from the multiplicativity of state sums on disjoint manifolds.
Our method of implementing arithmetic functions rests on two observations:
First, natural numbers can be encoded as manifolds in such a way that disjointness
of manifolds corresponds to coprimality of numbers, intersection corresponds to
taking the greatest common divisor and disjoint union corresponds to taking products
of coprime numbers. Second, the set of divisors of a natural number displays
exactly the same characteristics as fields on manifolds according to
Definition \ref{def.fields}.\\

Let $\Na$ be the set of all finite multisets $W$ of prime numbers and let
$\nat_{>0} = \{ 1,2,3,4,\ldots \}.$ A map $\Na \to \nat_{>0}$ is given by
sending $W$ to
\[ n(W) = \prod_{p\in W} p, \]
where $p$ occurs as many times in the product as its multiplicity in $W$ indicates.
We have $n(\varnothing)=1.$ In the other direction, a map $\nat_{>0} \to \Na$
is given by factoring an $a\geq 2$ into primes
\[ a = p_1^{\mu_1} \cdots p_k^{\mu_k},~ \mu_i \geq 1, \]
and defining $W(a)$ to be the multiset containing each $p_i$ with multiplicity $\mu_i$
for all $i=1, \ldots, k$. Furthermore, $W(1)=\varnothing$. Then
$n(W(a))=a$ for all $a\in \nat_{>0}$ and $W(n(W'))=W'$ for all $W' \in \Na$.
Thus the above maps define a bijection $\Na \cong \nat_{>0}$. Recall that the
intersection $A\cap B$ of two multisets $A,B$ assigns the multiplicity
$\min (\mu_A (x), \mu_B (x))$ to an element $x$, where $\mu_A (x), \mu_B (x)$
are the multiplicities of $x$ in $A,B$, respectively. (If $x\not\in A$, then
$\mu_A (x)=0$.) The formula
$W(a)\cap W(b) = W(\gcd (a,b))$
holds. In particular, $W(a)$ and $W(b)$ are disjoint if and only if $a$ and $b$ are
coprime. In that case,
\[ W(a)\sqcup W(b) = W(ab). \]
For $W\in \Na$, let $\Fa (W)$ be the set of all multisubsets $W' \subset W$.
The image of $\Fa (W(a))$ under $\Na \to \nat_{>0}$ is the set
of all divisors of $a$. An inclusion $W_0 \subset W$ induces a restriction map
\[ 
\Fa (W) \longrightarrow \Fa (W_0),~
W' \mapsto W' \cap W_0
\]
If $W_1 \subset W_0,$ then the composition $\Fa (W)\to \Fa (W_0)\to \Fa (W_1)$
agrees with $\Fa (W)\to \Fa (W_1)$, since
$(W' \cap W_0)\cap W_1 = W' \cap W_1$. Thus axiom (FRES) is satisfied.
Given disjoint multisets $W,W' \in \nat$, the map
\[ \begin{array}{rcl}
\Fa (W)\times \Fa (W') & \longrightarrow & \Fa (W\sqcup W')\\
(V,V') & \mapsto & V\sqcup V'
\end{array} \]
is inverse to the product map of restrictions
\[ \begin{array}{rcl}
\Fa (W\sqcup W') & \longrightarrow & \Fa (W)\times \Fa (W')\\
W_0 & \mapsto & (W_0 \cap W, W_0 \cap W').
\end{array} \]
Thus the latter map is a bijection and axiom (FDISJ) holds.
(If $W$ and $W'$ are not disjoint, then the above maps cease to be each others'
inverses. For example, when $W=\{ 2,3,5 \},$ $W' = \{ 2,3,7 \}$, $V=\{ 2,5 \}$
and $V' = \{ 3,7 \},$ then $((V\sqcup V')\cap W, (V\sqcup V')\cap W') =
(\{ 2,3,5 \}, \{ 2,3,7 \}) \not= (V,V').$)
Recall that a bijection of multisets is a bijection of the underlying sets which preserves
multiplicities. A homeomorphism (i.e. bijection) $\phi: W\to W'$ induces
$\phi^\ast: \Fa (W')\to \Fa (W)$ by
\[ \phi^\ast (W'') = \{ (\phi^{-1} (p),\mu) ~|~ (p,\mu)\in W'' \} \]
for $W'' \subset W'$, where we have used the notation $(\text{element},\text{multiplicity})$.
Then indeed $\phi^\ast (W'')\subset W$ as multisets, since
$\mu_{W''}(p)\leq \mu_{W'}(p)$ and $\mu_W (\phi^{-1}(p))=\mu_{W'}(p)$
imply
\[ \mu_{\phi^\ast W''}(\phi^{-1}(p)) = \mu_{W''}(p)\leq
  \mu_{W'} (p) = \mu_W (\phi^{-1}(p)). \]
Evidently, $(\psi \circ \phi)^\ast = \phi^\ast \circ \psi^\ast$ and
$\id^\ast_W = \id_{\Fa (W)}$. It follows that $\phi^\ast: \Fa (W')\to \Fa (W)$
is a bijection. A commutative square
\[ \xymatrix{
W \ar[r]^{\phi} & W' \\
W_0 \ar@{^{(}->}[u] \ar[r]_{\phi_0} & W'_0 \ar@{^{(}->}[u]
} \]
with $\phi, \phi_0$ homeomorphisms, induces a commutative square
\[ \xymatrix{
\Fa (W') \ar[d]_{\operatorname{res}} \ar[r]^{\phi^\ast} & \Fa (W) \ar[d]^{\operatorname{res}} \\
\Fa (W'_0)  \ar[r]_{\phi^\ast_0} & \Fa (W_0)
} \]
on fields, since for $q\in W_0$,
\begin{eqnarray*}
\mu_{\phi^\ast_0 \operatorname{res} (W'')}(q) & = &
  \mu_{\operatorname{res}(W'')}(\phi_0 (q)) = \mu_{W'' \cap W'_0}(\phi_0 (q)) \\
& = & \min (\mu_{W''} (\phi (q)), \mu_{W'_0} (\phi_0 (q)) =
   \min (\mu_{W''} (\phi (q)), \mu_{W_0} (q)) \\
& = & \min (\mu_{\phi^\ast (W'')}(q), \mu_{W_0} (q)) = 
  \mu_{\phi^\ast (W'')\cap W_0} (q) \\
& = & \mu_{\operatorname{res} \phi^\ast (W'')}(q).
\end{eqnarray*}
Thus axiom (FHOMEO) is satisfied and we conclude that the system of divisor sets constitutes
a valid system $\Fa$ of fields. \\

This setup opens the door to studying arithmetic functions as state sums of positive
TFTs by considering appropriate monoidal categories $\catc$ and actions
$\TT_W: \Fa (W)\to \Mor (\catc)$. The simplest example is the divisor function
\[ d(a) = \sum_{n|a} 1. \]
It is multiplicative, that is, $d(ab) = d(a)d(b)$ for coprime $a,b$.
From the viewpoint of state sums, this follows from
\[ Z_{W(ab)} = Z_{W(a)\sqcup W(b)} = Z_{W(a)} \hotimes_m Z_{W(b)} \]
by taking $S=\nat^{\infty}$,  the complete semiring of Example \ref{exple.completesemirings},
$\catc$ the trivial monoidal category with $\operatorname{Ob}(\catc)=\{ I \}$,
$\Mor (\catc)=\{ \id_I \},$ and perforce
$\TT_W (W') = \id_I$ for all $W' \in \Fa (W)$.

\bibliographystyle{amsalpha}
\bibliography{../../../mybib}

\end{document}